\documentclass[a4paper, 10pt]{amsart}
\usepackage{amsthm}
\usepackage[]{amsmath}
\usepackage{amssymb}
\usepackage{enumerate}
\usepackage{tabularx}
\usepackage[]{color}
\usepackage[left=3cm, right=3cm]{geometry}
\usepackage[colorlinks]{hyperref}
\usepackage{tikz}
\usepackage{multirow}
\usepackage{subcaption}
\usepackage{algorithm}
\usepackage{algorithmic}

\newcommand{\bm}[1]{\boldsymbol{#1}}
\newcommand{\bmr}[1]{\bm{\mr{#1}}}
\newcommand{\lj}{[ \hspace{-2pt} [}
\newcommand{\rj}{] \hspace{-2pt} ]}
\newcommand{\mb}[1]{\mathbb{#1}}
\newcommand{\mc}[1]{\mathcal{#1}}
\newcommand{\mr}[1]{\mathrm{#1}}
\newcommand{\jump}[1]{\lj #1 \rj}
\newcommand{\aver}[1]{ \{#1\}  }
\newcommand{\wt}[1]{ \widetilde{ #1}}
\newcommand{\tr}[1]{\ifmmode \mathrm{tr}\left( #1 \right) \else 
\text{tr} \left( #1 \right) \fi }
\newcommand{\vect}[2]{ \begin{bmatrix} #1 \\ #2 \\ \end{bmatrix}}
\newcommand{\vech}[3]{ \begin{bmatrix} #1 \\ #2 \\ #3 \\ \end{bmatrix}}
\newcommand\curl{\ifmmode \mathrm{curl} \else \text{curl}\fi}
\renewcommand\det{\ifmmode \mathrm{det} \else \text{det} \fi}
\newcommand\cof{\ifmmode \mathrm{cof} \else \text{cof} \fi}

\def\MA{Monge-Amp\`ere }

\newcommand\MTh{\mc{T}_h}
\newcommand\MEh{\mc{E}_h}
\newcommand\un{\bm{\mr n}}
\renewcommand{\d}[1]{\mathrm d \boldsymbol{#1}}
\newcommand\pnorm[1]{\| #1 \|_{\bm{\mathrm{p}}}}
\newcommand\unorm[1]{\| #1 \|_{\bm{\mathrm{u}}}}
\newcommand\enorm[1]{|\!|\!| #1 |\!|\!|}

\newtheorem{assumption}{Assumption}
\newtheorem{theorem}{Theorem}
\newtheorem{lemma}{Lemma}

\newtheorem{remark}{Remark}

\title[Reconstructed Discontinuous Approximation for Monge-Amp\`ere
Equation]{A Reconstructed Discontinuous Approximation to Monge-Amp\`ere
Equation in Least Square Formation}

\author[R. Li]{Ruo Li} \address{CAPT, LMAM and School of Mathematical
  Sciences, Peking University, Beijing 100871, P.R. China}
\email{rli@math.pku.edu.cn}

\author[F.-Y. Yang]{Fanyi Yang} \address{School of Mathematical
  Sciences, Peking University, Beijing 100871, P.R. China}
\email{yangfanyi@pku.edu.cn}

\begin{document}

\begin{abstract} 
  We propose a numerical method to solve the \MA equation which admits
  a classical convex solution. The \MA equation is reformulated into
  an equivalent first-order system. We adopt a novel reconstructed
  discontinuous approximation space which consists of piecewise
  irrotational polynomials. This space allows us to solve the first
  order system in two sequential steps. In the first step, we solve a
  nonlinear system to obtain the approximation to the gradient. A
  Newton iteration is adopted to handle the nonlinearity in the
  system. The approximation to the primitive variable is obtained from
  the approximate gradient by a trivial least squares finite element
  method in the second step. Numerical examples in both two dimensions
  and three dimensions are presented to show an optimal convergence
  rate in accuracy. It is interesting to observe that the
  approximate solution is piecewise convex in each element.
  Particularly, with the reconstructed approximation space, the
  proposed method demonstrates a remarkable robustness. The
  convergence of the Newton iteration does not rely on the initial
  values, which shows a very different behaviour from references. The
  dependence of the convergence on the penalty parameter in the
  discretization is also negligible, in comparison to the classical
  discontinuous approximation space.

  \noindent \textbf{keywords}: \MA equation, Least squares method,
  Reconstructed discontinuous approximation. 
\end{abstract}


\maketitle

\section{Introduction}
\label{sec:introduction}
The elliptic \MA equation is a fully nonlinear second-order partial
differential equation, which arises naturally from geometric surface
theory and from the applications such as optimal mass transportation,
kinetic theory, geometric optics, image processing and others, and we
refer to \cite{Caffarelli1995fully, Caffarelli2004recent,
  Gutierrez2001progress} and the references therein for an extensive
review of applications. Recently, the numerical scheme for solving the
elliptic \MA equation has been a subject of particular interests
\cite{Benamou2010numerical}. It is known that the classical solution
to the \MA equation is strictly convex on the domain with the positive
source term. Hence, the \MA equation is challenging to solve
numerically due to its full nonlinearity and the convex constraint.
We refer to the review papers \cite{Feng2013recent,
  Neilan2017numerical} for an overview of the numerical challenges and
the history of the work on this problem.

In 1988, Prussner and Oliker introduced a finite difference scheme in
\cite{Oliker1988numerical} for the \MA equation. The discretization
was based on the geometric interpretation of the equation and they
proved that the method converges to the generalized solution in two
dimensions. Froese and Oberman proposed a convergent monotone finite
difference scheme by constructing a wide stencil. We refer to
\cite{Froese2011covergent, Oberman2008wide, Benamou2016monotone,
Chen2018monotone} for more discussion and some improvements on the
wide stencil scheme. Another simple finite difference method was
proposed in \cite{Benamou2010numerical} but the proof of convergence
remains an open problem. Galerkin-type methods have also been
investigated for the \MA equation and an immediate challenge is the
problem does not naturally fit within the Galerkin framework
\cite{Chen2018monotone}. B\"ohmer introduced an $L^2$ projection
method in \cite{Bohmer2008finite} by applying the $C^1$ finite element
spaces. Brenner et al. \cite{Brenner2011fully} proposed a $C^0$ finite
element method. They proposed a discrete linearization which is
consistent with continuous linearization. Dean and Glowinski
\cite{Dean2006augmented, Dean2006numerical} reformulated the \MA
equation as a minimization problem by applying the augmented
Lagrangian method. The minimization problem can then be solved with
mixed finite element methods. Feng and Neilan added a small multiple
of the biharmonic operator to the \MA equation. The resulted
fourth-order PDE is solved by mixed finite element methods
\cite{Feng2009vanishing, Feng2009mixed}. Besides, there are some least
squares finite element methods proposed for the \MA equation and we
refer to \cite{Westphal2019newton, Caboussat2018least,
Caboussat2013least}.

In this paper, we propose a new least squares finite element method
for solving the \MA equation with classical solutions. As a
preparation, we reformulate the \MA equation into an equivalent
first-order system and we solve the first-order problem in two
sequential steps \cite{li2019sequential, li2019nondivergence}. In the
first step, we solve a nonlinear first-order system to obtain the
approximation to the gradient by a piecewise irrotational polynomial
space. This space is obtained by patch reconstruction with only one
unknown per element \cite{li2016discontinuous}. The second step is to
solve a linear first-order system to seek a numerical approximation to
the primitive variable.

The numerical scheme to the first nonlinear problem is the main
component in our method. We first employ a standard Newton-type
linearization to the nonlinear problem at the continuous level. In
each nonlinear step, we will solve an elliptic problem in
non-divergence form. In the discrete level, we define a least
squares functional for the non-divergence form problem and we minimize
this functional on the reconstructed space to seek a numerical
solution at each iteration, and we then update the numerical solution
for the next step via Newton method. In the second step, we introduce
another least squares functional to solve the linear problem. This
functional is then minimized in the Lagrange finite element space,
together with the numerical gradient from the first step, to seek a
numerical solution to the primitive variable. For this linear
first-order system, we present the error estimate which is verified by
the numerical tests.

By numerical examples in two and three dimensions, it is clear that
the numerical solutions achieve the optimal convergence order in
accuracy. It is very interesting for us to observe that in numerical
tests the numerical solution in the reconstructed space is
automatically piecewise convex, which meets the convex constraint of
the classical solution to the \MA equation. Particularly, the residual
in each Newton iteration has a certain kind of mysterious relation to
the piecewise convexity of the numerical solution. At the first
nonlinear iterations, the residual has a slow decreasing since there
are a lot of elements where the solution is non-convex. Once the
solution on most of elements is convex, the Newton iteration rushes to
meet the stop criterion in only a few steps.

The numerical tests demonstrate a remarkable robustness of the method,
which achieves a very rapid convergence of the nonlinear iteration and
it is insensitive to the initial value. Notice that the methods in the
references often require an initial value quite close to the exact
solution \cite{Brenner2011fully, Froese2011covergent,
Feng2009vanishing} and providing such an initial guess can be as
difficult as solving the nonlinear system itself
\cite{Feng2013recent}. The insensitivity to the initial value of our
method may be an attractive feature.

We make a comparison by replacing the reconstructed discontinuous
space with an approximation space that is used in standard
discontinuous Galerkin method \cite{li2019nondivergence}. It is a
surprise for us that the robustness we enjoyed is gone. Using the
space in standard discontinuous Galerkin method, the convergence of
the Newton iteration begins to rely on if the initial value is close
to the exact solution. Even if we already have a ``good'' initial
value, we can only get back a converged Newton iteration by tuning the
penalty parameters in the discretization. Such behaviours indicate us
that the robustness of our method may be largely due to the use of the
reconstructed discontinuous space. Underlying reason sounds an
interesting problem which requires our future study.

The rest of this paper is organized as follows. In Section
\ref{sec:preliminaries}, we define the notations that will be used
throughout the paper. In Section \ref{sec:reconstructedspace}, we
introduce the reconstructed approximation space and give some basic
properties of this space. In Section \ref{sec:scheme}, we present the
details on the numerical approach for the first nonlinear system and
the second linear system. In Section \ref{sec:numericalresult}, we
provide a series of numerical examples in both two dimensions and
three dimensions to show the convergence results and illustrate the
great robustness of our method. We also present two numerical
evidences in Section \ref{sec:comparison} to show the compelling
features of the reconstructed space, comparing the results that are
obtained with the standard discontinuous piecewise polynomial space. A
short conclusion remark ends the main part of this paper. The details
on the construction of the reconstructed space are presented in the
appendix.


\section{Preliminaries}
\label{sec:preliminaries}
Let $\Omega$ be a convex polygonal (polyhedral) domain in $\mb{R}^d (d
= 2, 3)$ with the boundary $\partial \Omega$. We denote by $\MTh$ a
regular and shape-regular triangular (tetrahedral) partition of the
domain $\Omega$ into triangles (tetrahedrons). We denote by $\MEh^i$
the set of all $d - 1$ dimensional interior faces in $\MTh$ and by
$\MEh^b$ the set of all $d - 1$ dimensional faces lying on the
boundary $\partial \Omega$. We then define $\MEh = \MEh^i \cup \MEh^b$
as the set consisting of all faces in partition. Further, we set the
diameter of the element $K \in \MTh$ as $h_K$ and set the size of the
face $e \in \MEh$ as $h_e$. We denote by $h = \max_{K \in \MTh} h_K$
the mesh size of the partition $\MTh$ and the shape-regularity of the
partition $\MTh$ reads: for each element $K \in \MTh$, its diameter
$h_K$ is bounded with a constant $\sigma$, 
\begin{displaymath}
  h_K \leq \sigma \rho_K,
\end{displaymath}
where $\rho_K$ denotes the radius of the largest disk (ball) inscribed
in $K$

Then we introduce the notations associated with weak formulations. Let
$K^+$ and $K^-$ be two adjacent elements that share a common face $e =
K^+ \cap K^-$. Let $\un^+$ and $\un^-$ be the outward unit normal on
$\partial K^+$ and $\partial K^-$, respectively. For the scalar-valued
function $v$ and that vector-valued function $\bm{q}$ that may be
discontinuous across interelement boundaries, we let  $v^+ := v|_{e
\subset \partial K^+}$, $v^- := v|_{e \subset \partial K^-}$,
$\bm{q}^+ := \bm{q}|_{e \subset \partial K^+}$, $\bm{q}^- :=
\bm{q}|_{e \subset \partial K^-}$ and further we define the average
operator $\aver{ \cdot}$ and the jump operator $\jump{ \cdot}$ as
\begin{displaymath}
  \aver{v} := \frac{1}{2} \left( v^+ + v^- \right), \quad
  \aver{\bm{q}} := \frac{1}{2} \left( \bm{q}^+ + \bm{q}^- \right),
  \quad \text{on } e, 
\end{displaymath}
and
\begin{displaymath}
  \begin{aligned}
    & \jump{v} := v^+\un^+ + v^- \un^-, \quad \jump{ \bm{q} \otimes \un
    } := \bm{q}^+ \otimes \un^+ + \bm{q}^- \otimes \un^-, \quad
    \text{on } e, \\
    & \jump{\bm{q} \times \un} := \bm{q}^+ \times \un^+ + \bm{q}^-
    \times \un^-, \quad \text{on } e, \\
  \end{aligned}
\end{displaymath}
where $\otimes$ denotes the tensor product between two vectors. For $e
\in \MEh^b$, the definitions for both trace operators are modified as 
\begin{displaymath}
  \aver{v} := v|_e, \quad \aver{\bm{q}} := \bm{q} |_e, \quad \jump{v}
  := v|_e \un, \quad \jump{\bm{q} \times \un} := \bm{q}|_e \times \un,
  \quad \jump{\bm{q} \otimes \un} := \bm{q}|_e \otimes \un, \quad
  \text{on } e, 
\end{displaymath}
where $\un$ denotes the outward unit normal.

Let us note that throughout this paper the capital $C$ or $C$ with a
subscript are generic constants which may vary from line to line but
are independent of the mesh size $h$. In addition, we will follow the
standard notations and definitions for these Sobolev spaces $L^2(D)$,
$L^2(D)^d$, $L^2(D)^{d \times d}$, $H^r(D)$, $H^r(D)^d$, $H^r(D)^{d
\times d}$ with a bounded domain $D \subset \mb{R}^d$ and a
non-negative integer $r$.  We would also use their associated inner
products and norms. We further define the Sobolev space of
irrotational vector fields by 
\begin{displaymath}
  \bmr{I}^r(D) := \left\{ \bm{v} \in H^r(D)^d \ | \  \nabla \times
  \bm{v} = 0 \text{ in } D \right\}.
\end{displaymath}

For the partition $\MTh$, we would follow the standard definitions for
broken Sobolev spaces $L^2(\MTh)$, $L^2(\MTh)^d$, $L^2(\MTh)^{d \times
d}$, $H^r(\MTh)$, $H^r(\MTh)^d$, $H^r(\MTh)^{d \times d}$ and their
corresponding broken norms and semi-norms \cite{arnold2002unified}. 

For the bounded domain $D \subset \mb{R}^d$ and for an integer $k \geq
0$, we let $\bmr{S}(D)^k$ denote the space of irrotational polynomials
of degree less than $k$, 
\begin{displaymath}
  \bmr{S}(D)^k := \left\{ \bm{v} \in \mb{P}_k(D)^d \ |\  \nabla \times
  \bm{v} = 0 \text{  in } D \right\}, 
\end{displaymath}
where $\mb{P}_k(\cdot)$ denotes the polynomial space of degree less
than $k$.  Then we give the basic approximation properties of the
space $\bmr{S}(D)^k$. 
\begin{lemma}
  For any $\bm{q} \in \bmr{I}^{k + 1}(K)$ and an element $K$, there
  exists a polynomial $\wt{\bm{q}} \in \bmr{S}(K)^k$ such that 
  \begin{displaymath}
    \|\bm{q} - \wt{\bm{q}}\|_{H^t(K)} \leq C h_K^{k + 1 - t} \| \bm{q}
    \|_{H^{t+1}(K)}, \quad 0 \leq t \leq k + 1. \\
  \end{displaymath}
  \label{le:interperror}
\end{lemma}
\begin{proof}
  The proof of Lemma \ref{le:interperror} could be found in
  \cite{li2019sequential}.
\end{proof}

Next, we define a local $L^2$-projection $\pi_{K}^{\bmr{S}, k}$ for
any function $\bm{q} \in \bmr{I}^{k+1}(K)$ such that
$\pi_{K}^{\bmr{S}, k} \bm{q} \in \bmr{S}^k (K)$ satisfying
\begin{displaymath}
  \|\bm{q} - \pi_{K}^{\bmr{S}, k} \bm{q} \|_{L^2(K)} = \min_{\bm{r}
  \in \bmr{S}^k(K)} \|\bm{q} - \bm{r} \|_{L^2(K)}. 
\end{displaymath}
We state the following approximation property of the $L^2$-projection.
\begin{lemma}
  For element $K \in \MTh$, the following estimates hold:
  \begin{equation}
    \begin{aligned}
      \|\bm{q} - \pi_{K}^{\bmr{S}, k} \bm{q} \|_{H^t(K)} & \leq C
      h_K^{k+1 - t} \|\bm{q} \|_{H^{k+1}(K)}, \quad 0 \leq t \leq k
      + 1, \\
      \| \partial^k ( \bm{q} -  \pi_{K}^{\bmr{S}, k} \bm{q}  )
      \|_{L^2(\partial K)} & \leq C h_K^{k + 1/2 - t} \| \bm{q}
      \|_{H^{k+1}(K)}, \quad  0 \leq t \leq k, \\
    \end{aligned}
    \label{eq:L2interpolation}
  \end{equation}
  for any $\bm{q} \in \bmr{I}^{k+1}(K)$.
  \label{le:L2interpolation}
\end{lemma}
\begin{proof}
  The proof of Lemma \ref{le:L2interpolation} could be found in
  \cite{li2019nondivergence}.
\end{proof}

For a given partition $\MTh$, we could define the piecewise
irrotational polynomial space $\bmr{S}_h^k$ as 
\begin{displaymath}
  \bmr{S}_h^k := \left\{\bm{v} \in L^2(\Omega)^d \ | \  \bm{v}|_K \in
  \bmr{S}^k(K), \  \forall K \in \MTh \right\},
\end{displaymath}
or equally the space $\bmr{S}_h^k$ can be compactly written as
$\bmr{S}_h^k = \Pi_{K \in \MTh} \bmr{S}^k(K)$. Ultimately, we
introduce a method for constructing the bases of the irrotational
polynomial space $\bmr{S}^k(\cdot)$ in Appendix \ref{sec:appendix}.

\section{Reconstructed Approximation Space}
\label{sec:reconstructedspace}
In this section, for the given partition $\MTh$, we define a
reconstruction operator from the piecewise constant space
$\bmr{U}_h^0$ into the piecewise irrotational polynomial space
$\bmr{S}_h^k$. Precisely, $\bmr{U}_h^0$ is given by
\begin{displaymath}
  \bmr{U}_h^0 := \left\{ \bm{v} \in L^2(\Omega)^d \ |\ \bm{v}|_K \in
  \mb{P}_0(K)^d, \ \forall K \in \MTh \right\}.
\end{displaymath}
The reconstruction procedure mainly includes two parts. First, for
every element $K \in \MTh$, a point located inside the element $K$ is
specified as its corresponding collocation point $\bm{x}_K$. The
choice of $\bm{x}_K$ is flexible and we can particularly assign
$\bm{x}_K$ as the barycenter of the element $K$. The next step is to
aggregate an element patch $S(K)$ for each element $K$ in partition.
The element patch $S(K)$ is a collection of elements and consists of
the element $K$ and some surrounding elements. Here $S(K)$ is
constructed with a recursive strategy. For element $K$, we start the
recursion from appointing a threshold $\# S(K)$ which is used to
control the size of the set $S(K)$, and $S(K)$ is constructed by
agglomerating the neighbours and recursively going from there.
Specifically speaking, we construct a sequence of element sets
$S_0(K)$, $S_1(K)$, $S_2(K)$, $\cdots$, where $S_0(K)$ is just
$\left\{ K \right\}$ and $S_t(K) (t \geq 1)$ is defined in a recursive
manner: 
\begin{displaymath}
  S_t(K) = \left\{ \widetilde{K} \in \MTh \ |\ \text{there exists
  $\widehat{K} \in S_{t - 1}(K)$ such that } \widetilde{K} \cap
  \widehat{K} = e \in \MEh \right\}, \quad t = 1, 2, \cdots
\end{displaymath}
The recursive procedure ends if the depth $t$ satisfies that the set
$S_t(K)$ has collected as least $\# S(K)$ elements. After that, we
sort the distances between the collocation points of elements
belonging to $S_t(K)$ and the collocation point $\bm{x}_{K}$. We select
the $\# S(K)$ smallest values and gather the corresponding elements
to form the element patch $S(K)$. The cardinality of $S(K)$ directly
equals to $\# S(K)$. In numerical experiments, we use the same $\#
S(K)$ for all elements and here we present the steps in Algorithm
\ref{alg:patch} to show the details of construction of element
patches. 
\begin{algorithm}[htb]
  \caption{Constructing Element Patch}
  \label{alg:patch}
  \begin{algorithmic}[1]
    \renewcommand{\algorithmicrequire}{\textbf{Input:}}
    \REQUIRE
    partition $\MTh$ and a uniform threshold $\# S(K)$; \\
    \renewcommand{\algorithmicrequire}{\textbf{Output:}}
    \REQUIRE
    the element patches of all elements; \\
    \FOR{every $K \in \MTh$}
    \STATE{initialize $t=0$, $S_t(K) = \left\{ K \right\}$}
    \WHILE{the cardinality of $S_t(K)$ $<$ $\# S(K)$}
    \STATE{set $S_{t+1}(K) = S_t(K)$}
    \FOR{every $\wt{K} \in S_t(K)$}
    \STATE{add all adjacent face-neighbouring elements of $\wt{K}$ to
    $S_{t+1}(K)$}
    \ENDFOR
    \STATE{let $t = t+1$}
    \ENDWHILE
    \STATE{collect collocation points of all elements in $S_t(K)$ in
    $\mc{I}(K)$;}
    \STATE{sort the distances between points in $\mc{I}(K)$ and
    $\bm{x}_K$;}
    \STATE{select the $\# S(K)$ smallest values and collect the
    corresponding elements to form $S(K)$;}
    \ENDFOR
  \end{algorithmic}
\end{algorithm}

Moreover, for element $K \in \MTh$, we denote by $\mc{I}(K)$ the point
set formed of all collocation points with respect to the elements in
$S(K)$, 
\begin{displaymath}
  \mc{I}(K) := \left\{ \bm{x}_{\wt{K}}\ |\ \forall \wt{K} \in
  S(K) \right\}.
\end{displaymath}
Given a piecewise constant function $\bm{g} \in \bmr{U}_h^0$, for
every element $K \in \MTh$ we seek an irrotational polynomial
$\mc{R}_K \bm{g}$ of degree $m \geq 1$ defined on the patch $S(K)$ by
solving the following discrete least squares problem:
\begin{equation} 
  \begin{aligned} 
    \mc{R}_K \bm{g}  = &\mathop{\arg \min}_{\bm{p} \in
    \bmr{S}^m( S(K))} \sum_{ \bm{x} \in \mc{I}(K) } \left| \bm{p}(
    \bm{x})  - \bm{g}(\bm{x}) \right|^2 \\
    & \text{s.t. } \bm{p}(\bm{x}_K) = \bm{g}(\bm{x}_K). \\
  \end{aligned}
  \label{eq:lsp}
\end{equation}
The existence and uniqueness of the solution to \eqref{eq:lsp}
entirely depends on the geometrical positions of the points in
$\mc{I}(K)$. We thus make the following assumption
\cite{li2016discontinuous} to guarantee the well-posedness of the
problem \eqref{eq:lsp}.
\begin{assumption}
  For any element $K \in \MTh$ and $\bm{p} \in \bmr{S}^m(S(K))$, one
  has that
  \begin{displaymath}
    \bm{p}|_{\mc{I}(K)} = \bm{0} \quad \text{implies} \quad
    \bm{p}|_{S(K)} \equiv \bm{0}.
  \end{displaymath}
\end{assumption}
This assumption implies the number $\# S(K)$ shall be greater than
$\text{dim}(\bmr{S}^m(\cdot))$ and further rules out the case that all
the points in $\mc{I}(K)$ lie on an algebraic curve of degree $m$.

It should be noted that the solution to \eqref{eq:lsp} is linearly
dependent on $\bm{g}$, which inspires us to define a linear
reconstruction operator $\mc{R}$ for the function in $\bmr{U}_h^0$ in
an element-wise manner:
\begin{displaymath}
  (\mc{R} \bm{g})|_K = (\mc{R}_K \bm{g})|_K, \quad \text{on any } K
  \in \MTh. 
\end{displaymath}
Hence, the operator $\mc{R}$ maps the piecewise constant space
$\bmr{U}_h^0$ onto a subspace of $\bmr{S}_h^m$ and we denote by
$\bmr{U}_h^m = \mc{R} \bmr{U}_h^0$ the image of the operator. The
reconstructed space $\bmr{U}_h^m$ is actually the finite element space
we would adopt in next section.  For any function $\bm{q} \in
\bmr{I}^{m+1}(\Omega)$, we could define a function $\wt{\bm{q}} \in
\bmr{U}_h^0$ such that 
\begin{displaymath}
  \wt{\bm{q}}(\bm{x}_K) = \bm{q}(\bm{x}_K), \quad \forall K \in \MTh,
\end{displaymath}
which allows us to extend the operator $\mc{R}$ to act on the space
$\bmr{I}^{m+1}(\Omega)$ by directly letting $\mc{R}\bm{q} = \mc{R}
\wt{\bm{q}}$. Therefore, with the reconstruction operator $\mc{R}$,
the irrotational vector fields $\bm{q} \in \bmr{I}^{m+1}(\Omega)$ is
mapped to a piecewise irrotational polynomial in $\bmr{U}_h^m$. 

Further we outline a group of basis functions of the space
$\bmr{U}_h^m$ to present more details. We define a group of
characteristic functions $\left\{ \bm{w}_K^i \right\}$(for all $K$, $1
\leq i \leq d$) such that $\bm{w}_K^i(\bm{x}) \in \bmr{U}_h^0$ and 
\begin{displaymath}
  \bm{w}_K^i(\bm{x}) = \begin{cases}
    \bm{e}_i, &\text{in } K, \\
    \bm{0}, &\text{in other elements}, \\
  \end{cases}
\end{displaymath}
where $\bm{e}_i$ is a $d \times 1$ unit vector whose $i$-th entry is
1. Then we define $\bm{\lambda}_K^i = \mc{R} \bm{w}_K^i$ and we state
the following lemma to ensure the functions $ \left\{ \bm{\lambda}_K^i
\right\}$ is a group of basis functions of the space $\bmr{U}_h^m$. 
\begin{lemma}
  For any element $K$ and $1 \leq i \leq d$, the functions $ \{
  \bm{\lambda}_K^i \}$ are linearly independent and then $\bmr{U}_h^m$
  is spanned by $ \left\{\bm{\lambda}_K^i \right\}$. 
  \begin{proof}
    Since $\bm{\lambda}_K^i = \mc{R} \bm{w}_K^i$ and the constraint in
    \eqref{eq:lsp} guarantees that at the collocation points, $\bm{
    \lambda}_K^i$ satisfies
    \begin{equation}
      \bm{\lambda}_K^i(\bm{x}_{\wt{K}}) =
      \begin{cases} 
        \bm{e}_i, & \wt{K} = K, \\
        \bm{0}, & \wt{K} \neq K. \\
      \end{cases}
      \label{eq:lambdaorth}
    \end{equation}
    Then we consider a group of constants $ \left\{ a_{K, i} \right\}$
    such that 
    \begin{displaymath}
      \sum_{j = 1}^d \sum_{\widehat{K} \in \MTh} a_{\widehat{K}, j}
      \bm{\lambda}_{\widehat{K}}^j(\bm{x}) = 0, \quad \forall \bm{x}
      \in \Omega.
    \end{displaymath}
    For any element $\wt{K}$, we let $\bm{x} = \bm{x}_{\wt{K}}$ and by
    \eqref{eq:lambdaorth} we conclude that 
    \begin{displaymath}
      \sum_{j = 1}^d a_{\wt{K}, j}
      \bm{\lambda}_{\wt{K}}^j(\bm{x}_{\wt{K}})  =  \sum_{j = 1}^d
      a_{\wt{K}, j} \bm{e}_j = 0,
    \end{displaymath}
    which directly indicates $a_{\wt{K}, i} = 0$ for $1 \leq i \leq
    d$. Hence we obtain the coefficient $a_{K, i} = 0$ for any element
    $K$ and $i$, which gives the linear independence of $ \left\{
    \bm{\lambda}_K^i \right\}$. The linear operator $\mc{R}$ maps
    $\bmr{U}_h^0$ onto $\bmr{U}_h^m$ and the property of linear
    operator leads to $\text{dim}(\bmr{U}_h^m) \leq
    \text{dim}(\bmr{U}_h^0)$. Note that the number of $\left\{
    \bm{\lambda}_K^i \right\}$ is actually $ \text{dim}(\bmr{U}_h^0)$,
    thus we have $\bmr{U}_h^m = \text{span} \left\{ \bm{\lambda}_K^i
    \right\}$, which completes the proof.
  \end{proof}
  \label{le:basisfunctions}
\end{lemma}
Together with the basis functions, we can write the reconstruction
operator $\mc{R}$ explicitly: for any function $\bm{g}(\bm{x}) =
(g^1(\bm{x}), \ldots, g^d(\bm{x})) \in \bmr{I}^{m+1}(\Omega)$ or
$\bm{g}(\bm{x}) \in \bmr{U}_h^0$, we have that 
\begin{equation}
  \mc{R} \bm{g} =  \sum_{K \in \MTh} \sum_{i=1}^d g^i(\bm{x}_K)
  \bm{\lambda}_K^i(\bm{x}).
  \label{eq:writeR}
\end{equation}
In Appendix \ref{sec:appendix}, we give more details of the computer
implementation of the reconstructed space. 

Then we focus on the approximation properties of the operator
$\mc{R}$.  We define a constant $\Lambda(m, S(K))$ for element $K$ as 
\begin{displaymath}
  \Lambda(m, S(K)) := \max_{p \in \mb{P}_m(S(K))}
  \frac{\max_{\bm{x} \in S(K)} |p(\bm{x}) |  }{ \max_{\bm{x} \in
  \mc{I}(K) } |p( \bm{x})| }.
\end{displaymath}
Let $\Lambda_m := \max_{K \in \MTh} \left( 1 + \Lambda(m, S(K))
\sqrt{\# S(K)} \right)$, and under some mild and
practical conditions on element patches, $\Lambda_m$ admits a uniform
upper bound. We refer to \cite{li2012efficient, li2016discontinuous}
for these conditions and more discussion about the upper bound of 
$\Lambda_m$. Combining the approximation property of the irrotational
polynomial space \eqref{le:L2interpolation}, we state the following
approximation estimates of the reconstruction operator $\mc{R}$. 
\begin{theorem}
  For element $K$, the following estimates hold,
  \begin{equation}
    \begin{aligned}
      \| \bm{g} - \mc{R} \bm{g} \|_{H^q(K)} & \leq C \Lambda_m h_K^{m
      + 1 - q} \| \bm{g} \|_{H^{m+1}( S(K))}, \quad q = 0, 1, \\
      \| D^q (\bm{g} - \mc{R} \bm{g}) \|_{L^2(\partial K)} & \leq C
      \Lambda_m h_K^{m + 1 - q - 1/2} \| \bm{g} \|_{H^{m+1}(S(K))},
      \quad q = 0, 1, \\
    \end{aligned}
    \label{eq:apperror}
  \end{equation}
  for any $\bm{g} \in \bmr{I}^{m+1}(\Omega)$.
  \label{th:apperror}
\end{theorem}
\begin{proof}
  The proof of Theorem \ref{th:apperror} can be found in
  \cite{li2019sequential}. 
\end{proof}

\section{Numerical Scheme for Monge-Amp\`ere Equation}
\label{sec:scheme}
In this section, we present our numerical algorithm for approximating
solutions to the fully nonlinear elliptic Monge-Amp\`ere equation
which is given by 
\begin{equation}
  \begin{aligned}
    \det( D^2 u ) &= f, \quad \text{in } \Omega,  \\
    u &= g, \quad \text{on } \partial \Omega,
  \end{aligned}
  \label{eq:mae}
\end{equation}
where $D^2 u$ denotes the Hessian matrix of the unknown $u$ and $f$ is
a strictly positive function on $\Omega$. In this paper, we focus on
developing a numerical scheme for classical solutions to the
Monge-Amp\`ere equation. Particularly, the data function $f$ and $g$
are assumed to be sufficiently smooth to guarantee that there is a
strictly convex solution $u \in H^s(\Omega)$ with $s > 3$ to the
problem \eqref{eq:mae}. For the smoothness of $u$, we refer to
\cite{Caffarelli1984dirichlet, Courant1989methods,
Caffarelli1991regularity, Caffarelli1997properties} for more
regularity results. 

To solve the fully nonlinear problem \eqref{eq:mae}, we propose a new
numerical scheme which is based on the first-order system of
\eqref{eq:mae} and the Newton-type iteration. Our method rewrites the
problem \eqref{eq:mae} into an equivalent first-order system to define
a sequence of linear problems that can be solved by a discontinuous
least squares finite element method. The first-order system is defined
with introducing an auxiliary gradient variable $\bm{p} = \nabla u$,
which reads
\begin{equation}
  \begin{aligned}
    \det( \nabla \bm{p}) &= f, \quad \text{in } \Omega, \\
    \nabla u &= \bm{p}, \quad \text{in } \Omega, \\
    u &= g, \quad \text{on } \partial \Omega. \\
  \end{aligned}
  \label{eq:fosmae}
\end{equation}
We are motivated by the idea presented in \cite{li2019sequential,
li2019nondivergence}, which provides an idea of decoupling the
problem \eqref{eq:mae} into two sequential steps. The first step is to
solve a nonlinear first-order system, which reads
\begin{equation}
  \begin{aligned}
    \det( \nabla \bm{p}) &= f, &&\text{in } \Omega, \\
    \bm{p} \times \un &= \nabla g \times \un, &&\text{on } \partial
    \Omega. \\
  \end{aligned}
  \label{eq:firstfos}
\end{equation}
This system includes the first equation in \eqref{eq:fosmae} and the
boundary condition of $u$ provides the tangential trace $\bm{p} \times
\un$ on the boundary. The other two equations in \eqref{eq:fosmae}
form the second linear first-order system, which reads
\begin{equation}
  \begin{aligned}
    \nabla u &= \bm{p}, \quad \text{in } \Omega, \\
    u &= g, \quad \text{on } \partial \Omega. \\
  \end{aligned}
  \label{eq:secondfos}
\end{equation}
We first solve the nonlinear problem \eqref{eq:firstfos} using the
piecewise irrotational space $\bmr{U}_h^m$ to obtain a numerical
approximation $\bm{p}_h$ to the gradient $\bm{p}$. Then together with
$\bm{p}_h$ we solve the second linear problem \eqref{eq:secondfos} to
get the numerical solution to $u$. The numerical scheme of the
nonlinear problem \eqref{eq:firstfos} is the main component of our
whole algorithm. Now we begin by focusing on the linearization to the
problem \eqref{eq:firstfos}. 
\begin{lemma}
  For any piecewise smooth functions $\bm{v}, \bm{w} \in H^1(\MTh)^d$,
  there holds:
  \begin{equation}
    \begin{aligned}
      \mathrm{det}(\nabla_h (\bm{v} + \bm{w})) = \det(\nabla_h \bm{v})
      + \det(\nabla_h \bm{w}) + \cof(\nabla_h \bm{v}) : \nabla_h
      \bm{w}, \quad d=2,  \\
      \mathrm{det}(\nabla_h(\bm{v} + \bm{w}))  = \det(\nabla_h \bm{v})
      + \det(\nabla_h \bm{w}) + \cof(\nabla_h \bm{v}) : \nabla_h
      \bm{w} + \cof(\nabla_h \bm{w}) : \nabla_h \bm{v}, \quad d=3,\\
    \end{aligned}
    \label{eq:detAB}
  \end{equation}
  where $\cof(\cdot)$ denotes the cofactor matrix. 
  \label{le:detAB}
\end{lemma}
\begin{proof}
  We refer to \cite{Brenner2012finite, Gerard2015standard} for the
  proof.
\end{proof}


\begin{lemma}
  For any piecewise smooth functions $\bm{v}, \bm{w} \in H^1(\MTh)^d$,
  there holds
  \begin{equation}
    \lim_{t \rightarrow 0} \frac{\det(\nabla_h ( \bm{v} + t \bm{w})) -
    \det(\nabla_h \bm{v})  }{t} = \cof(\nabla_h \bm{v}) : \nabla_h
    \bm{w}.
    \label{eq:linearied}
  \end{equation}
  \label{le:linearied}
\end{lemma}
\begin{proof}
  For $d=2$, by Lemma \ref{le:detAB} we observe that 
  \begin{displaymath}
    \begin{aligned}
      \frac{\det(\nabla_h ( \bm{v} + t \bm{w})) - \det(\nabla_h
      \bm{v})}{t} &= \cof(\nabla_h \bm{v}) : \nabla_h \bm{w} + t
      \det(\nabla_h \bm{w}),
    \end{aligned}
  \end{displaymath}
  and for $d=3$, we have 
  \begin{displaymath}
    \begin{aligned}
      \frac{\det(\nabla_h ( \bm{v} + t \bm{w})) - \det(\nabla_h
      \bm{v})}{t} &=\cof(\nabla_h \bm{v}) : \nabla_h \bm{w} + t \left(
      \cof(\nabla_h \bm{w}) : \nabla_h \bm{v} \right) + t^2 \det(
      \nabla_h \bm{w}).
    \end{aligned}
  \end{displaymath}
  We obtain the equation \eqref{eq:linearied} by letting $t
  \rightarrow 0$, which completes the proof.
\end{proof}

On the continuous level, by Lemma \ref{le:linearied}, the
linearization of the first-order system \eqref{eq:firstfos} at the
function $\bm{p}$ is given by 
\begin{displaymath}
  \lim_{t \rightarrow 0} \frac{\det(\nabla \bm{p} + t \nabla \bm{q}) -
  \det(\nabla \bm{p}) }{t} = \cof(\nabla \bm{p}) : \nabla \bm{q}.
\end{displaymath}

Then the basic idea of the Newton iteration is that with a given
approximation $\wt{\bm{p}}$ we seek the next numerical approximation
by solving the problem 
\begin{equation}
  \cof( \nabla \wt{\bm{p}}) : \delta \bm{p} = f - \det( \nabla \wt{
  \bm{p}}),
  \label{eq:newton1}
\end{equation}
and then update with $\widehat{\bm{p}} = \wt{\bm{p}} + \delta \bm{p}$.
It is noticeable that this linearization is just formulated on the
continuous level. We shall consider the Newton iteration on the
discrete level. With a numerical approximation $\bm{p}_h^n$, the
problem \eqref{eq:newton1} shall be formally adapted as to find
$\bm{p}_h^{n+1}$ such that
\begin{equation}
  \cof( \nabla_h \bm{p}_h^n ) : (\bm{p}_h^{n+1} - \bm{p}_h^n) = f -
  \det( \nabla_h \bm{p}_h^n). 
  \label{eq:newton2}
\end{equation}
The reconstructed approximation space $\bmr{U}_h^m$ we introduced in
the previous section involves discontinuity across the interelement
\cite{li2019sequential}. In this finite element setting, the problem
\eqref{eq:newton2} results in a non-divergence elliptic problem with
discontinuous coefficients $\cof(\nabla_h \bm{p}_h^n)$ and such a
problem also does not naturally fit within the standard Galerkin
framework.  Instead, we propose a discontinuous least squares
variational problem that allows discontinuous approximation space and
piecewise discontinuous coefficients. The least squares functional
$J_h^{\bmr{p}}(\cdot; \cdot)$ is defined as 
\begin{equation}
  \begin{aligned}  
    J_h^{\bmr{p}}(\bm{q}; \bm{w}) := \sum_{K \in \MTh} \|\cof( \bm{w})
    : \nabla \bm{q} - & \left( f - \det( \bm{w}) + \cof( \bm{w}) :
    \nabla \bm{w} \right)\|_{L^2(K)}^2 \\ 
    + & \sum_{e \in \MEh^i} \frac{\eta }{h_e} \| \jump{ \bm{ \bm{q}
    \otimes \un} } \|_{L^2(e)}^2 + \sum_{e \in \MEh^b} \frac{\eta
    }{h_e} \| \bm{q} \times \un - \nabla g \times \un \|_{L^2(e)}^2,
    \\
  \end{aligned}
  \label{eq:lsfp}
\end{equation}
where $\eta$ is the positive penalty parameter. The first term in
\eqref{eq:lsfp} corresponds to the problem \eqref{eq:newton2} when
$\bm{w}$ just takes $\bm{w} = \nabla_h \bm{p}_h^n$.
The second term is used to weakly impose the continuity condition
since our goal is to approximate the classical smooth solution to the
Monge-Amp\`ere equation. The last term in \eqref{eq:lsfp} is to weakly
impose the boundary condition in \eqref{eq:firstfos}.  Given the
approximation $\bm{p}_h^n$ at step $n$, we minimize the functional
$J_h^{\bmr{p}}(\cdot; \nabla_h \bm{p}_h^n)$ over the reconstructed
space $\bmr{U}_h^m$ to obtain the  next level approximation
$\bm{p}_h^{n + 1}$: 
\begin{equation}
  \bm{p}_h^{n + 1} = \mathop{\arg \min}_{ \bm{q}_h \in \bmr{U}_h^m}
  J_h^{\bmr{p}}( \bm{q}_h; \nabla_h \bm{p}_h^n).
  \label{eq:minpn1}
\end{equation}
We write its corresponding Euler-Lagrange equation to solve the
minimization problem. The problem \eqref{eq:minpn1} is equivalent to
the variational equation which takes the form: find $\bm{p}_h^{n+1}
\in \bmr{U}_h^m$ such that 
\begin{equation}
  a_h^{\bmr{p}}(\bm{p}_h^{n+1}, \bm{q}_h; \nabla_h \bm{p}_h^n) =
  l_h^{\bmr{p}} (\bm{q}_h; \nabla_h \bm{p}_h^n), \quad \forall
  \bm{q}_h \in \bmr{U}_h^m,
  \label{eq:ELequation}
\end{equation}
where the bilinear form $a_h^{\bmr{p}}(\cdot, \cdot; \cdot)$ is 
\begin{displaymath}
  \begin{aligned}
    a_h^{\bmr{p}}(\bm{p}, \bm{q}; \bm{w}) = \sum_{K \in \MTh} \int_K
    (\cof(\bm{w}) : \bm{p}) ( \cof(\bm{w}) : \bm{q}) \d{x} & +
    \sum_{e \in \MEh^i} \int_e \frac{\eta }{h_e} \jump{\bm{p} \otimes
    \un } :  \jump{\bm{q} \otimes \un} \d{s} \\ 
    & + \sum_{e \in \MEh^b} \int_e \frac{\eta}{h_e} (\bm{p} \times
    \un) (\bm{q} \times \un) \d{s},\\
  \end{aligned}
\end{displaymath}
and the linear form $l_h^{\bmr{p}}(\cdot; \cdot)$ is 
\begin{displaymath}
  \begin{aligned}
    l_h^{\bmr{p}}(\bm{q}; \bm{w}) = \sum_{K \in \MTh} \int_K (\cof(
    \bm{w}) : \nabla \bm{q})\left( f - \det( \bm{w})  + \cof(\bm{w}) :
    \nabla \bm{w} \right) \d{x} + \sum_{e \in \MEh^b} \int_e
    \frac{\eta}{h_e} (\bm{q} \times \un) (\nabla g \times \un) \d{s}.
  \end{aligned}
\end{displaymath}
Finally, we present the algorithm for computing the numerical solution
$\bm{p}_h$ in approximation to the gradient $\bm{p}$ in Algorithm
\ref{alg:iteration}. In Algorithm \ref{alg:iteration}, the stop
criterion can be taken as 
\begin{displaymath}
  \frac{\enorm{ \bm{p}_h^n - \bm{p}_h^{n-1}}}{ \enorm{ \bm{p}_h^{n -
  1 }}} < \varepsilon, 
\end{displaymath}
where the norm $\enorm{ \cdot}$ could be $L^2$ norm $\| \cdot
\|_{L^2(\Omega)}$ or the discrete $l^2$ vector norm $\| \cdot
\|_{l^2}$ which acts on the corresponding finite element solution. 

\begin{algorithm}[htb]
  \caption{Newton Iteration with Least Squares Method} 
  \label{alg:iteration}
  \begin{algorithmic}[1]
    \renewcommand{\algorithmicrequire}{\textbf{Input:}}
    \REQUIRE The initial value $\bm{p}_h^0$; \\
    \renewcommand{\algorithmicrequire}{\textbf{Output:}}
    \REQUIRE The numerical solution to the gradient $\bm{p}_h$; \\
    \STATE{initialize $n=0$;} 
    \WHILE{not satisfy the stop criterion}
    \STATE{compute the coefficient $\cof(\nabla_h \bm{p}_h^n)$;}
    \STATE{solve the minimization problem \eqref{eq:minu} to obtain
    $\bm{p}_h^{n+1}$;}
    \STATE{$n = n + 1$}
    \ENDWHILE
    \STATE{$\bm{p}_h = \bm{p}_h^n$;}
  \end{algorithmic}
\end{algorithm}

We define an energy norm $\pnorm{\cdot}$ for any vector-valued
function $\bm{q} \in H^1(\MTh)^d$ by
\begin{equation}
  \pnorm{ \bm{q}}^2 := \sum_{K \in \MTh} \| \nabla \bm{q}
  \|_{L^2(K)}^2 + \sum_{e \in \MEh^i} \frac{1}{h_e} \| \jump{ \bm{q}
  \otimes \un} \|_{L^2(e)}^2 + \sum_{e \in \MEh^b} \frac{1}{h_e} \|
  \bm{q} \times \un \|_{L^2(e)}^2.
  \label{eq:pnorm}
\end{equation}
The norm $\pnorm{\cdot}$ is proven to be stronger than the broken
Sobolev norm $\| \cdot \|_{H^1(\MTh)}$ by the following lemma. 
\begin{lemma}
  The following inequality holds, 
  \begin{displaymath}
    \| \bm{q} \|_{H^1(\MTh)} \leq C \pnorm{\bm{q}},
  \end{displaymath}
  for any $\bm{q} \in H^1(\MTh)^d$.
  \label{le:pisnorm}
\end{lemma}
\begin{proof}
  The proof can be found in \cite{li2019sequential, Bensow2005div}.
\end{proof}
With respect to the energy norm $\pnorm{\cdot}$, we may expect the
numerical solution $\bm{p}_h$ to the gradient has the finite element
estimate which reads
\begin{equation} 
  \pnorm{\bm{p} - \bm{p}_h} \leq C h^m \|\bm{p}\|_{H^{m+1}(\Omega)}, 
  \label{eq:perror}
\end{equation} 
where we assume the problem \eqref{eq:firstfos} has the exact solution
$\bm{p} \in \bmr{I}^{m+1}(\Omega)$. In Section
\ref{sec:numericalresult}, this estimate is confirmed by a series of
numerical results. 

We note that the problem at each nonlinear iteration step could be
regarded as solving the non-divergence form elliptic problem which has
the form $\cof(\nabla_h \bm{p}_h^n) : \nabla \bm{p} = \bm{f}$. 
The coefficient $\cof(\nabla_h \bm{p}_h^n)$ is
symmetric since the piecewise irrotational property of the
approximation space. If the discontinuous coefficient satisfies
SPD condition in two dimensions and satisfies Cord\`e condition
\cite{Smears2013discontinuous} in three dimensions, the error estimate
of the least squares method of \eqref{eq:lsfp} and \eqref{eq:minpn1}
under the energy norm $\pnorm{\cdot}$ has been established in detail
in \cite{li2019nondivergence}. Hence, we may expect the numerical
solution to the nonlinear system \eqref{eq:firstfos} satisfies the
error estimate in \cite{li2019nondivergence}, which actually is the
estimate \eqref{eq:perror}.  Moreover, in the iteration setting the
coefficient comes from the finite element solution and we obviously
only have symmetric piecewise polynomial coefficient. This fact seems
like a notable difference between nonlinear problem and the standard
linear non-divergence form elliptic problem and leads to hard
theoretical verification for the convergence. In Section
\ref{sec:numericalresult}, the numerical results demonstrate that the
Newton iteration together with least squares framework retains fast
convergence speed and keeps optimal finite element convergence rate
with respect to the error measurement $\pnorm{\cdot}$ and we also
demonstrate the piecewise convexity of the numerical solution.


\begin{remark}
  The minimization problem \eqref{eq:minpn1} for the nonlinear
  iteration only requires a piecewise irrotational approximation
  space and here we adopt the reconstructed space $\bmr{U}_h^m$ to
  seek a numerical approximation to the nonlinear system
  \eqref{eq:firstfos}.  We note that the standard piecewise
  irrotational space $\bmr{S}_h^m$ can also be used for solving if we
  substitute the space $\bmr{U}_h^m$ by $\bmr{S}_h^m$ in
  \eqref{eq:minpn1} and \eqref{eq:ELequation}. In Section
  \ref{sec:comparison}, we make a comparison between $\bmr{U}_h^m$ and
  $\bmr{S}_h^m$ to show the great efficiency and robustness of the
  proposed method which may be due to the use of the space
  $\bmr{U}_h^m$.
  \label{re:space}
\end{remark}

After solving the nonlinear problem \eqref{eq:firstfos}, we then move
on to the second first-order system \eqref{eq:secondfos}. With a given
numerical approximation $\bm{p}_h$, we propose another least squares
functional $J_h^{\bmr{u}}(\cdot; \cdot)$ to seek a numerical solution
for $u$, where $J_h^{\bmr{u}}(\cdot; \cdot)$ is defined by 
\begin{equation}
  J_h^{\bmr{u}}(v; \bm{q}) := \sum_{K \in \MTh} \| \nabla u - \bm{q}
  \|_{L^2(K)}^2 + \sum_{e \in \MEh^b} \frac{1}{h} \|u - g
  \|_{L^2(e)}^2.
  \label{eq:lsfu}
\end{equation}
We adopt the standard $C^0$ finite element space $\wt{V}_h^m$ of
degree $m$ to minimize the functional $J_h^{\bmr{u}}(\cdot; \cdot)$ to
obtain a numerical approximation $u_h \in \wt{V}_h^m$.  Precisely, the
minimization problem reads 
\begin{equation}
  u_h = \mathop{\arg \min}_{v_h \in \wt{V}_h^m} J_h^{\bmr{u}}(v_h;
  \bm{p}_h),
  \label{eq:minu}
\end{equation}
and its corresponding variational problem reads: find $u_h \in
\wt{V}_h^m$ such that 
\begin{displaymath}
  a_h^{\bmr{u}}(u_h, v_h) = l_h^{\bmr{u}}(v_h; \bm{p}_h), \quad
  \forall v_h \in \wt{V}_h^m,
\end{displaymath}
where the bilinear form $a_h^{\bmr{u}}(\cdot, \cdot)$ is 
\begin{displaymath}
  a_h^{\bmr{u}}(u, v) = \sum_{K \in \MTh} \int_K \nabla u \cdot
  \nabla v \d{x} + \sum_{e \in \MEh^b} \int_e \frac{1}{h} uv
  \d{s},
\end{displaymath}
and the linear form $l_h^{\bmr{u}}(\cdot; \cdot)$ is 
\begin{displaymath}
  l_h^{\bmr{u}}(v; \bm{q}) = \sum_{K \in \MTh} \int_K \nabla v \cdot
  \bm{q} \d{x} + \sum_{e \in \MEh^b} \int_e \frac{1}{h} v g \d{s}.
\end{displaymath}

We introduce an energy norm $\unorm{\cdot}$ that is naturally induced
from the bilinear form $a_h^{\bmr{u}}(\cdot, \cdot)$, 
\begin{displaymath}
  \unorm{v}^2 := \sum_{K \in \MTh} \| \nabla v \|_{L^2(K)}^2 + \sum_{e
  \in \MEh^b} \frac{1}{h} \| v\|_{L^2(e)}^2, \quad \forall v \in
  H^1(\Omega).
\end{displaymath}
It is trivial to check that $\unorm{\cdot}$ is a norm on
$H^1(\Omega)$. The error estimates under $L^2$ norm $\| \cdot
\|_{L^2(\Omega)}$ and energy norm $\unorm{\cdot}$ have been proven in
\cite{li2019nondivergence, li2019sequential} and here we present the
main results. 
\begin{theorem}
  Let $u \in H^{m+1}(\Omega)$ be the exact solution to \eqref{eq:mae}
  and let $u_h \in \wt{V}_h^m$ be the numerical solution to
  \eqref{eq:minu}, then the following estimates hold:
  \begin{equation} 
    \begin{aligned} 
      \unorm{u - u_h} & \leq C \left( h^m \| u \|_{H^{m+1}(\Omega)} +
      \| \bm{p} - \bm{p}_h \|_{L^2(\Omega)} \right), \\ 
      \|u - u_h \|_{L^2(\Omega)} & \leq C \left( h^{m + 1} \| u
      \|_{H^{m+1}(\Omega)} + \| \bm{p} - \bm{p}_h \|_{L^2(\Omega)}
      \right), 
      \label{eq:uerror} 
    \end{aligned}
  \end{equation}
  where $\bm{p}_h$ is the given numerical approximation to the
  gradient. 
  \label{th:uerror}
\end{theorem}


\section{Numerical Results}
\label{sec:numericalresult}
In this section we present some numerical computational examples both
in two and three dimensions to demonstrate numerical performance of
the proposed method.

\subsection{2D Examples} 
We first present the two-dimensional numerical examples. For all
tests, the computational domain $\Omega$ is chosen as $(0, 1)^2$ and
the penalty parameter $\eta$ in \eqref{eq:lsfp} is taken as $20$. We
employ the triangular meshes with resolutions of $h = 1/10$, $h =
1/20$, $h = 1/40$ and $h = 1/80$, see Fig~\ref{fig:partition}. In
Tab.~\ref{tab:numSK}, we list the values of $\# S(K)$ that are used in
the numerical tests.  We take the pair of approximation spaces
$\bmr{U}_h^m$ and $\wt{V}_h^m$ with the accuracy $1 \leq m \leq 3$ to
solve $\bm{p}$ and $u$, respectively. The stop criterion in Algorithm
\ref{alg:iteration} is selected as 
\begin{displaymath} 
  \frac{ \| \bm{p}_h^{n} - \bm{p}_h^{n - 1} \|_{l^2} }{
  \| \bm{p}_h^{n - 1} \|_{l^2} } < 10^{-10}. 
\end{displaymath}

\begin{figure}
  \centering
  \includegraphics[width=0.4\textwidth]{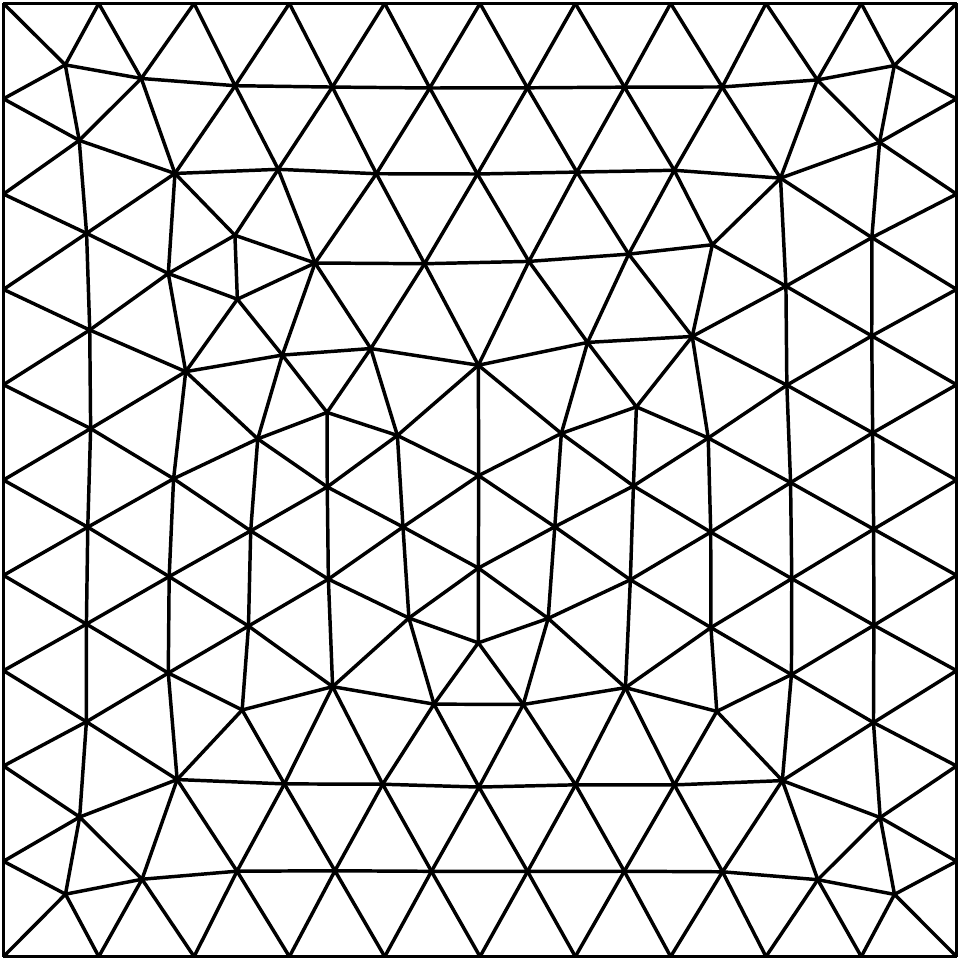}
  \hspace{25pt}
  \includegraphics[width=0.4\textwidth]{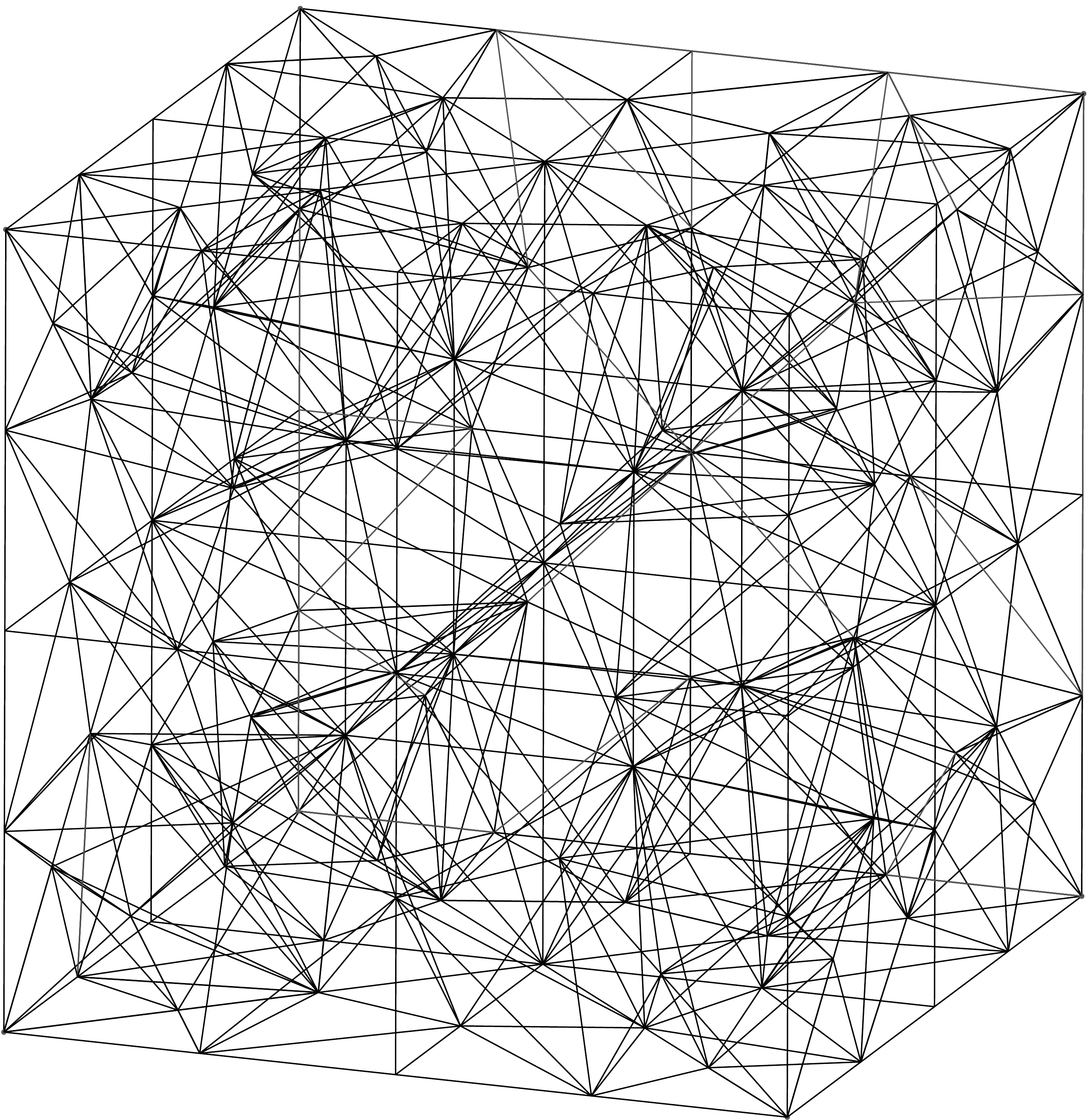}
  \caption{2d triangular partition with $h = 1/10$ (left) / 3d
  tetrahedral partition with $h = 1/4$ (right).}
  \label{fig:partition}
\end{figure}

\begin{table}
  \centering
  \renewcommand{\arraystretch}{1.3}
  \begin{tabular}{p{1.5cm} | p{1.5cm} | p{1cm} | p{1cm} | p{1cm}  }
    \hline\hline 
    & $m$ & 1 & 2 & 3 \\
    \hline
    $d=2$ & \multirow{2}{1.5cm}{$\# S(K)$} & 5 & 9 & 14 \\
    \cline{1-1}\cline{3-5}
    $d=3$ &  & 9 & 18 & 35 \\
    \hline\hline
  \end{tabular}
  \caption{The uniform $\# S(K)$ for $1 \leq m \leq 3$ in two and
  three dimensions.}
  \label{tab:numSK}
\end{table}

\paragraph{\textbf{Example 1:}} We first consider a
two-dimensional benchmark problem \cite{Benamou2010numerical} to
investigate convergence behaviour of the proposed numerical scheme. We
take the source function 
\begin{displaymath}
  f(x, y) = \left( 1 + x^2 + y^2 \right) \exp\left( x^2 + y^2 \right),
\end{displaymath}
which gives the smooth convex analytical solution 
\begin{displaymath}
  u(x, y) = \exp\left( \frac{1}{2} (x^2 + y^2) \right).
\end{displaymath}
The initial value is chosen as the solution of the following Poisson's
equation 
\begin{displaymath}
  \Delta u = 2 \sqrt{f}, \quad \text{in } \Omega, \quad u = g, \quad
  \text{on } \partial \Omega.
\end{displaymath}
This Poisson problem with the Dirichlet boundary condition is
considered to be closely related to the corresponding Monge-Amp\`ere
equation \cite{Dean2006numerical}. We thus apply its finite element
solution to start the Newton iteration.  The numerical errors for both
$\bm{p}$ and $u$ are listed in Tab.~\ref{tab:example1}. We first
observe that the convergence rates for the two unknowns under the
energy norms $\pnorm{\cdot}$ and $\unorm{\cdot}$ are optimal $O(h^m)$.
The optimal $L^2$ convergence rates are numerically detected for odd
$m$. This odd/even situation has also been observed in
\cite{li2019sequential}.  Besides, all computed convergence rates are
consistent with the estimate \eqref{eq:perror} and \eqref{eq:uerror}.
We also include the number of Newton steps in Tab.~\ref{tab:example1}.
Our method shows a very fast convergence speed and almost only 5
Newton steps are required for all $1 \leq m \leq 3$.

\begin{table}
  \renewcommand{\arraystretch}{1.3}
  \centering
  \scalebox{0.88}{
  \begin{tabular}{c|r|c|c|c|c|c|c|c|c|c}
    \hline\hline
    $m$ & $h~~~$ & $ \pnorm{\bm{p} - \bm{p}_h} $ & order & $ \| \bm{p} -
    \bm{p}_h\|_{L^2(\Omega)}$  & order & $ \unorm{u - u_h} $
    & order & $ \| u - u_h\|_{L^2(\Omega)}$ & order & \# iter
    \\
    \hline 
    \multirow{4}{0.6cm}{1} & $1/10$ & 
    1.643e-1 & -    & 7.217e-3 & - & 
    6.799e-2 & -    & 1.273e-3 & - & 6 \\
    \cline{2-11}
    & $1/20$ &
    8.078e-2 & 1.02 & 2.268e-3 & 1.67 & 
    3.416e-2 & 0.99 & 3.739e-4 & 1.77 & 6 \\
    \cline{2-11}
    & $1/40$ &
    3.911e-2 & 1.04 & 5.522e-4 & 2.03 & 
    1.709e-2 & 1.00 & 8.602e-5 & 2.11 & 6 \\
    \cline{2-11}
    & $1/80$ &
    1.952e-2 & 1.01 & 1.459e-4 & 1.95 & 
    8.553e-3 & 1.00 & 2.071e-5 & 2.06 & 5 \\
    \hline
    \multirow{4}{0.6cm}{2} & $1/10$ & 
    1.621e-2 & -    & 2.519e-3 & - & 
    2.972e-3 & -    & 5.368e-4 & - & 5 \\
    \cline{2-11}
    & $1/20$ &
    3.843e-3 & 2.07 & 6.961e-4 & 1.86 & 
    8.039e-4 & 1.89 & 1.482e-4 & 1.85 & 5 \\
    \cline{2-11}
    & $1/40$ &
    9.837e-4 & 1.98 & 1.807e-4 & 1.95 & 
    2.071e-4 & 1.96 & 3.826e-5 & 1.95 & 5 \\
    \cline{2-11}
    & $1/80$ &
    2.492e-4 & 1.99 & 4.575e-5 & 1.99 & 
    5.231e-5 & 1.99 & 9.666e-6 & 1.99 & 5 \\
    \hline
    \multirow{4}{0.6cm}{3} & $1/10$ & 
    1.416e-3 & -    & 4.783e-5 & - & 
    5.007e-5 & -    & 3.923e-6 & - & 5 \\
    \cline{2-11}
    & $1/20$ &
    1.515e-4 & 3.22 & 4.317e-6 & 3.47 & 
    5.453e-6 & 3.19 & 4.640e-7 & 3.08 & 4 \\
    \cline{2-11}
    & $1/40$ &
    1.845e-5 & 3.03 & 2.912e-7 & 3.90 & 
    5.347e-7 & 3.35 & 3.151e-8 & 3.89 & 5 \\
    \cline{2-11}
    & $1/80$ &
    2.287e-6 & 3.01 & 1.892e-8 & 3.96 & 
    6.026e-8 & 3.13 & 2.026e-9 & 3.98 & 5 \\
    \hline\hline
  \end{tabular}}
  \caption{Convergence results of the Example 1.}
  \label{tab:example1}
\end{table}

\paragraph{\textbf{Example 2:}} In this example, we solve the
problem with the same analytical solution as Example 1 but we choose a
very different initial value to demonstrate the robustness of the
proposed method. We apply the smooth function 
\begin{equation}
  u^0(x, y) = 5x^4 + 10x^2 - xy + 0.1y^2 -5x -3y,
  \label{eq:test2initial}
\end{equation}
to start the nonlinear iteration. We note that this initial guess is
convex but it is very far from the exact solution. The numerical
errors and the number of nonlinear iterations are collected in
Tab.~\ref{tab:example2}.  We see a very similar numerical result as in
Example 1. For such an initial guess, the number of Newton iterations
is still less than $15$.  Fig.~\ref{fig:ex2history} shows the
convergence history of the relative error  $ { \| \bm{p}_h^{n} -
\bm{p}_h^{n - 1} \|_{l^2} } / { \| \bm{p}_h^{n - 1} \|_{l^2} }$  on
the finest mesh $h = 1/80$ for $1 \leq m \leq 3$. The history of
errors seems to be consistent with the performance of the Newton
method, which provides a slow convergence speed when the finite
element solution is far from the exact solution.  As the numerical
solution gets closer, the numerical solution may be quadratically
convergent to the exact solution. We report the $L^2$ errors $\|
\bm{p}_h^n - \bm{p} \|_{L^2(\Omega)}$ in Fig.~\ref{fig:ex2l2history}
where $\bm{p}_h^n$ is the numerical solution at iteration step $n$ and
$\bm{p}_h^0$ is the starting value.  The decrease of the $L^2$ error
is significant even with such a bad initial value. 

\begin{table}
  \renewcommand{\arraystretch}{1.3}
  \centering
  \scalebox{0.88}{
  \begin{tabular}{c|r|c|c|c|c|c|c|c|c|c}
    \hline\hline
    $m$ & $h~~~$ & $ \pnorm{\bm{p} - \bm{p}_h} $ & order & $ \| \bm{p} -
    \bm{p}_h\|_{L^2(\Omega)}$  & order & $ \unorm{u - u_h} $
    & order & $ \| u - u_h\|_{L^2(\Omega)}$ & order & \# iter
    \\
    \hline 
    \multirow{4}{0.6cm}{1} & $1/10$ & 
    1.644e-1 & -    & 7.315e-3 & - & 
    6.801e-2 & -    & 1.290e-3 & - & 10 \\
    \cline{2-11}
    & $1/20$ &
    8.088e-2 & 1.02 & 2.310e-3 & 1.67 & 
    3.416e-2 & 1.00 & 3.813e-4 & 1.75 & 10 \\
    \cline{2-11}
    & $1/40$ &
    3.913e-2 & 1.05 & 5.620e-4 & 2.04 & 
    1.709e-2 & 1.00 & 8.801e-5 & 2.12 & 10 \\
    \cline{2-11}
    & $1/80$ &
    1.953e-2 & 1.00 & 1.460e-4 & 1.96 & 
    8.553e-3 & 1.00 & 2.100e-5 & 2.07 & 11 \\
    \hline
    \multirow{4}{0.6cm}{2} & $1/10$ & 
    1.621e-2 & -    & 2.519e-3 & - & 
    2.972e-3 & -    & 5.368e-4 & - & 9 \\
    \cline{2-11}
    & $1/20$ &
    3.844e-3 & 2.08 & 6.960e-4 & 1.86 & 
    8.039e-4 & 1.89 & 1.481e-4 & 1.85 & 10 \\
    \cline{2-11}
    & $1/40$ &
    9.837e-4 & 1.97 & 1.807e-4 & 1.95 & 
    2.071e-4 & 1.95 & 3.826e-5 & 1.95 & 12 \\
    \cline{2-11}
    & $1/80$ &
    2.492e-4 & 1.99 & 4.576e-5 & 1.99 & 
    5.231e-5 & 1.99 & 9.667e-6 & 1.99 & 13 \\
    \hline
    \multirow{4}{0.6cm}{3} & $1/10$ & 
    1.415e-3 & -    & 4.783e-5 & - & 
    5.007e-5 & -    & 3.923e-6 & - & 9 \\
    \cline{2-11}
    & $1/20$ &
    1.515e-4 & 3.22 & 4.317e-6 & 3.47 & 
    5.453e-6 & 3.20 & 4.640e-7 & 3.08 & 10 \\
    \cline{2-11}
    & $1/40$ &
    1.845e-5 & 3.04 & 2.912e-7 & 3.89 & 
    5.348e-7 & 3.35 & 3.157e-8 & 3.88 & 11 \\
    \cline{2-11}
    & $1/80$ &
    2.287e-6 & 3.01 & 1.892e-8 & 3.95 & 
    6.031e-8 & 3.15 & 2.027e-9 & 3.97 & 14 \\
    \hline\hline
  \end{tabular}}
  \caption{Convergence of the Example 2.}
  \label{tab:example2}
\end{table}

\begin{figure}
  \centering
  \includegraphics[width=0.32\textwidth]{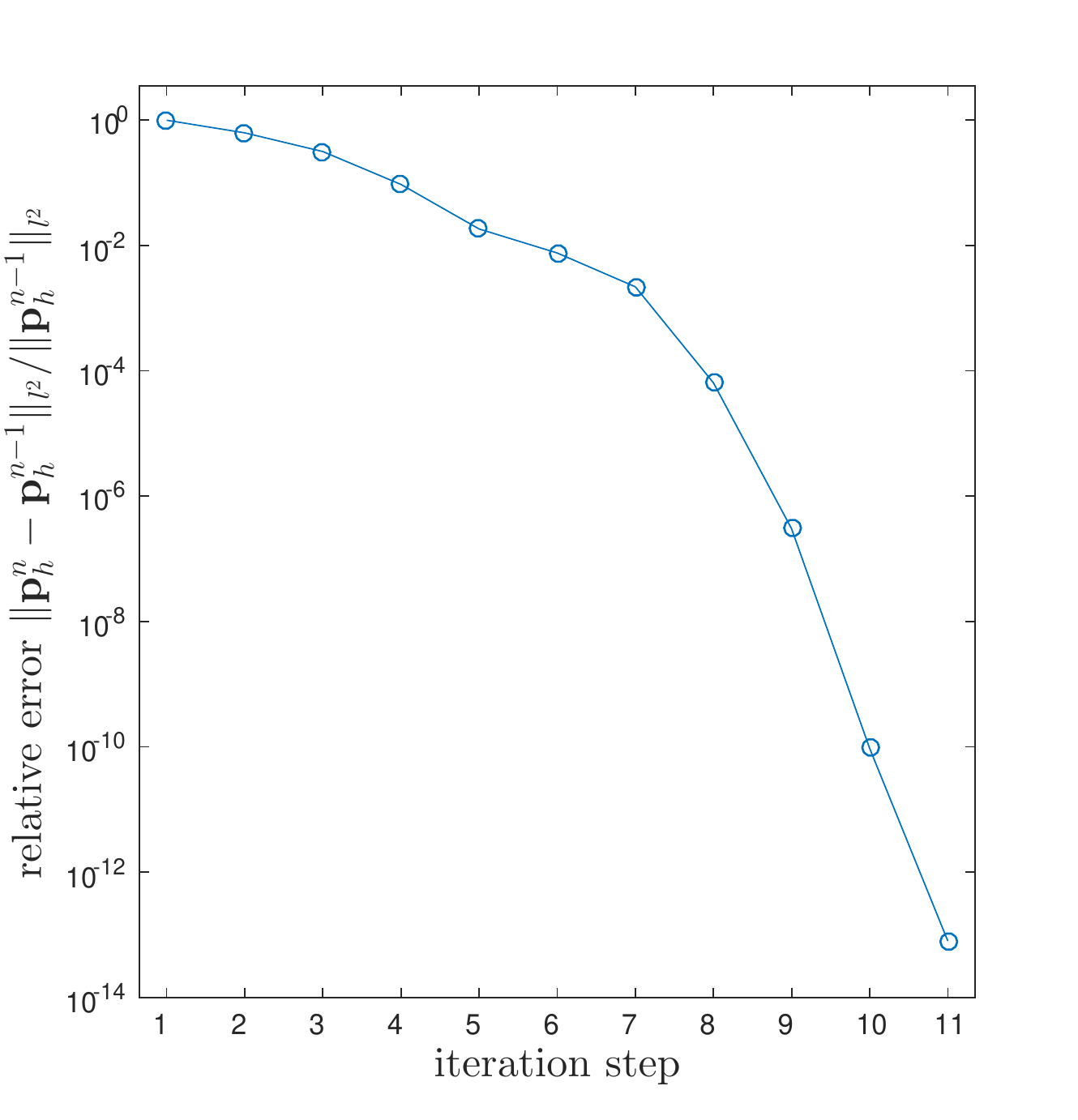}
  \hspace{0pt}
  \includegraphics[width=0.32\textwidth]{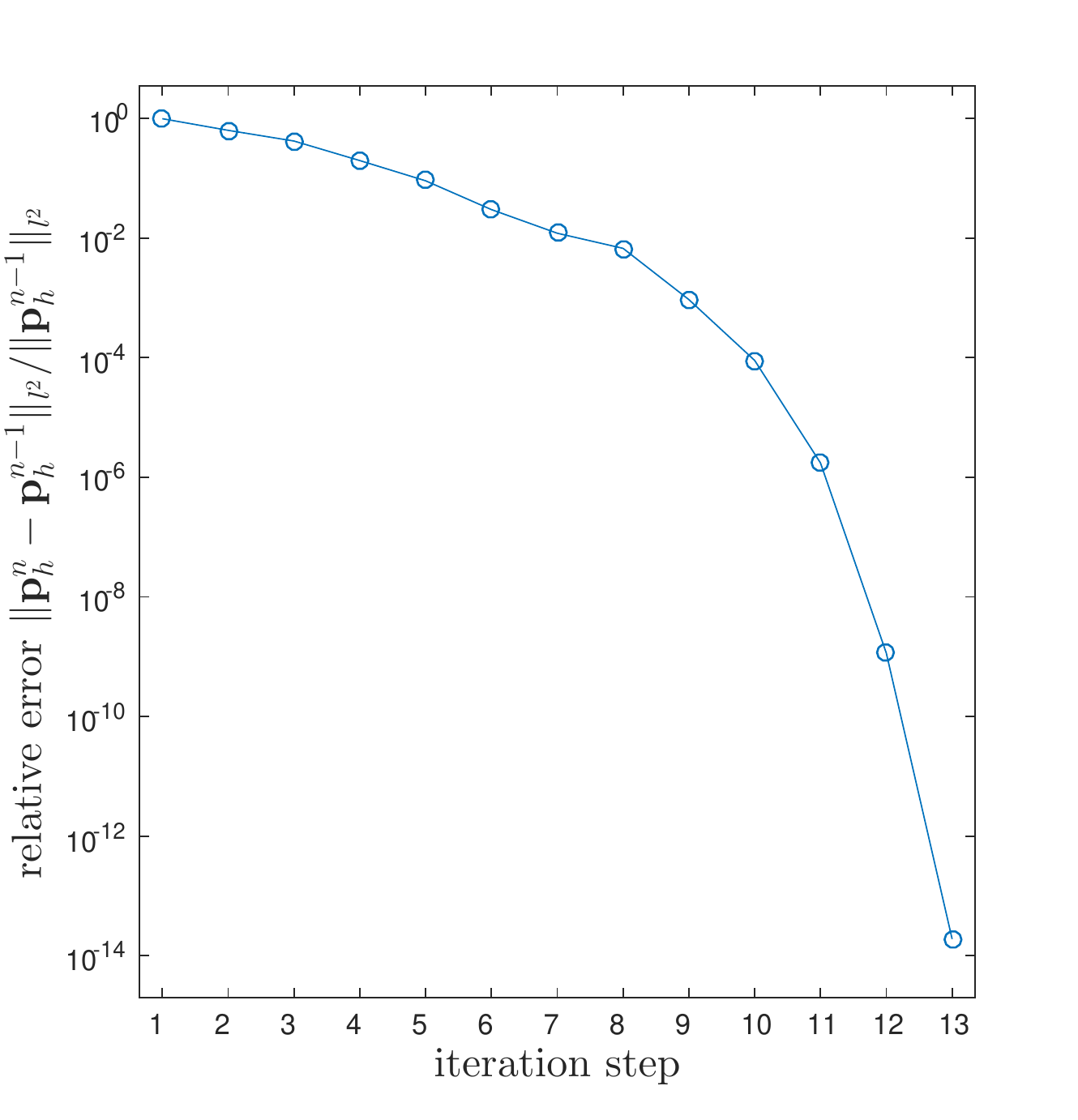}
  \hspace{0pt}
  \includegraphics[width=0.32\textwidth]{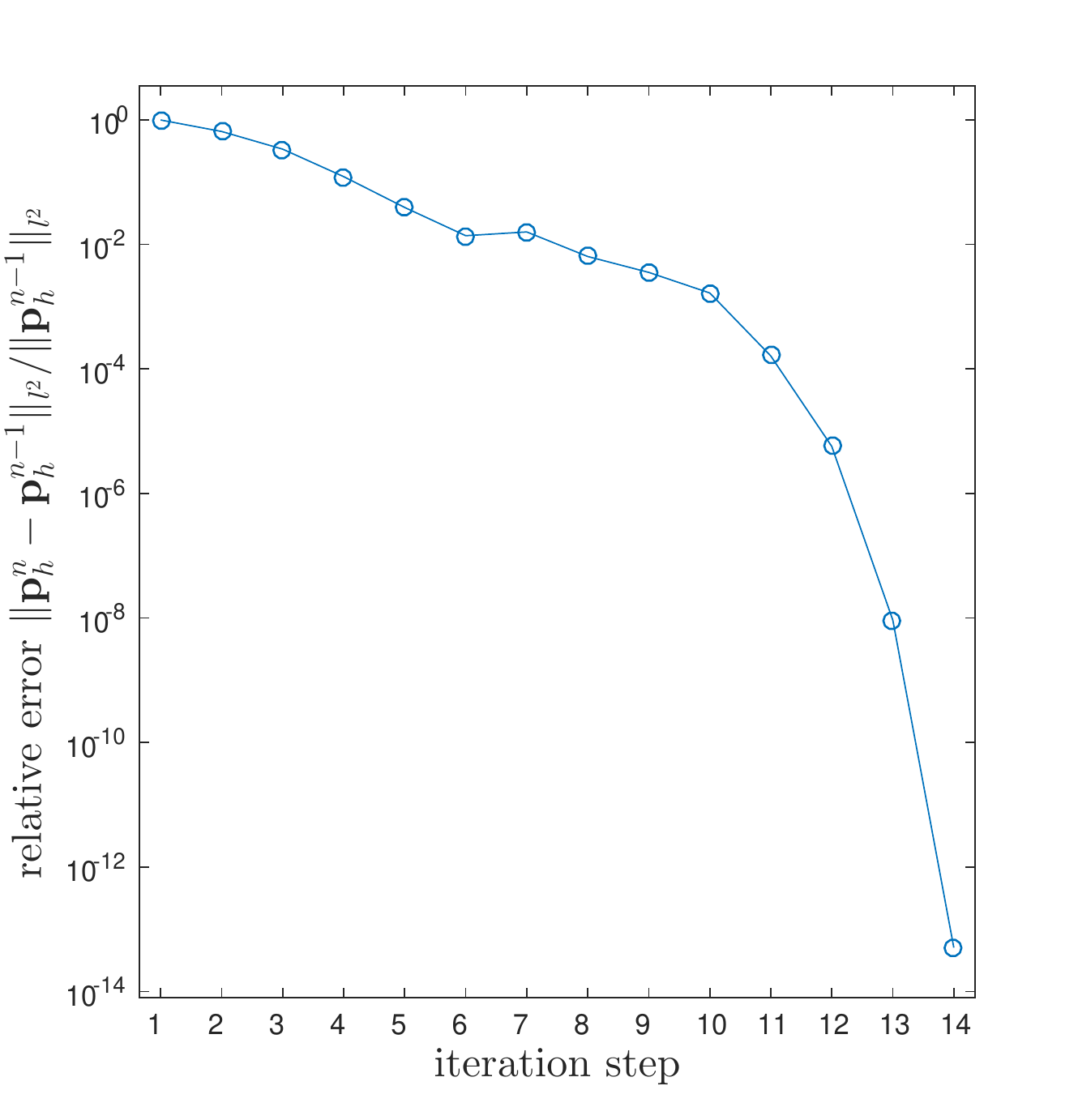}
  \caption{Convergence history of relative error with $m = 1$ (left) /
  $m = 2$ (middle) / $m=3$ (right).}
  \label{fig:ex2history}
\end{figure}

\begin{figure}
  \centering
  \includegraphics[width=0.32\textwidth]{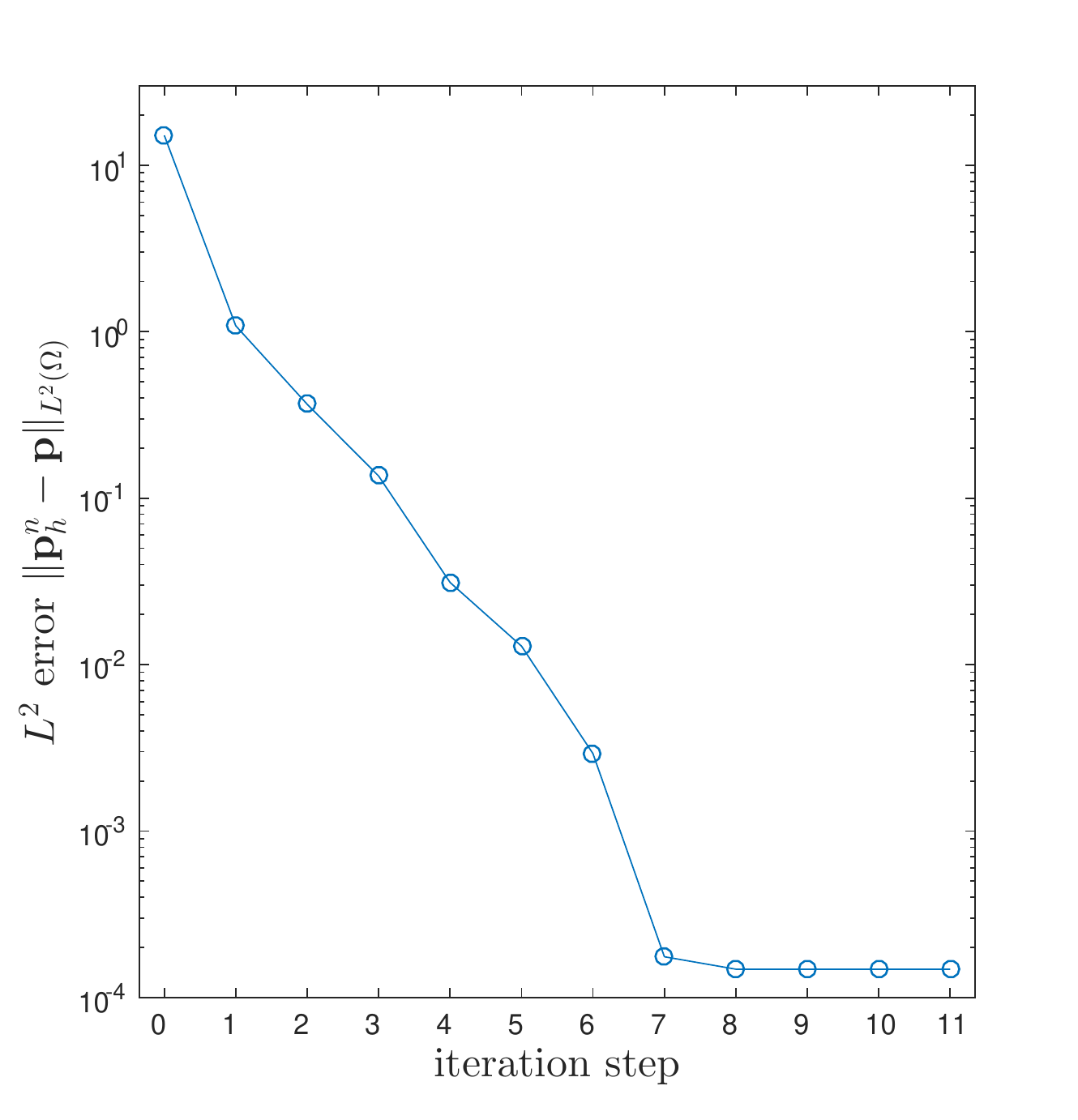}
  \hspace{0pt}
  \includegraphics[width=0.32\textwidth]{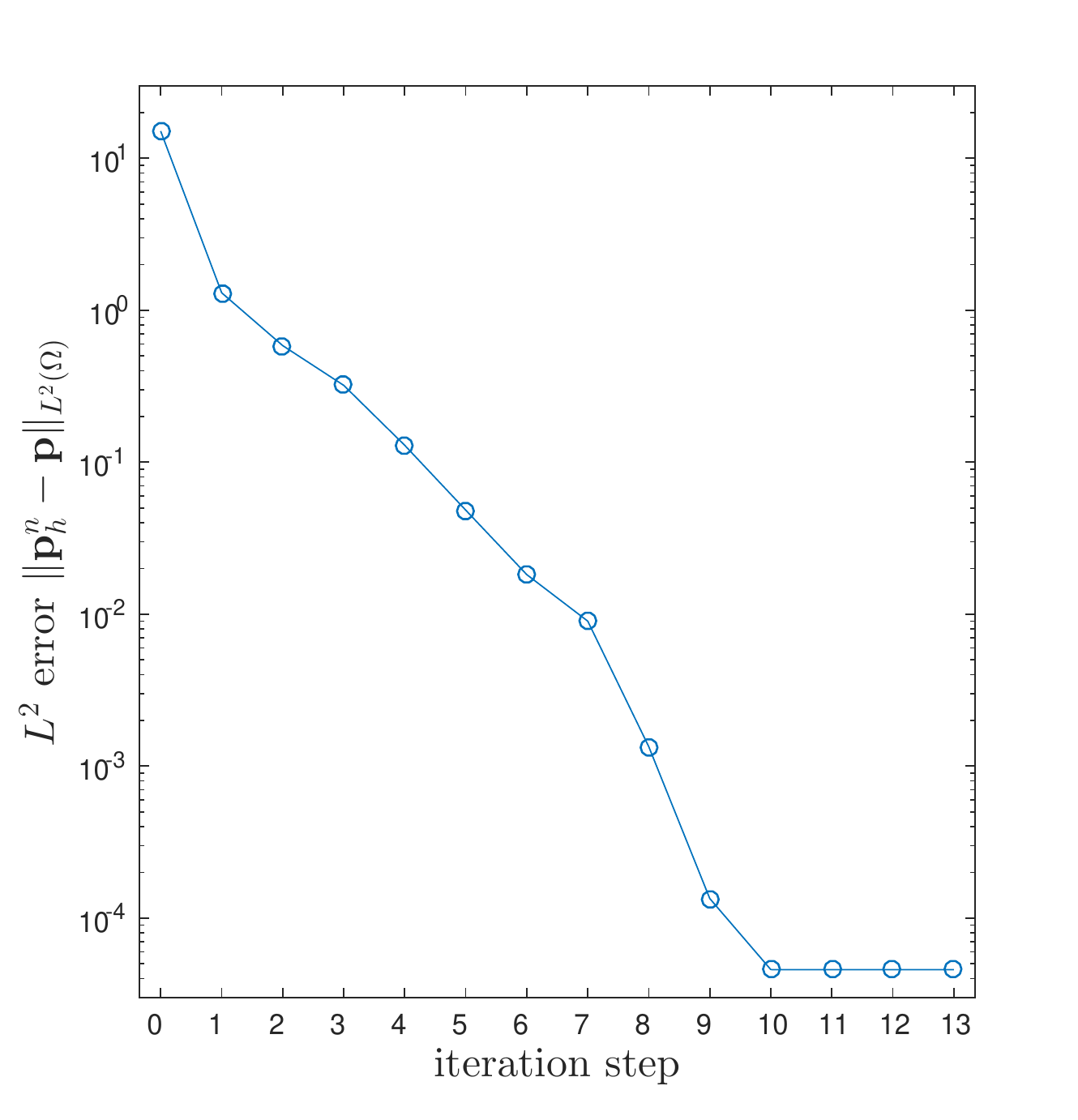}
  \hspace{0pt}
  \includegraphics[width=0.32\textwidth]{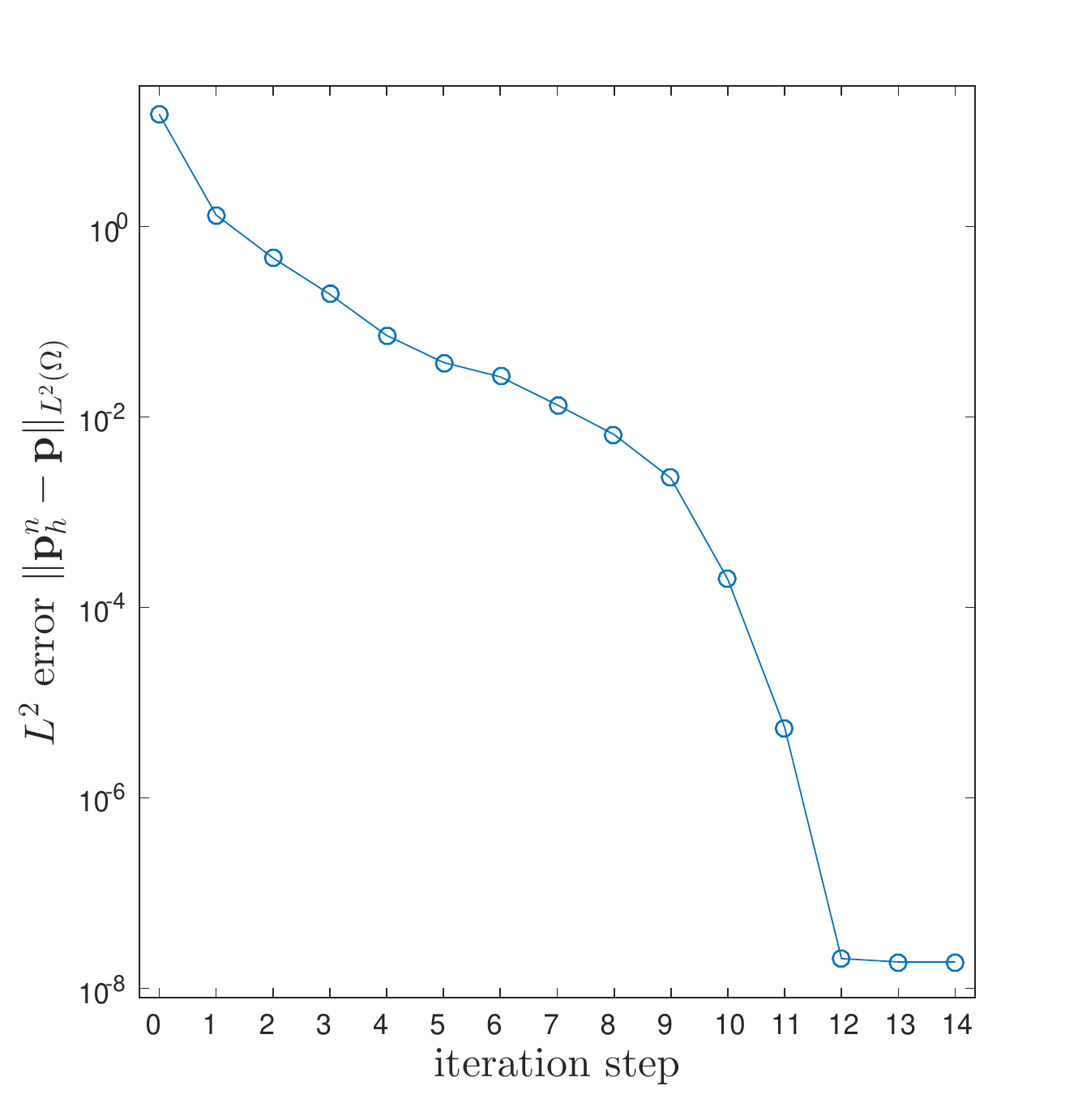}
  \caption{Convergence history of $L^2$ error with $m = 1$ (left) / $m
  = 2$ (middle) / $m=3$ (right).}
  \label{fig:ex2l2history}
\end{figure}

Moreover, we test our method with more initial values that are in a
wider range. We multiply the initial value $u^0(x, y)$ by a
coefficient $\alpha$ and we let $\alpha u^0(x, y)$ be the initial
value. Here $\alpha$ is taken as $\alpha = 0.01$, $0.1$, $1$,
$10$, $100$ and these initial values are from very close to zero to  
very far from the exact solution. We list the errors $\pnorm{ \bm{p} -
\bm{p}_h}$ and the number of iterations on the mesh level $h = 1/40$
in Tab.~\ref{tab:ex2alpha}. The convergence history is presented in
Fig.~\ref{fig:ex2alphahistory} and we observe that for all initial
values the convergence would speed up when the iteration solution gets
close to the exact solution. The convergence speed to the nonlinear
iteration may be related to the piecewise convexity of the numerical
solution and we will discuss in detail in next example. Different from
the classical Newton iteration, our method does not require a
sufficiently good initial guess and demonstrates a very wide range of
the convergence to the nonlinear system. In Section
\ref{sec:comparison}, we illustrate that this compelling feature is
due to the use of the reconstructed space.

\begin{table}
  \renewcommand{\arraystretch}{1.2}
  \centering
  \begin{tabular}{c|c|c|c|c|c|c}
    \hline\hline
    \multicolumn{2}{r|}{$\alpha = \quad$} & $0.01$ & $0.1$ & $1$ & $10$ & $100$ \\
    \hline 
    \multirow{2}{1.2cm}{$m=1$} & \# iter & 10 & 8 & 10 & 8 & 13  \\
    \cline{2-7}
    & $\pnorm{\bm{p} - \bm{p}_h}$ & 3.913e-2 &  3.913e-2 & 3.913e-2
    &3.913e-2 &3.913e-2 \\
    \hline 
    \multirow{2}{1.2cm}{$m=2$} & \# iter & 11 & 8 & 11 & 7 & 13  \\
    \cline{2-7}
    & $\pnorm{\bm{p} - \bm{p}_h}$ & 9.837e-4 &  9.837e-4 & 9.837e-4
    &9.837e-4 &9.837e-4 \\
    \hline 
    \multirow{2}{1.2cm}{$m=3$} & \# iter & 12 & 9 & 11 & 8 & 14  \\
    \cline{2-7}
    & $\pnorm{\bm{p} - \bm{p}_h}$ & 1.845e-5 &  1.845e-5 & 1.845e-5
    &1.845e-5 &1.845e-5 \\
    \hline\hline
  \end{tabular}
  \caption{Iteration steps and numerical errors for the initial values
  with different $\alpha$. }
  \label{tab:ex2alpha}
\end{table}

\begin{figure}
  \centering
  \includegraphics[width=0.32\textwidth]{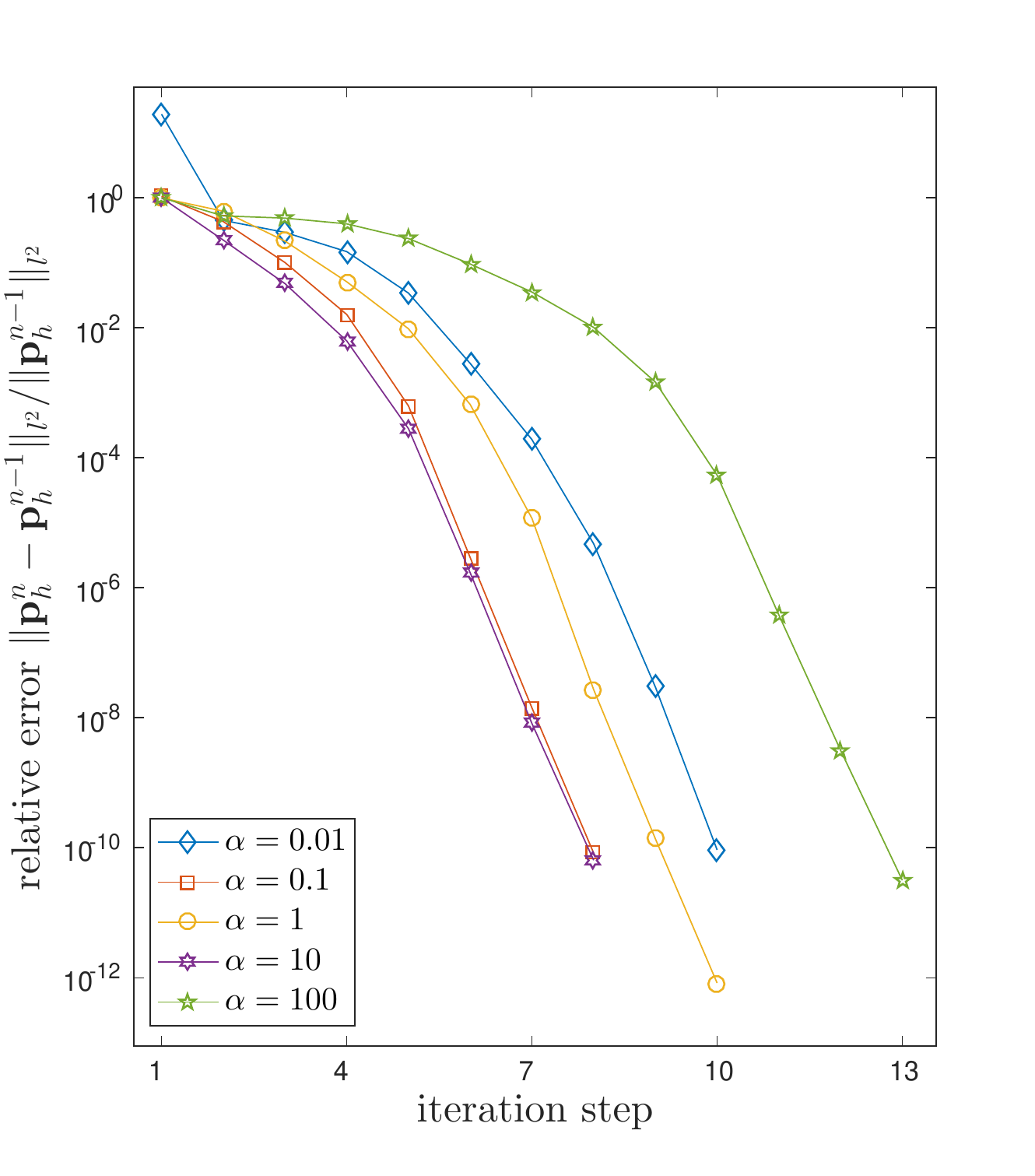}
  \hspace{0pt}
  \includegraphics[width=0.32\textwidth]{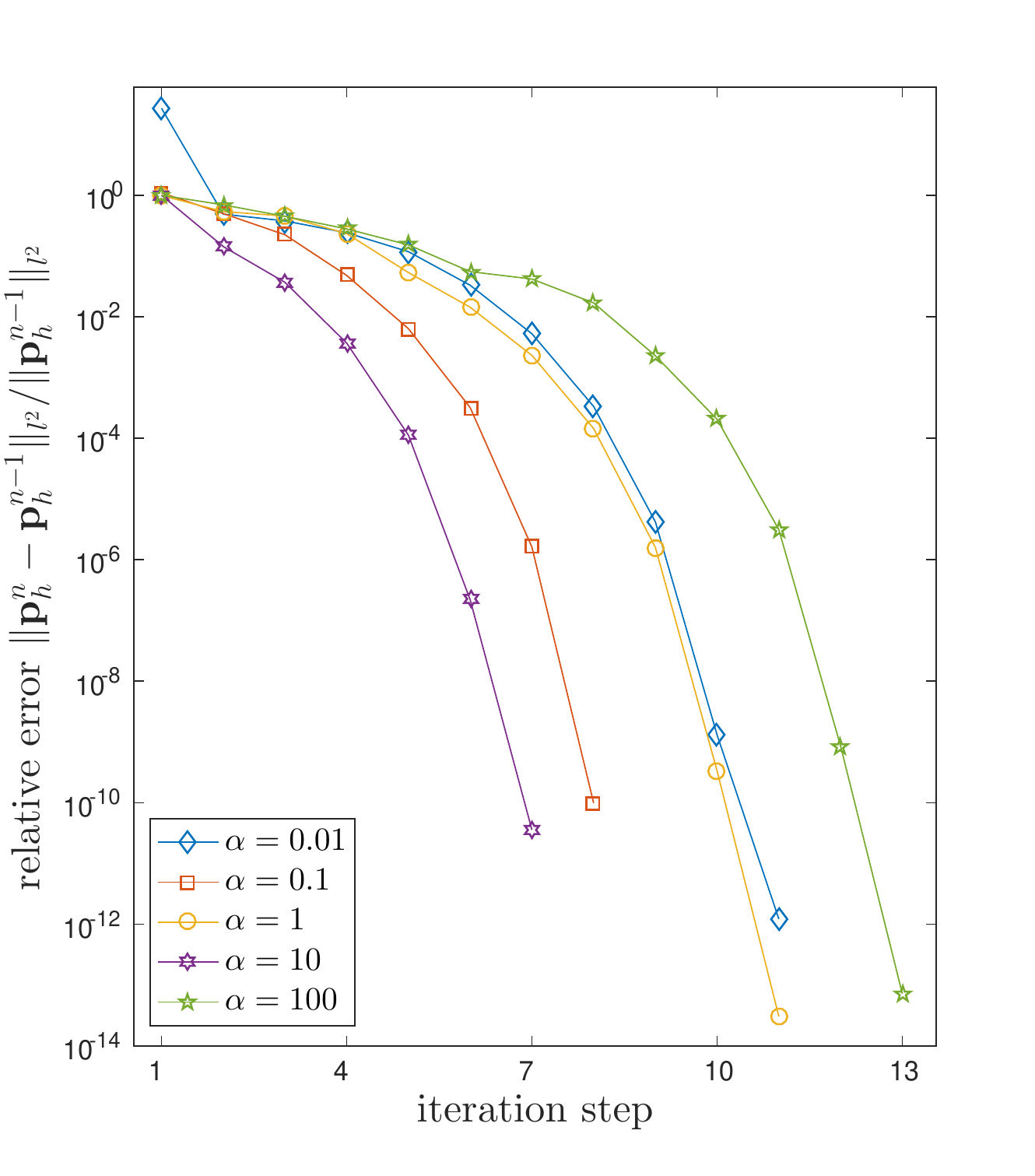}
  \hspace{0pt}
  \includegraphics[width=0.32\textwidth]{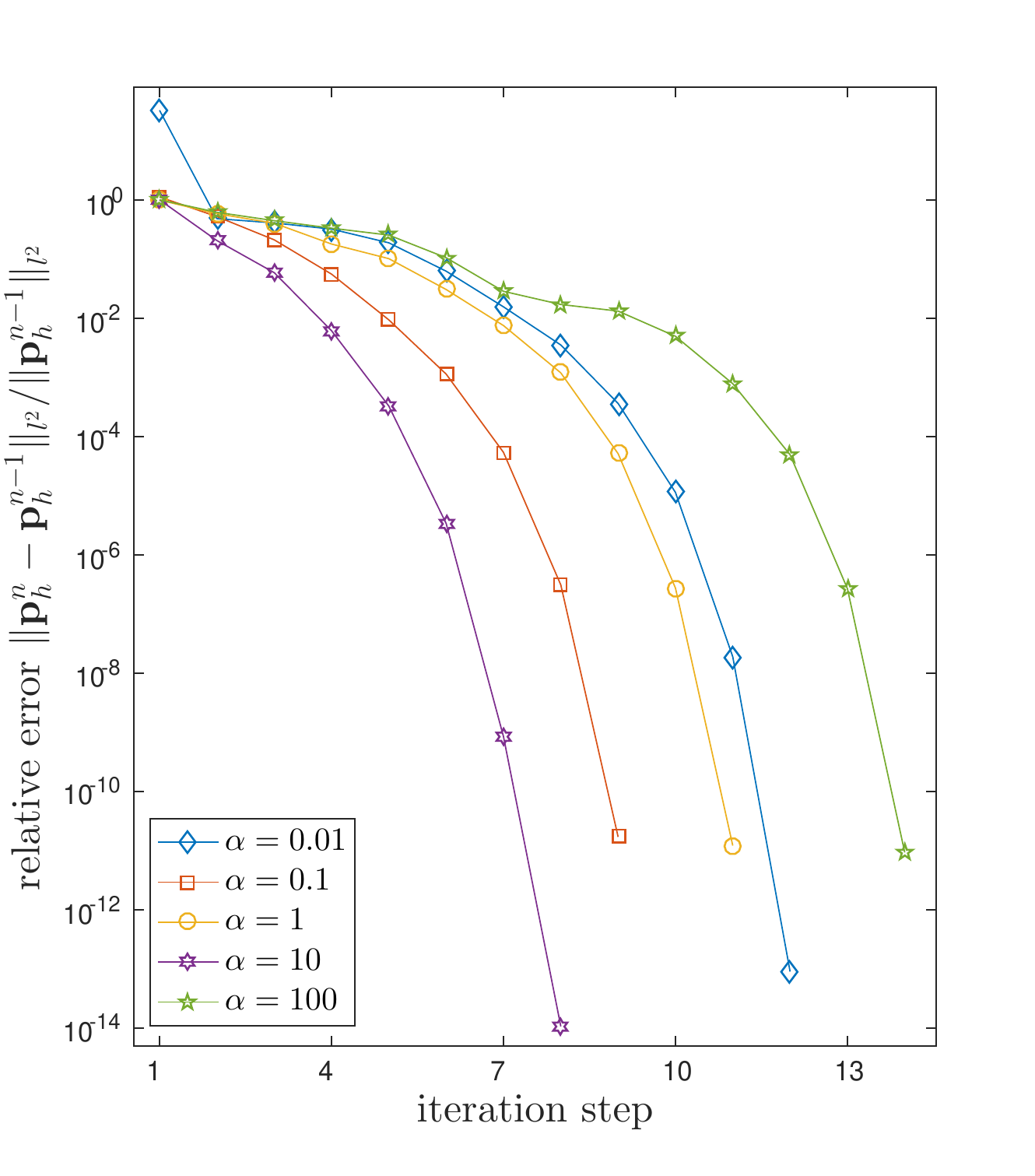}
  \caption{Convergence history of relative error with $m = 1$ (left) /
  $m = 2$ (middle) / $m=3$ (right).}
  \label{fig:ex2alphahistory}
\end{figure}

\paragraph{\textbf{Example 3:}} In this example, we consider the
case with exact solution \cite{Davydov2013numerical,
Westphal2019newton}
\begin{displaymath}
  u(x, y) = -\sqrt{R^2 - x^2 - y^2}
\end{displaymath}
with $R^2 = 3$. The data function $f$ and $g$ are taken accordingly.
The function of the initial guess is taken as the smooth convex
function
\begin{displaymath}
  u^0(x, y) = 5x^2 + 5y^2 + \sin(\pi x) \sin(\pi y).
\end{displaymath}
The numerical errors and the number of nonlinear iterations are
gathered in Tab.~\ref{tab:example3}.  In this test, we want to
emphasize the numerical convexity of the numerical solution. The
classical solution to the elliptic \MA equation must be strictly
convex on the whole domain and the convexity restriction is important
for the numerical scheme to capture the classical solution.  In this
test, we present the numerical convexity of the numerical solution at
each nonlinear step. At the iteration step $n$, we count the element
on which the numerical gradient $\nabla \bm{p}_h^n$ is not strictly
convex and such an element is called Non-Convex element. We present
the results at mesh level $h = 1/40$ in Fig.~\ref{fig:numnospd}, which
summarizes the number of Non-Convex elements and the ratio between the
number of Non-Convex elements and the number of all elements in
partition at each iteration. We observe that at the first two
steps, there are about $25\%$ Non-Convex elements since the initial
guess is very far from the exact solution.  In the first nonlinear
steps, the number of
Non-Convex elements decreases rapidly and after at most five steps the
numerical solution on all elements is strictly convex. We also report
the history of the relative error $  { \| \bm{p}_h^{n} - \bm{p}_h^{n -
1} \|_{l^2} } / { \| \bm{p}_h^{n - 1} \|_{l^2} }$  and it can be seen
that the convergence speeds up after the numerical solution becomes
strictly convex on any element.  For $m=1$, $4$ steps are required to
lead the numerical solution to be piecewise strictly convex. Hence,
the convergence speeds up $1$ steps in advance compared to the cases
$m=2, 3$, which meets our expectation. In
Fig.~\ref{fig:ex3spdhistory}, we mark Non-Convex elements at iteration
step $2$, $3$, $4$ for the case $m=3$ and clearly the Non-Convex
elements disappear rapidly. 

\begin{table}
  \renewcommand{\arraystretch}{1.3}
  \centering
  \scalebox{0.88}{
  \begin{tabular}{c|r|c|c|c|c|c|c|c|c|c}
    \hline\hline
    $m$ & $h~~~$ & $ \pnorm{\bm{p} - \bm{p}_h} $ & order & $ \| \bm{p} -
    \bm{p}_h\|_{L^2(\Omega)}$  & order & $ \unorm{u - u_h} $
    & order & $ \| u - u_h\|_{L^2(\Omega)}$ & order & \# iter
    \\
    \hline 
    \multirow{4}{0.6cm}{1} & $1/10$ & 
    5.322e-2 & -    & 3.659e-3 & - & 
    2.467e-2 & -    & 6.307e-4 & - & 9 \\
    \cline{2-11}
    & $1/20$ &
    2.752e-2 & 0.95 & 1.372e-3 & 1.41 & 
    1.237e-2 & 0.99 & 2.275e-4 & 1.47 & 9 \\
    \cline{2-11}
    & $1/40$ &
    1.296e-2 & 1.08 & 3.674e-4 & 1.91 & 
    6.177e-3 & 1.00 & 6.038e-5 & 1.92 & 9 \\
    \cline{2-11}
    & $1/80$ &
    6.246e-3 & 1.03 & 9.208e-5 & 2.00 & 
    3.087e-3 & 1.00 & 1.420e-5 & 2.09 & 9 \\
    \hline
    \multirow{4}{0.6cm}{2} & $1/10$ & 
    8.991e-3 & -    & 7.772e-4 & - & 
    2.972e-3 & -    & 1.595e-4 & - & 8 \\
    \cline{2-11}
    & $1/20$ &
    2.078e-3 & 2.11 & 1.929e-4 & 2.01 & 
    8.039e-4 & 1.89 & 3.972e-5 & 2.00 & 9 \\
    \cline{2-11}
    & $1/40$ &
    5.257e-4 & 1.98 & 5.207e-5 & 1.90 & 
    2.071e-4 & 1.96 & 1.045e-5 & 1.93 & 10 \\
    \cline{2-11}
    & $1/80$ &
    1.284e-4 & 2.03 & 1.339e-5 & 1.96 & 
    5.231e-5 & 1.99 & 2.656e-6 & 1.98 & 10 \\
    \hline
    \multirow{4}{0.6cm}{3} & $1/10$ & 
    2.692e-3 & -    & 1.088e-4 & - & 
    2.526e-5 & -    & 1.188e-5 & - & 8 \\
    \cline{2-11}
    & $1/20$ &
    2.466e-4 & 3.44 & 4.882e-6 & 4.47 & 
    2.582e-6 & 3.29 & 4.522e-7 & 4.72 & 9 \\
    \cline{2-11}
    & $1/40$ &
    2.813e-5 & 3.13 & 2.812e-7 & 4.11 & 
    2.832e-7 & 3.19 & 2.581e-8 & 4.13 & 10 \\
    \cline{2-11}
    & $1/80$ &
    3.272e-6 & 3.10 & 1.809e-8 & 3.96 & 
    3.492e-8 & 3.03 & 1.615e-9 & 3.99 & 11 \\
    \hline\hline
  \end{tabular}}
  \caption{Convergence of the Example 3.}
  \label{tab:example3}
\end{table}

\begin{figure}
  \centering
  \includegraphics[width=0.3\textwidth]{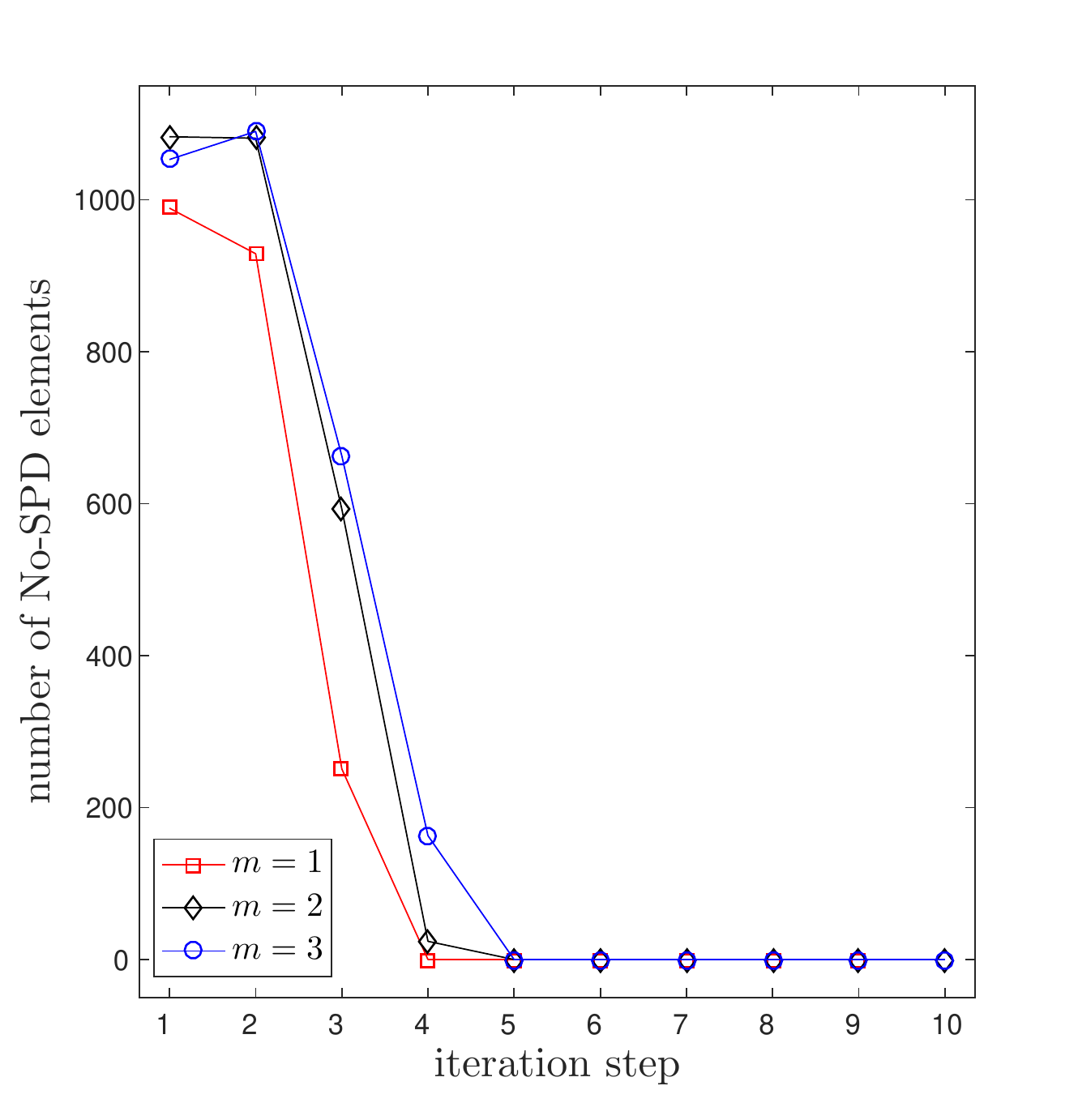}
  \includegraphics[width=0.3\textwidth]{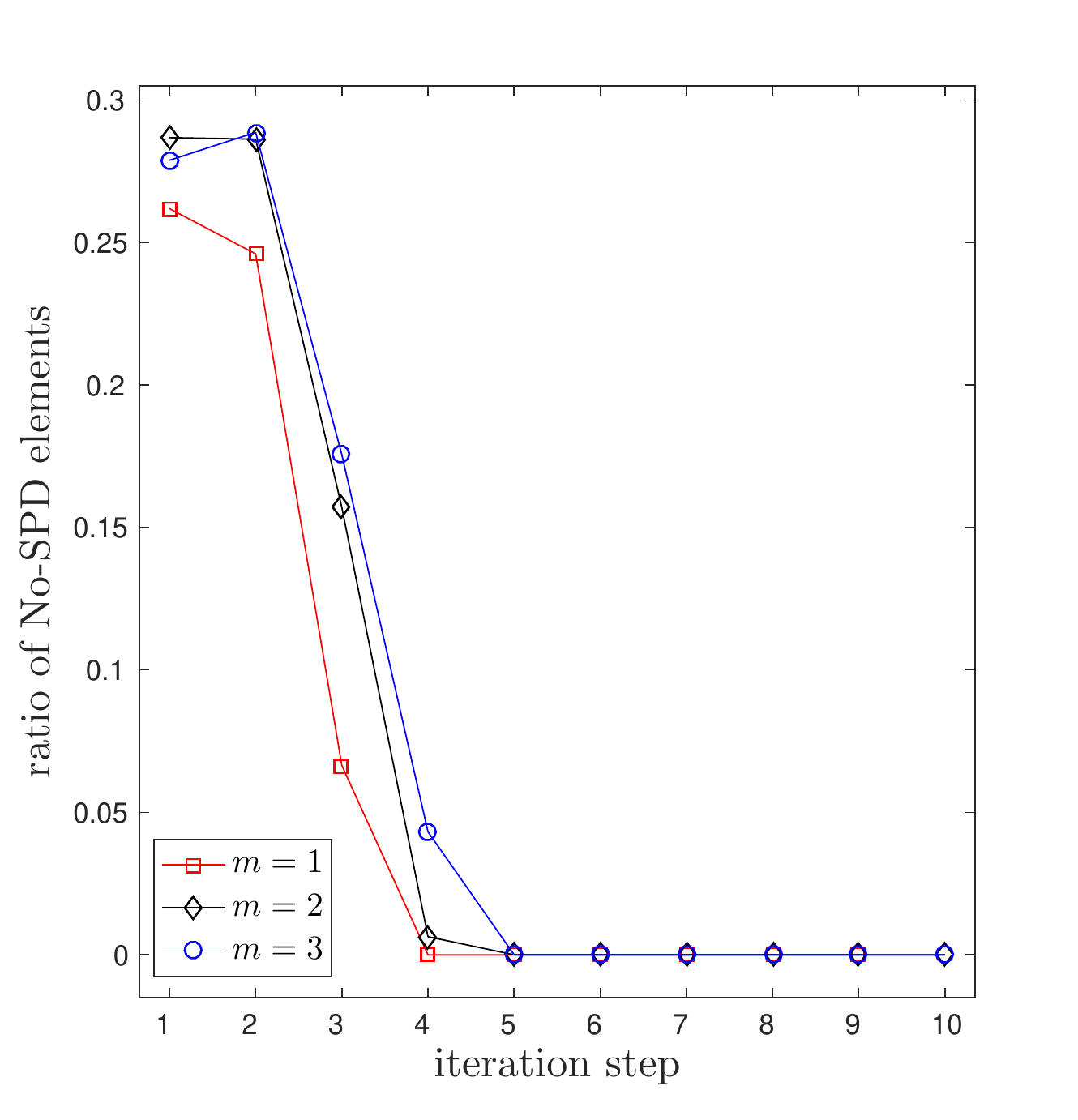}
  \includegraphics[width=0.3\textwidth]{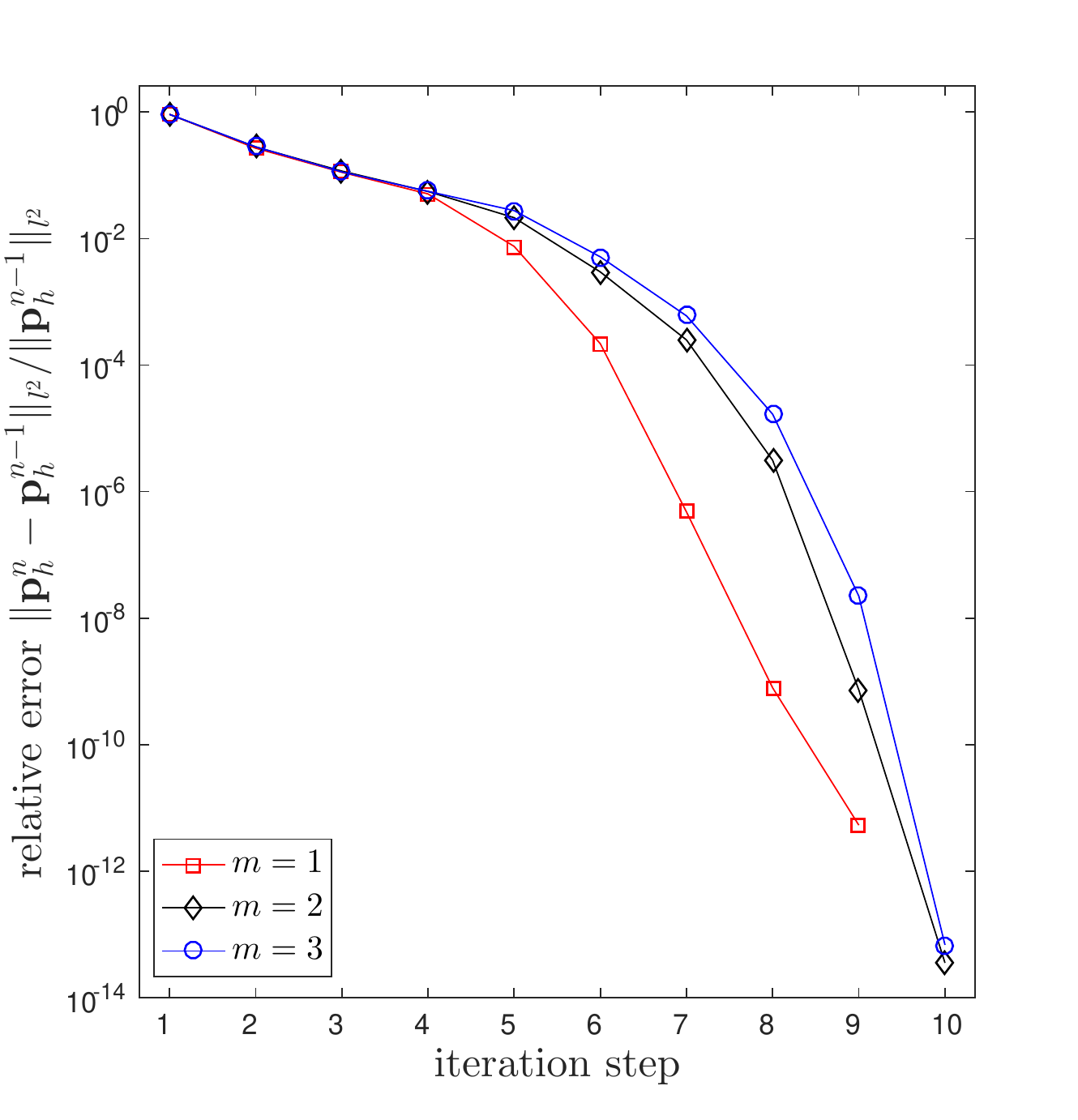}
  \caption{Number of Non-Convex elements (left) / ratio of Non-Convex
  elements to all elements (middle) / convergence history of relative
  error (right). }
  \label{fig:numnospd}
\end{figure}

\begin{figure}
  \centering
  \includegraphics[width=0.28\textwidth]{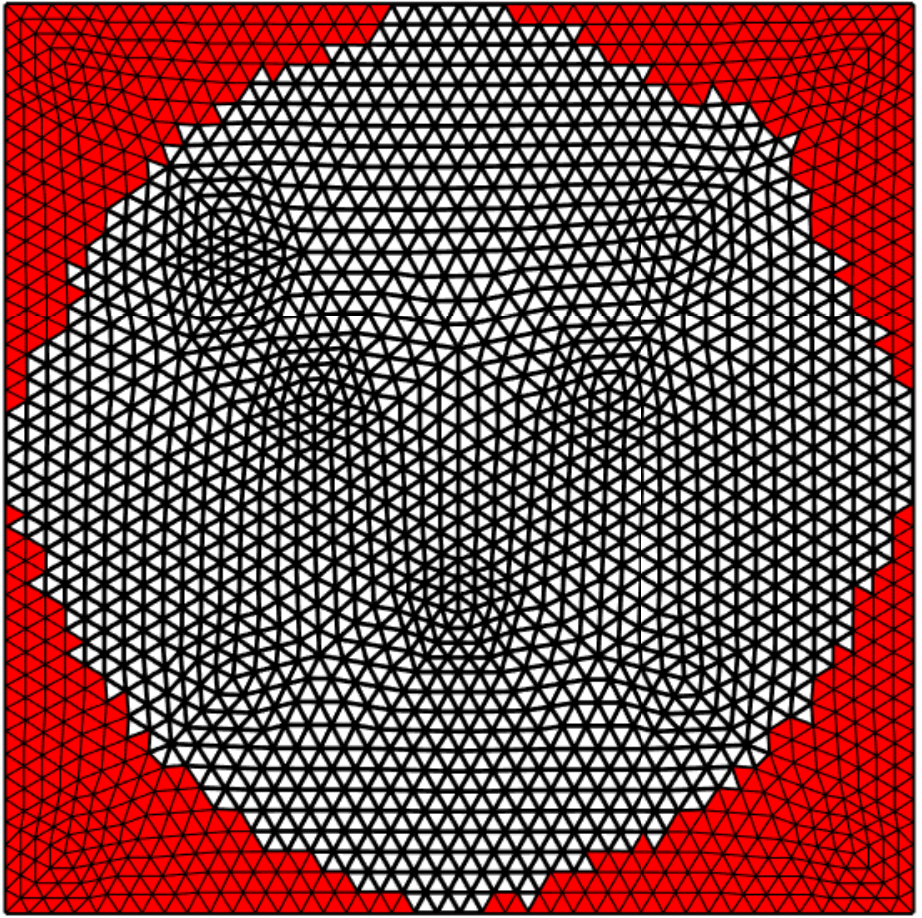}
  \hspace{10pt}
  \includegraphics[width=0.28\textwidth]{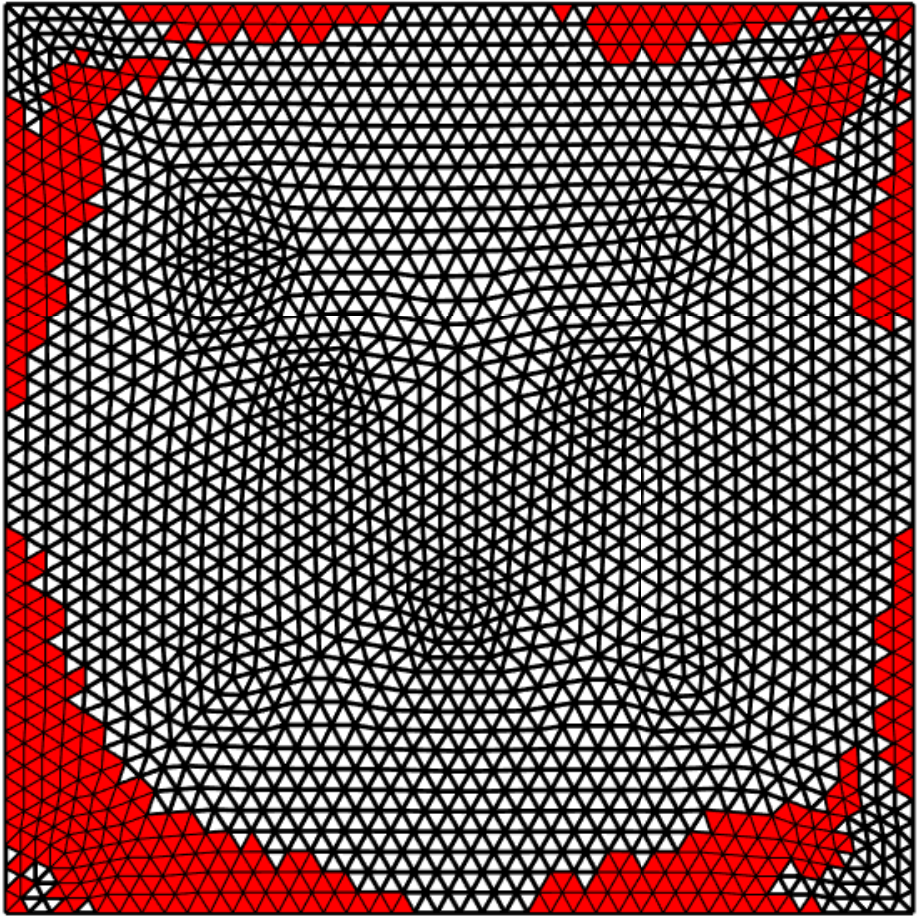}
  \hspace{10pt}
  \includegraphics[width=0.28\textwidth]{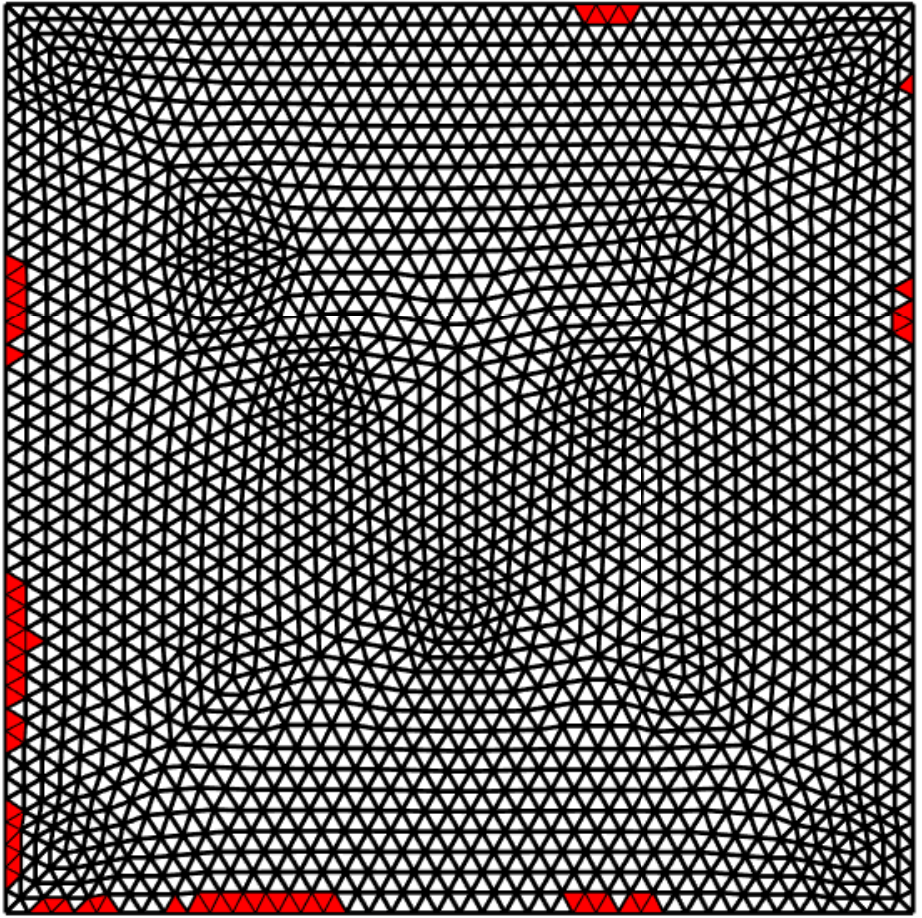}
  \caption{Non-Convex elements (red elements) in partition at step $2$
  (left) / at step $3$ (middle) / at step $4$ (right).}
  \label{fig:ex3spdhistory}
\end{figure}

\paragraph{\textbf{Example 4:}} In this test, we consider the
problem that the exact solution is given by
\begin{displaymath}
  u(x, y) = (x^2 + y^2)^{9/2} + 0.5(x^2 + y^2),
\end{displaymath}
and the source term $f$ and the boundary term $g$ are taken
accordingly. This function $u \in H^{5.5}(\Omega)$ is strictly convex
over the domain $(0, 1)^2$. The initial values for all previous
numerical tests are strictly convex or are the solution to a closely
related problem. In this test, we adopt a smooth but non-convex
function to start the iteration to show the robustness of our method.
We take the function
\begin{displaymath}
  u^0(x, y) = x^2 - y^2
\end{displaymath}
to be the initial guess for the nonlinear system. The numerical errors
are listed in Tab.~\ref{tab:example4} and the convergence rates are
consistent with our previous numerical results. With the non-convex
initial guess, more nonlinear iterations are required for converging
to the exact solution.  In Fig.~\ref{fig:ex4his}, we plot the history
of the relative error ${ \| \bm{p}_h^{n} - \bm{p}_h^{n - 1} \|_{l^2} }
/ { \| \bm{p}_h^{n - 1} \|_{l^2} }$ and the ratio of Non-Convex
elements on the mesh level $h = 1/40$. The numerical results clearly
confirm our expectation that the convergence may be divided into two
parts. The first process is to make the numerical solution be
piecewise convex and during this process the decrease of the relative
error $ { \| \bm{p}_h^{n} - \bm{p}_h^{n - 1} \|_{l^2} } / { \|
\bm{p}_h^{n - 1} \|_{l^2} }$ might be slow. In this process, our
method can automatically correct the Non-Convex elements and the
number of Non-Convex elements is decreasing to zero.  The second
process is iterating with the piecewise convex numerical solution.
In this process, the residual drops very rapidly and after very few
nonlinear iterations the stop criterion is met. For this non-convex
initial guess, the first process requires more iterations to ensure
the piecewise convexity and the second process only involves at most
five steps for all accuracy $m$.  In addition, in
Fig.~\ref{fig:ex4spdhistory} we depict the Non-Convex elements in
partition at three iteration steps with accuracy $m=3$. The numerical
solution produced by our method seems to be automatically piecewise
convex, which meets the convex constraint to the \MA equation. This
property is still due to the use of the reconstructed space, as
illustrated in Section \ref{sec:comparison}.

\begin{table}
  \renewcommand{\arraystretch}{1.3}
  \centering
  \scalebox{0.88}{
    \begin{tabular}{c|r|c|c|c|c|c|c|c|c|c}
    \hline\hline
    $m$ & $h~~~$ & $ \pnorm{\bm{p} - \bm{p}_h} $ & order & $ \| \bm{p} -
    \bm{p}_h\|_{L^2(\Omega)}$  & order & $ \unorm{u - u_h} $
    & order & $ \| u - u_h\|_{L^2(\Omega)}$ & order & \# iter
    \\
    \hline 
    \multirow{4}{0.6cm}{1} & $1/10$ & 
    9.529e-1 & -    & 4.739e-2 & - & 
    3.200e-1 & -    & 7.366e-3 & - &    19 \\
    \cline{2-11}
    & $1/20$ &
    4.623e-1 & 1.04 & 1.301e-2 & 1.86 & 
    1.602e-1 & 1.00 & 1.973e-3 & 1.90 & 24 \\
    \cline{2-11}
    & $1/40$ &
    2.183e-1 & 1.08 & 3.343e-3 & 1.96 & 
    8.013e-2 & 1.00 & 4.787e-4 & 2.03 & 32 \\
    \cline{2-11}
    & $1/80$ &
    1.036e-2 & 1.07 & 8.601e-4 & 1.96 & 
    4.008e-2 & 1.00 & 1.179e-4 & 2.02 & 55 \\
    \hline
    \multirow{4}{0.6cm}{2} & $1/10$ & 
    1.563e-2 & -    & 7.895e-2 & - & 
    1.816e-2 & -    & 3.358e-3 & - &    21 \\
    \cline{2-11}
    & $1/20$ &
    4.067e-3 & 1.94 & 1.906e-2 & 2.05 & 
    4.691e-3 & 1.95 & 8.783e-4 & 1.93 & 29 \\
    \cline{2-11}
    & $1/40$ &
    1.028e-3 & 1.98 & 4.779e-3 & 2.00 & 
    1.184e-3 & 1.99 & 2.222e-4 & 1.98 & 24 \\
    \cline{2-11}
    & $1/80$ &
    2.578e-4 & 1.99 & 1.197e-3 & 2.00 & 
    2.969e-4 & 2.00 & 5.573e-5 & 2.00 & 33 \\
    \hline
    \multirow{4}{0.6cm}{3} & $1/10$ & 
    2.668e-3 & -    & 1.253e-4 & - & 
    1.518e-4 & -    & 7.632e-6 & - &    21 \\
    \cline{2-11}
    & $1/20$ &
    2.928e-4 & 3.19 & 1.1083-5 & 3.50 & 
    1.773e-5 & 3.10 & 1.053e-6 & 2.86 & 24 \\
    \cline{2-11}
    & $1/40$ &
    3.387e-5 & 3.11 & 6.811e-7 & 4.02 & 
    2.005e-6 & 3.13 & 6.623e-8 & 3.99 & 38 \\
    \cline{2-11}
    & $1/80$ &
    4.058e-6 & 3.06 & 4.145e-8 & 4.03 & 
    2.446e-7 & 3.03 & 4.327e-9 & 3.95 & 33 \\
    \hline\hline
  \end{tabular}}
  \caption{Convergence of the Example 4.}
  \label{tab:example4}
\end{table}

\begin{figure}
  \centering
  \includegraphics[width=0.4\textwidth]{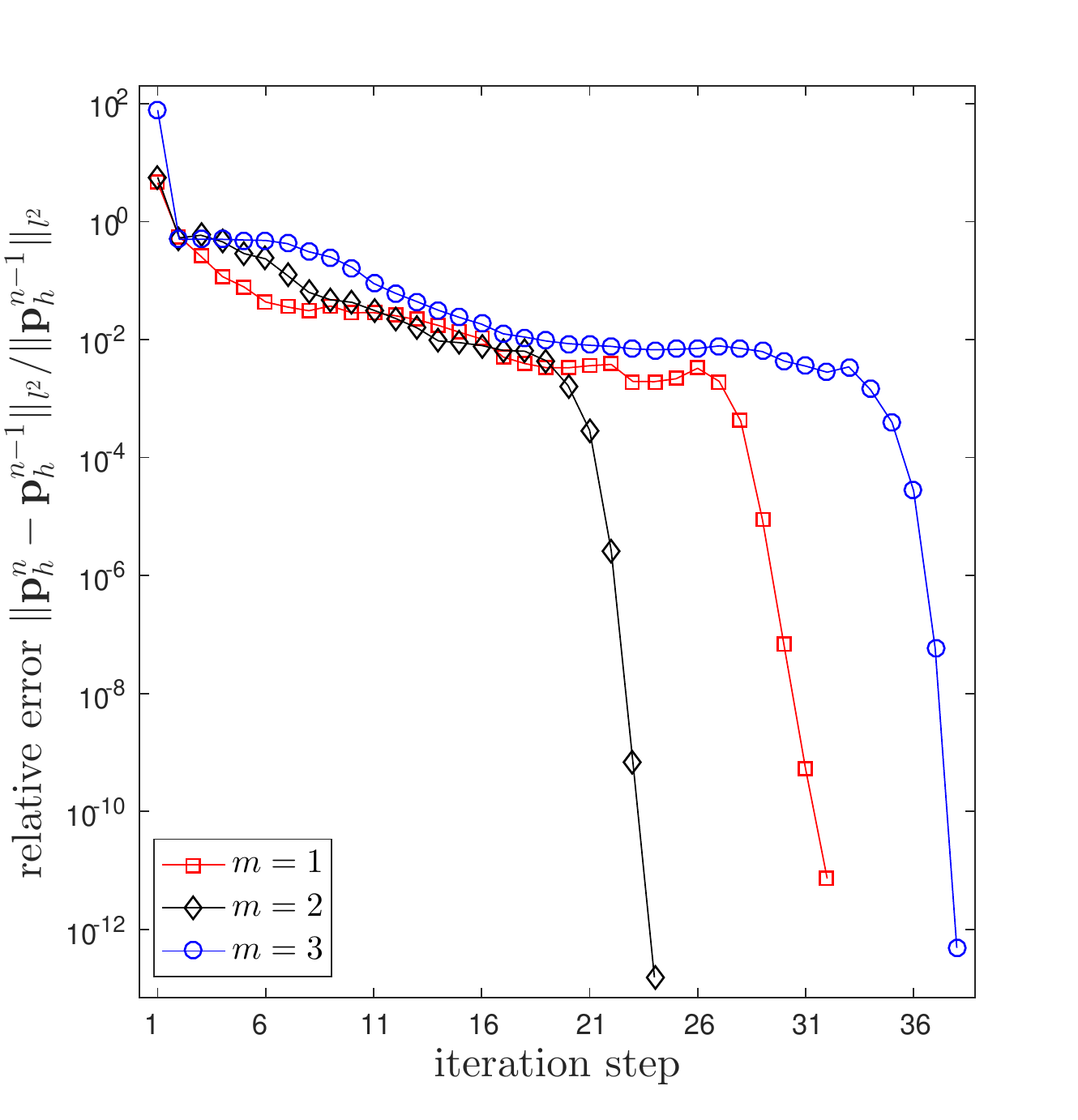}
  \includegraphics[width=0.4\textwidth]{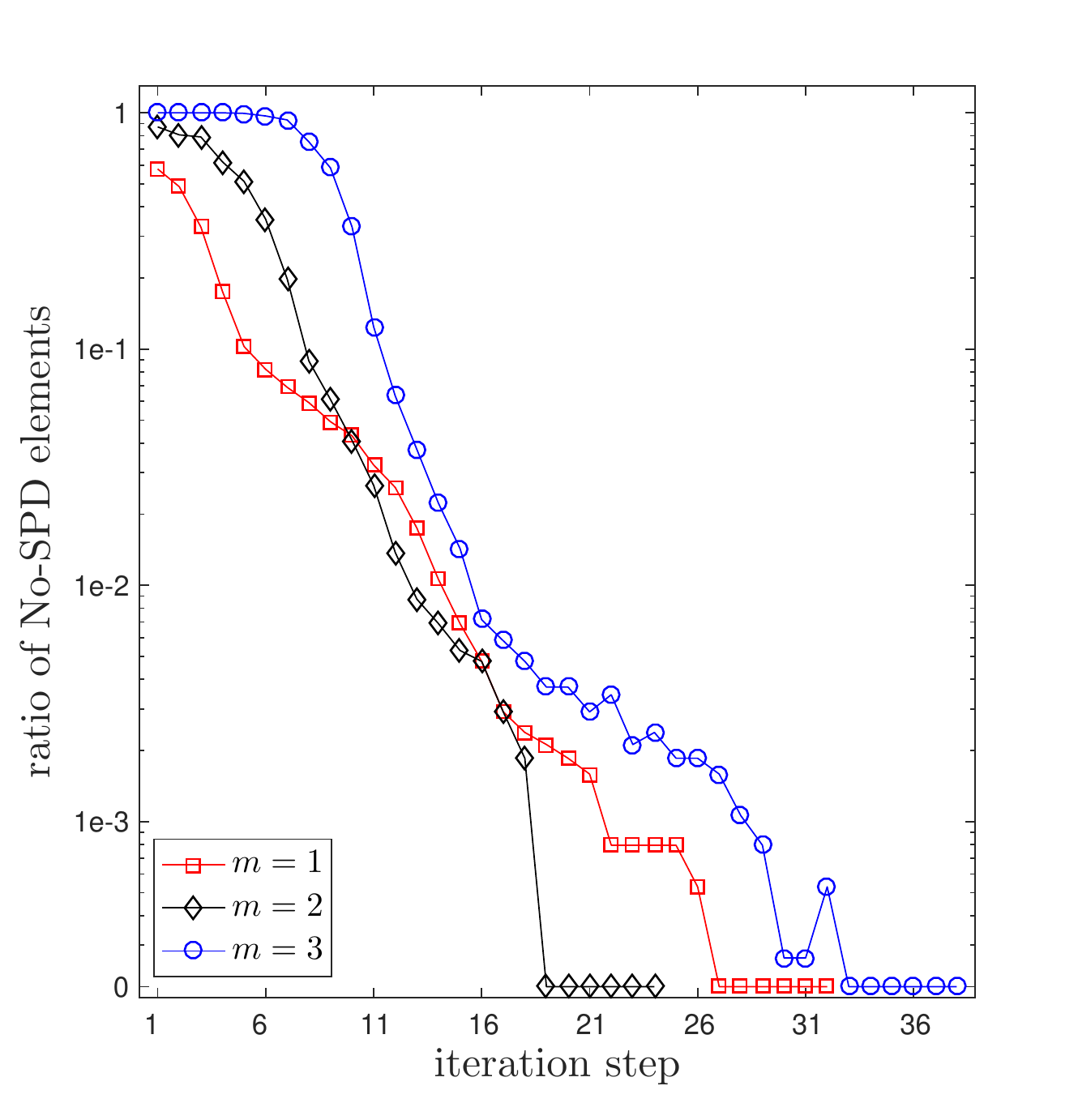}
  \caption{Ratio of Non-Convex elements to all elements (left) /
  convergence history of relative error (right). }
  \label{fig:ex4his}
\end{figure}

\begin{figure}
  \centering
  \includegraphics[width=0.28\textwidth]{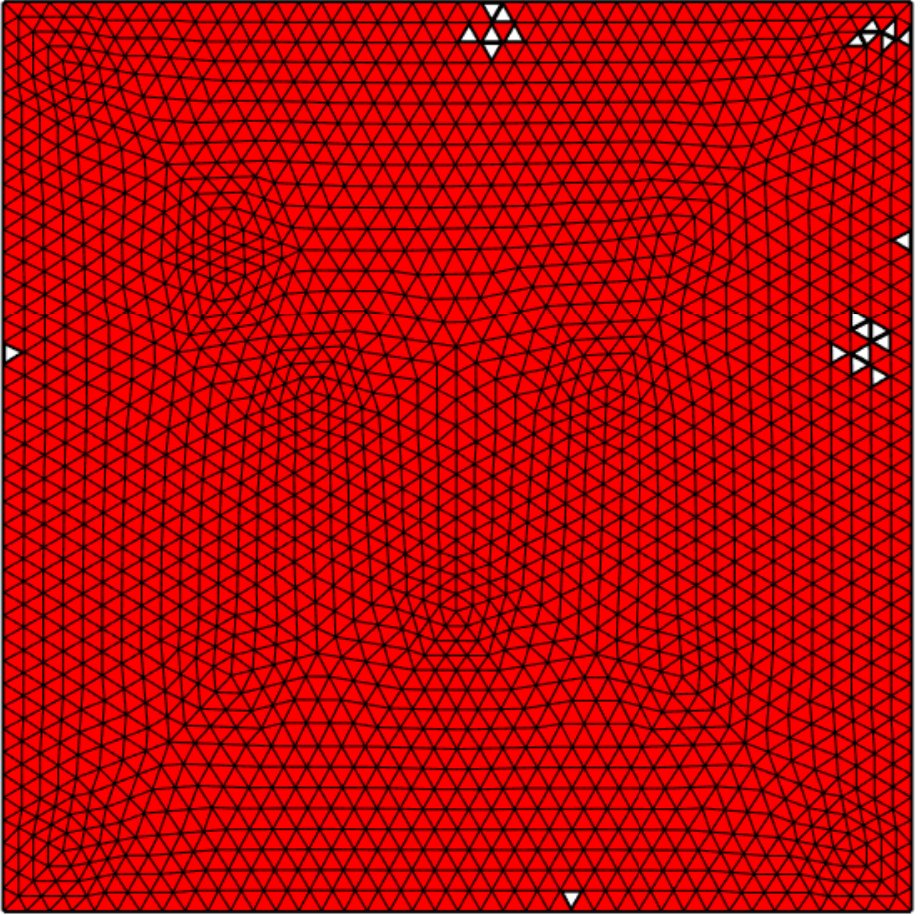}
  \hspace{10pt}
  \includegraphics[width=0.28\textwidth]{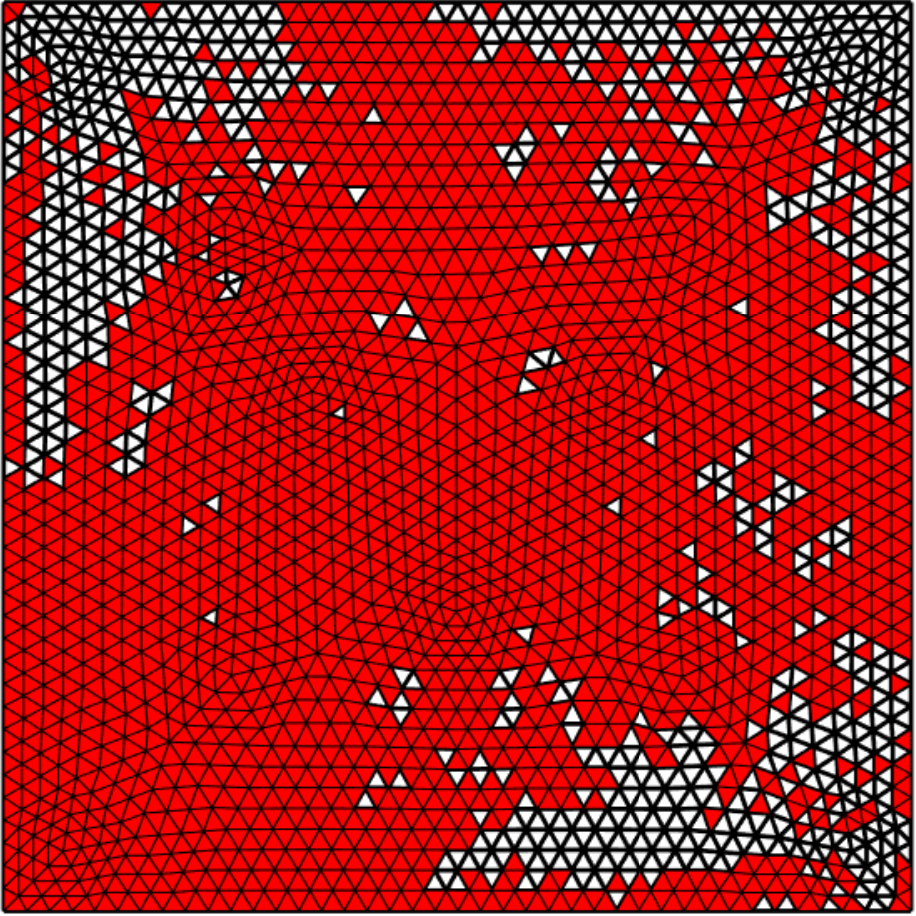}
  \hspace{10pt}
  \includegraphics[width=0.28\textwidth]{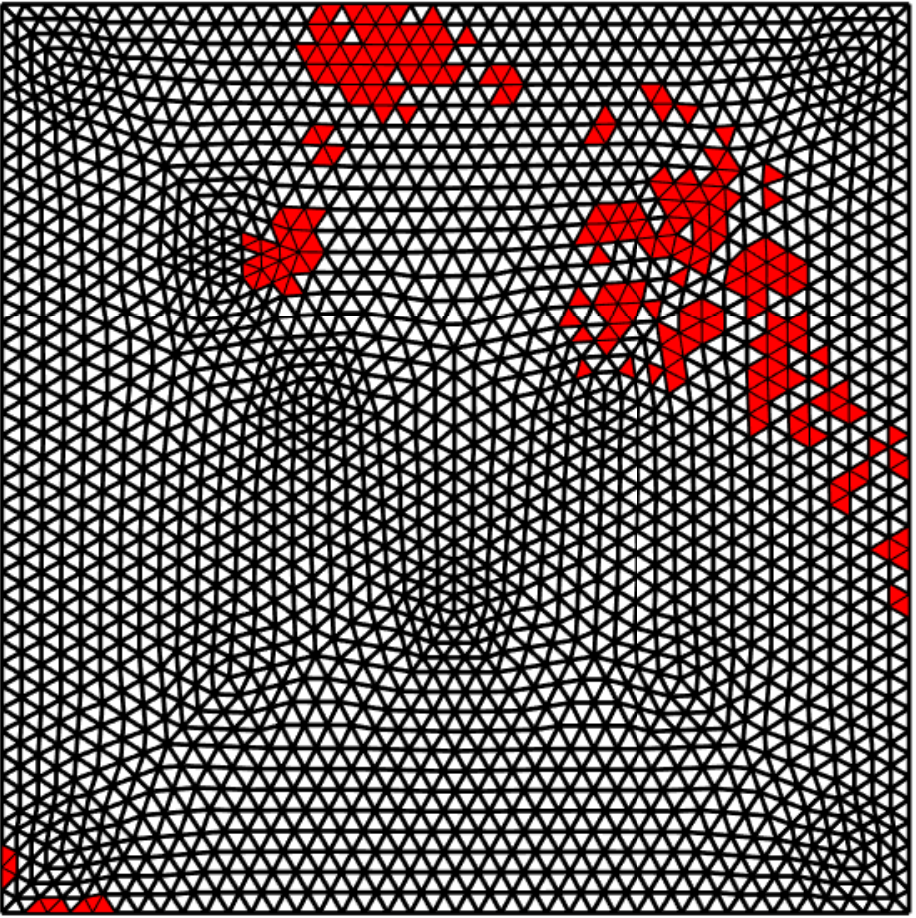}
  \caption{Non-Convex elements (red elements) in partition at step $5$
  (left) / at step $8$ (middle) / at step $12$ (right).}
  \label{fig:ex4spdhistory}
\end{figure}

\subsection{3D Examples}
In this section, we present the three-dimensional numerical
examples. In three dimensions, the computational domain $\Omega$ is
taken as the unit cube $\Omega = (0, 1)^3$ and the penalty parameter
$\eta$ is selected to be $60$. We adopt a series of tetrahedra meshes
with mesh size $h = 1/4$, $h = 1/8$, $h = 1/16$ and $h = 1/32$, see
Fig~\ref{fig:partition} and the values of $\# S(K)$ for different $m$
are given in Tab.~\ref{tab:numSK}. We use the approximation spaces
$\bmr{U}_h^m \times \wt{V}_h^m$ with $1 \leq m \leq 3$ to solve the
\MA equation. The stop criterion in nonlinear iteration is chosen to
be 
\begin{displaymath} 
  \frac{ \| \bm{p}_h^{n} - \bm{p}_h^{n - 1} \|_{l^2} }{
  \| \bm{p}_h^{n - 1} \|_{l^2} } < 10^{-8}. 
\end{displaymath}

\paragraph{\textbf{Example 5:}} In this test, we consider a
three-dimensional Monge-Amp\`ere equation. The exact solution is
defined as
\begin{displaymath}
  u(x, y, z) = -\sqrt{R^2 - \left( x^2 + y^2 + z^2 \right)},
\end{displaymath}
with $R^2 = 4$. The data functions $f$ and $g$ are given by 
\begin{displaymath}
  f(x, y, z) = \frac{R^2}{(R^2 - r^2)^{5/2}}, \quad g(x, y, z) =
  -\sqrt{R^2 - \left( x^2 + y^2 + z^2 \right)}.
\end{displaymath}
We select the function 
\begin{displaymath}
  u^0(x, y, z) = x^2 + y^2 + z^2
\end{displaymath}
to be the initial guess to start the nonlinear iteration. The
numerical errors and the number of the nonlinear iteration are
reported in Tab.~\ref{tab:example5}. It is observed that our method
retains the fast convergence in the nonlinear iteration and the
optimal convergence rates under the energy norms for the
three-dimensional problem. For the error under $L^2$ norm, the
odd/even situation is also been detected. Fig.~\ref{fig:ex5his}
shows the history of the relative error ${ \| \bm{p}_h^{n} -
\bm{p}_h^{n - 1} \|_{l^2} } / { \| \bm{p}_h^{n - 1} \|_{l^2} }$ for
all accuracy $1 \leq m \leq 3$ and the performance is consistent with
the results of the two-dimensional cases. In Fig.~\ref{fig:ex5slice},
we plot two slices of the computed solution with the accuracy $m=3$ at
the mesh level $h=1/32$. 

\begin{table}
  \renewcommand{\arraystretch}{1.3}
  \centering
  \scalebox{0.88}{
    \begin{tabular}{c|r|c|c|c|c|c|c|c|c|c}
    \hline\hline
    $m$ & $h~~~$ & $ \pnorm{\bm{p} - \bm{p}_h} $ & order & $ \| \bm{p} -
    \bm{p}_h\|_{L^2(\Omega)}$  & order & $ \unorm{u - u_h} $
    & order & $ \| u - u_h\|_{L^2(\Omega)}$ & order & \# iter
    \\
    \hline 
    \multirow{4}{0.6cm}{1} & $1/4$ & 
    9.385e-2 & -    & 1.046e-2 & - & 
    7.507e-2 & -    & 3.357e-3 & - &    8 \\
    \cline{2-11}
    & $1/8$ &
    5.448e-2 & 0.78 & 6.032e-3 & 0.79 & 
    3.878e-2 & 0.95 & 1.772e-3 & 0.92 & 8 \\
    \cline{2-11}
    & $1/16$ &
    2.936e-2 & 0.89 & 2.125e-3 & 1.51 & 
    1.969e-2 & 0.98 & 6.567e-4 & 1.43 & 9 \\
    \cline{2-11}
    & $1/32$ &
    1.460e-2 & 1.01 & 6.023e-4 & 1.82 & 
    9.830e-3 & 1.00 & 1.803e-4 & 1.86 & 9 \\
    \hline
    \multirow{4}{0.6cm}{2} & $1/4$ & 
    2.016e-2 & -    & 2.353e-3 & - & 
    1.273e-3 & -    & 1.656e-4 & - &    8 \\
    \cline{2-11}
    & $1/8$ &
    5.859e-3 & 1.63 & 6.967e-4 & 1.76 & 
    4.109e-4 & 1.78 & 6.782e-5 & 1.29 & 9 \\
    \cline{2-11}
    & $1/16$ &
    1.548e-3 & 1.95 & 1.843e-4 & 1.92 & 
    1.0663-4 & 1.92 & 1.779e-5 & 1.93 & 9 \\
    \cline{2-11}
    & $1/32$ &
    2.679e-5 & 1.99 & 4.746e-5 & 1.96 & 
    2.969e-4 & 2.00 & 4.391e-6 & 2.01 & 9 \\
    \hline
    \multirow{4}{0.6cm}{3} & $1/4$ & 
    8.613e-3 & -    & 5.325e-4 & - & 
    5.989e-4 & -    & 2.989e-5 & - &    9 \\
    \cline{2-11}
    & $1/8$ &
    1.220e-3 & 2.82 & 4.098e-5 & 3.86 & 
    8.323e-5 & 2.85 & 2.009e-6 & 3.89 & 8 \\
    \cline{2-11}
    & $1/16$ &
    1.448e-4 & 3.08 & 2.558e-6 & 4.00 & 
    1.023e-5 & 3.02 & 1.232e-7 & 4.02 & 8 \\
    \cline{2-11}
    & $1/32$ &
    1.812e-5 & 3.02 & 1.578e-8 & 4.02 & 
    1.290e-6 & 2.99 & 7.753e-9 & 3.99 & 8 \\
    \hline\hline
  \end{tabular}}
  \caption{Convergence of the Example 5.}
  \label{tab:example5}
\end{table}

\begin{figure}
  \centering
  \includegraphics[width=0.3\textwidth]{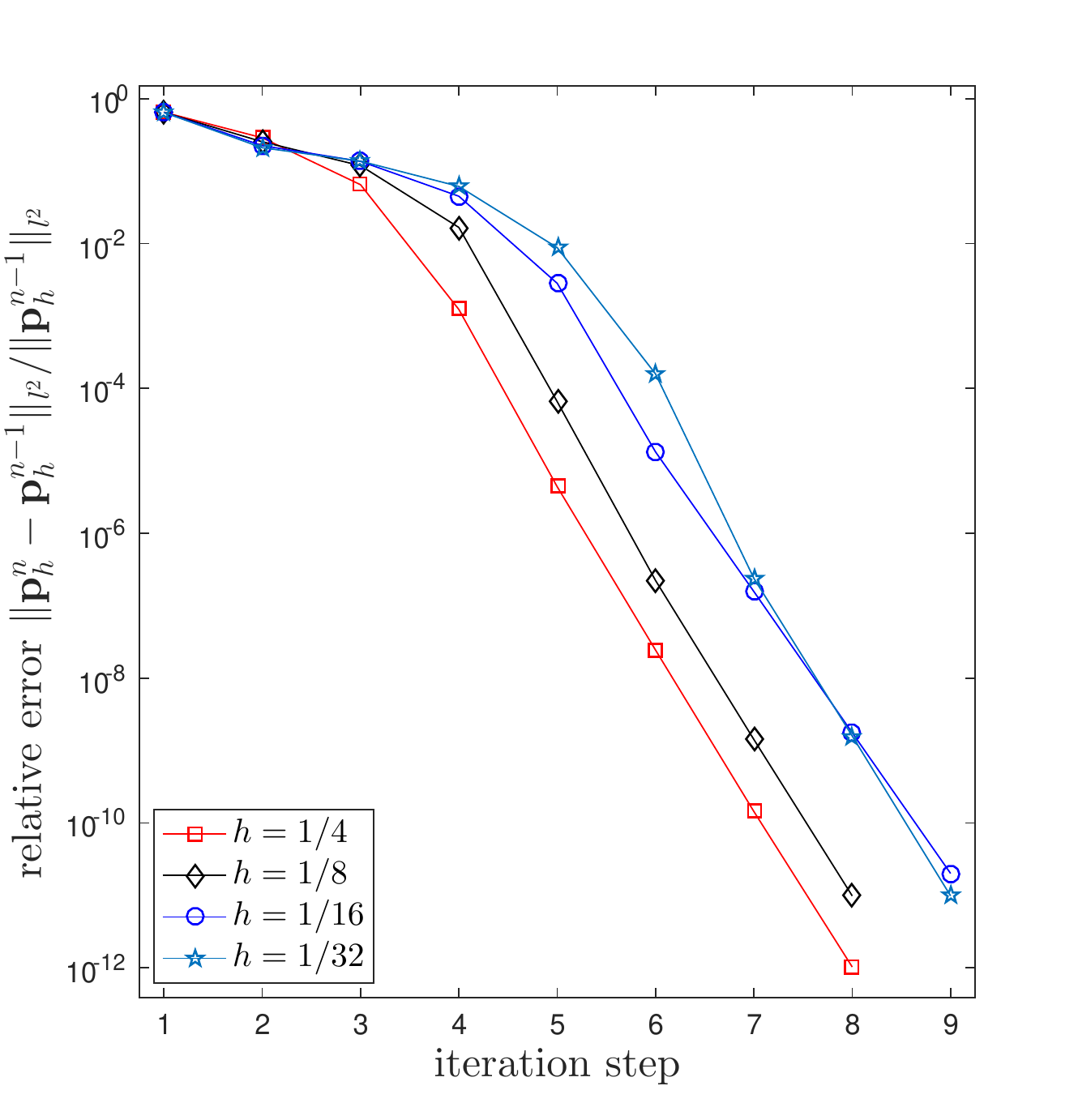}
  \hspace{5pt}
  \includegraphics[width=0.3\textwidth]{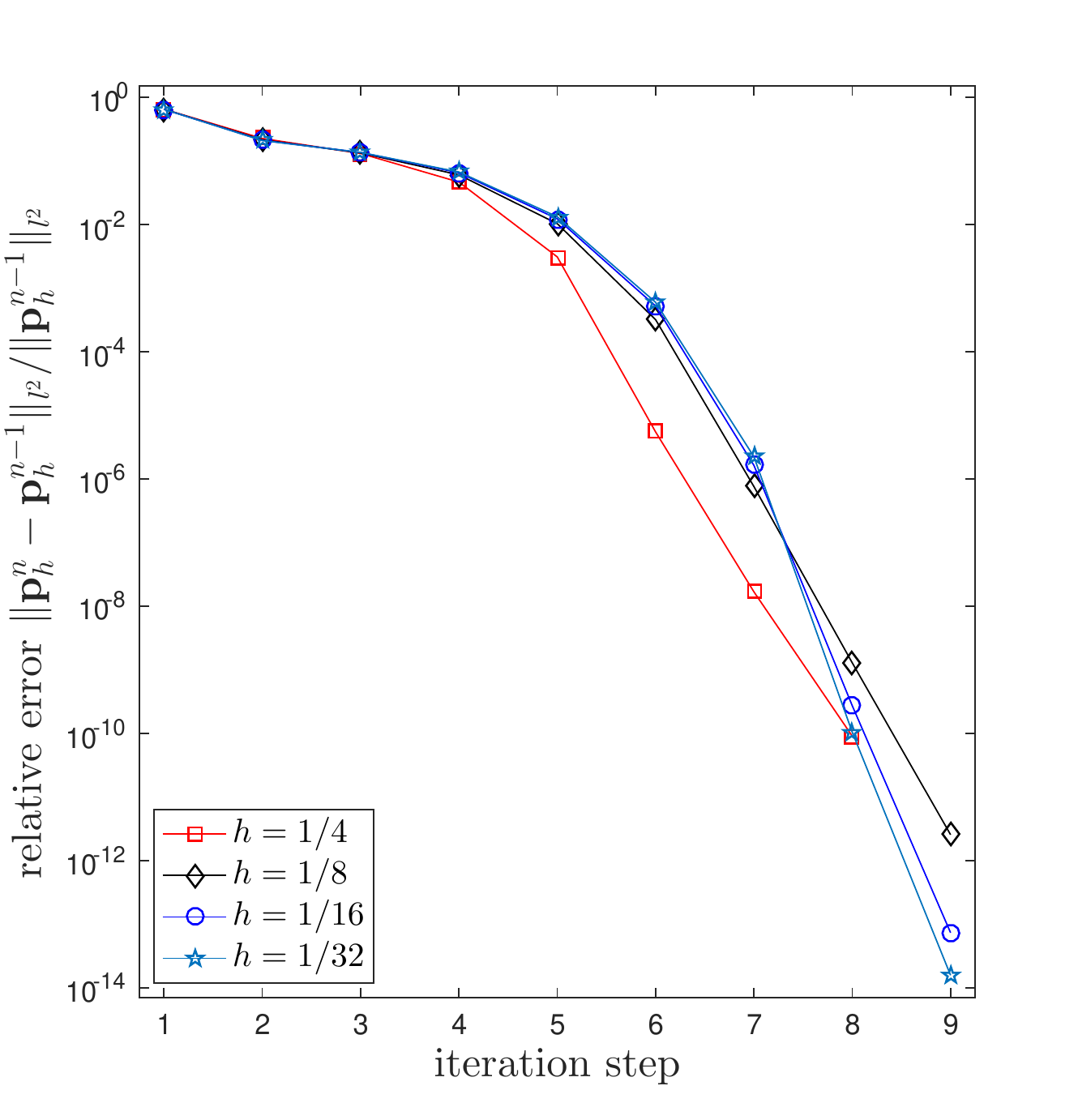}
  \hspace{5pt}
  \includegraphics[width=0.3\textwidth]{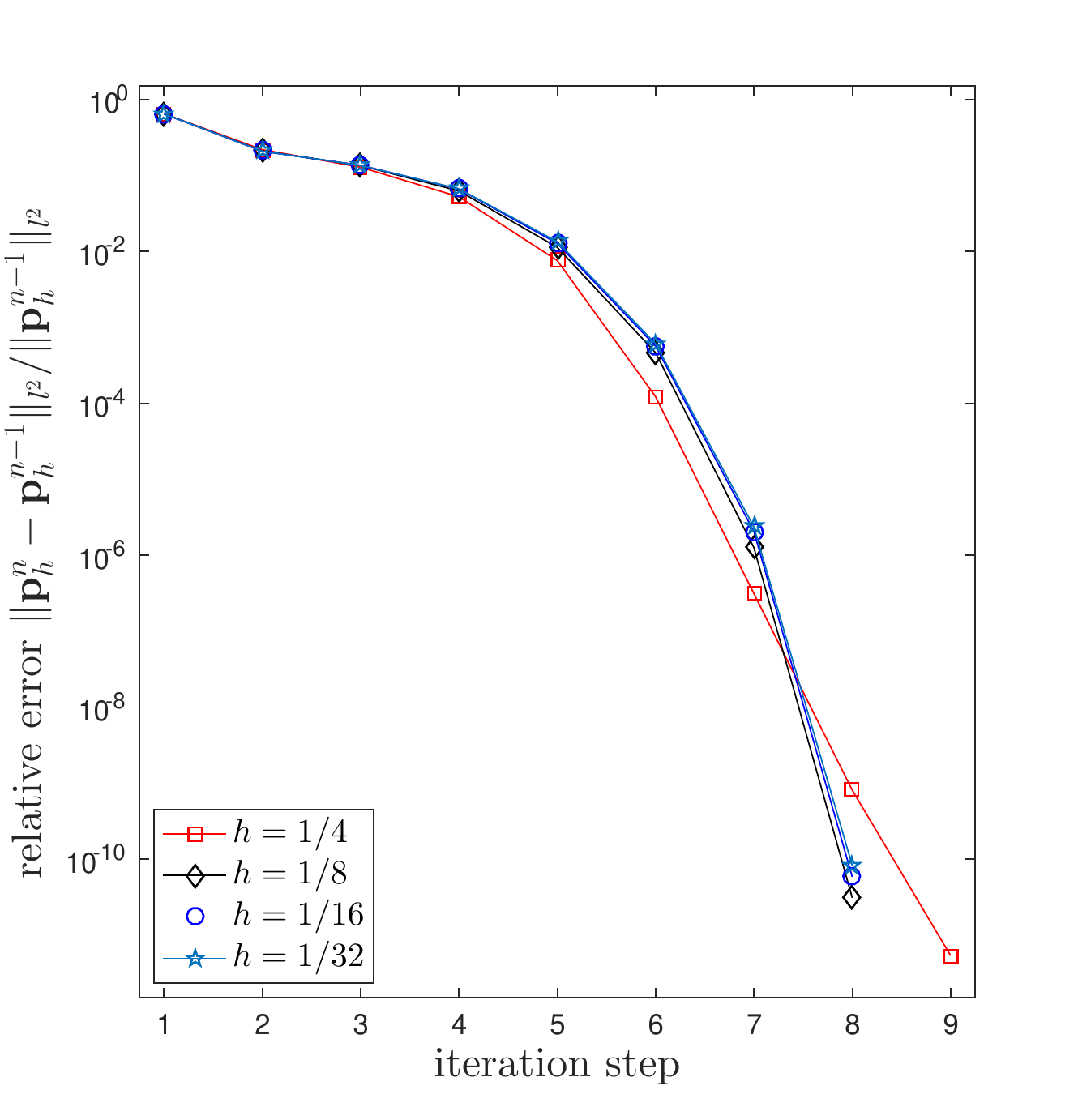}
  \caption{The convergence history of Example 5 for $m=1$ (left) /
  $m=2$ (middle) / $m=3$ (right).}
  \label{fig:ex5his}
\end{figure}

\begin{figure}
  \centering
  \includegraphics[width=0.4\textwidth]{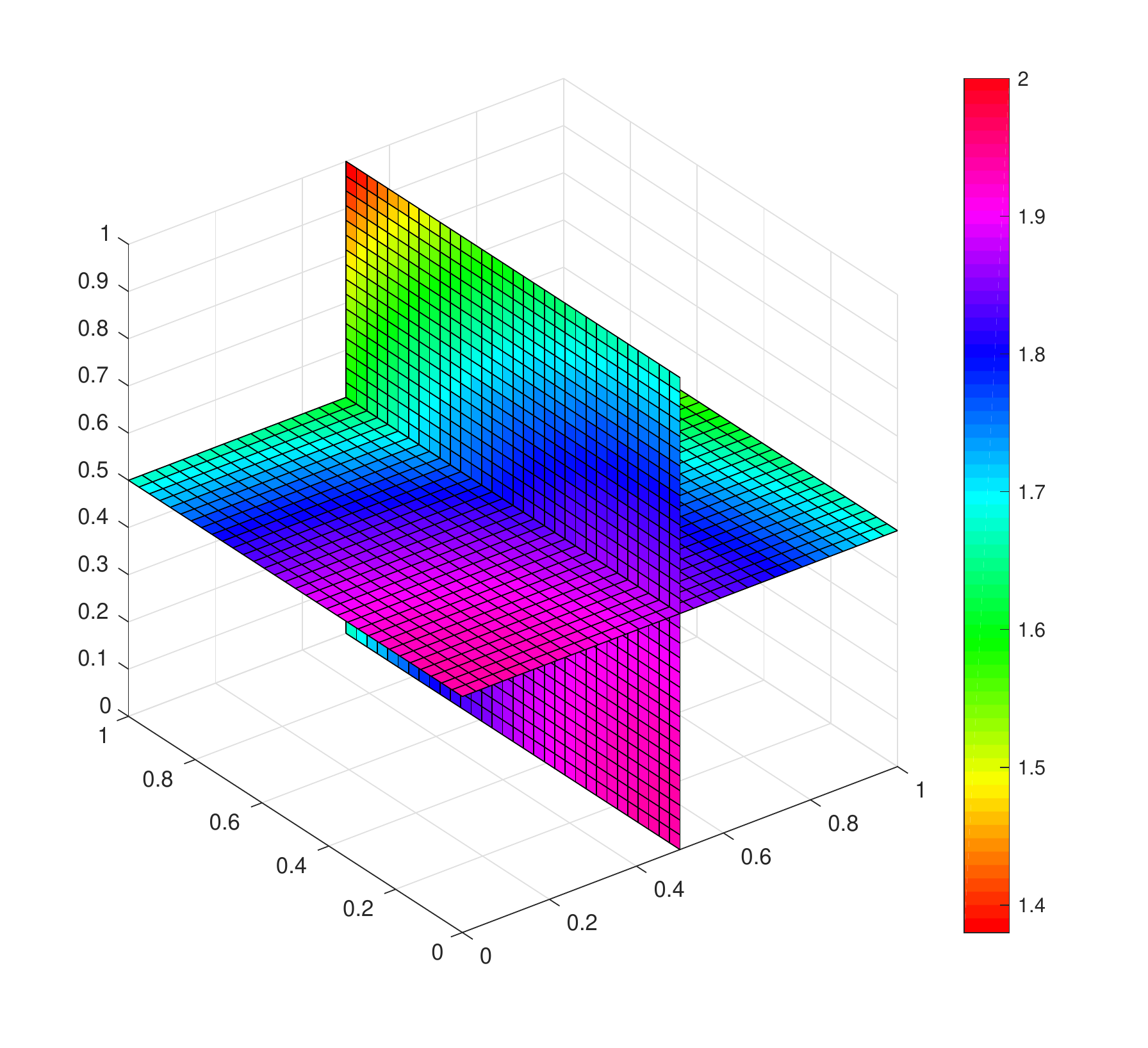}
  \includegraphics[width=0.4\textwidth]{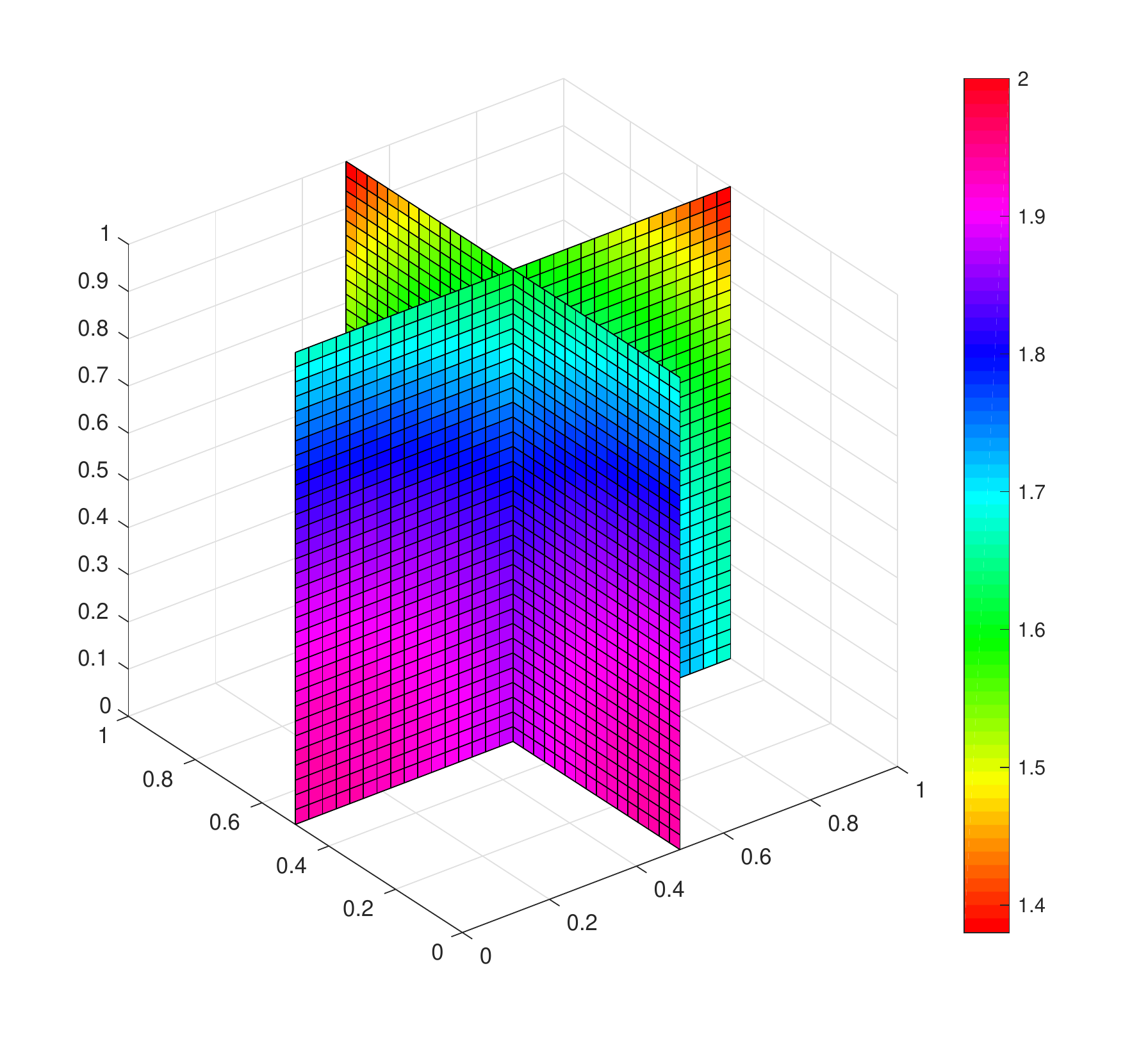}
  \caption{Two slices of the numerical solution for Example 5 on the
  mesh level $h =1/32$.}
  \label{fig:ex5slice}
\end{figure}

\paragraph{\textbf{Example 6:}} Here we test another
three-dimensional Monge-Amp\`ere equation. We solve for the data
function 
\begin{displaymath}
  \begin{aligned}
    f(x, y, z) &= (1 + x^2 + y^2 + z^2) \exp (x^2 + y^2 + z^2), \text{
     in } \Omega, \\
    g(x, y, z) &= \exp \left( \frac{1}{2} \left( x^2 + y^2 + z^2
    \right) \right), \text{ on } \partial \Omega,
  \end{aligned}
\end{displaymath}
which defines the analytical solution $u(x, y, z) = \exp\left(
\frac{1}{2} (x^2 + y^2 + z^2) \right)$. We adopt the function 
\begin{displaymath}
  u^0 = (0.5x^2 + y^2 + 5z^2 + 1)^{3/2}
\end{displaymath}
to be the initial guess.  The numerical results are summarized in
Tab.~\ref{tab:example6} and the optimal convergence rates under the
energy norms are observed. Fig.~\ref{fig:ex6his} plots the error ${ \|
\bm{p}_h^{n} - \bm{p}_h^{n - 1} \|_{l^2} } / { \| \bm{p}_h^{n - 1}
\|_{l^2} }$ for all nonlinear steps, and Fig.~\ref{fig:ex6slice}
depicts two slices of the numerical solution at the mesh level $h =
1/32$.  In this test, we show the robustness of the proposed method in
three dimensions by using the initial values that are in a wider
range. We take $\alpha = 0.01$, $0.1$, $1$, $10$, $100$ and let
$\alpha u^0$ be the initial guess to solve the \MA equation on the
mesh with $h = 1/16$. The number of nonlinear iterations for $1 \leq m
\leq 3$ is listed in Tab.~\ref{tab:ex6alpha}. We note that many
numerical methods for solving the two-dimensional \MA equation cannot
immediately generalize to higher dimensions
\cite{Benamou2010numerical}, and our method still demonstrates a fast
convergence speed and a very wide range of the convergence to the
nonlinear system in three dimensions. 

\begin{table}
  \renewcommand{\arraystretch}{1.3}
  \centering
  \scalebox{0.88}{
    \begin{tabular}{c|r|c|c|c|c|c|c|c|c|c}
    \hline\hline
    $m$ & $h~~~$ & $ \pnorm{\bm{p} - \bm{p}_h} $ & order & $ \| \bm{p} -
    \bm{p}_h\|_{L^2(\Omega)}$  & order & $ \unorm{u - u_h} $
    & order & $ \| u - u_h\|_{L^2(\Omega)}$ & order & \# iter
    \\
    \hline 
    \multirow{4}{0.6cm}{1} & $1/4$ & 
    5.893e-1 & -    & 3.873e-2 & - & 
    2.603e-1 & -    & 1.192e-2 & - &    13 \\
    \cline{2-11}
    & $1/8$ &
    2.785e-1 & 1.17 & 1.198e-2 & 1.69 & 
    1.358e-1 & 0.93 & 3.393e-3 & 1.81 & 10 \\
    \cline{2-11}
    & $1/16$ &
    1.350e-1 & 1.05 & 3.392e-3 & 1.82 & 
    6.926e-2 & 0.97 & 9.023e-4 & 1.91 & 9 \\
    \cline{2-11}
    & $1/32$ &
    6.653e-2 & 1.02 & 9.073e-4 & 1.90 & 
    3.505e-2 & 0.98 & 2.235e-4 & 2.01 & 10 \\
    \hline
    \multirow{4}{0.6cm}{2} & $1/4$ & 
    1.128e-1 & -    & 1.449e-2 & - & 
    2.141e-2 & -    & 2.338e-3 & - &    12 \\
    \cline{2-11}
    & $1/8$ &
    2.853e-2 & 1.98 & 3.250e-3 & 2.16 & 
    5.483e-3 & 1.97 & 5.382e-4 & 2.11 & 8 \\
    \cline{2-11}
    & $1/16$ &
    7.123e-3 & 2.00 & 7.902e-4 & 2.00 & 
    1.409e-4 & 1.96 & 1.323e-4 & 2.02 & 8 \\
    \cline{2-11}
    & $1/32$ &
    1.773e-3 & 2.01 & 1.973e-4 & 2.00 & 
    3.313e-5 & 2.00 & 3.123e-5 & 2.00 & 11 \\
    \hline
    \multirow{4}{0.6cm}{3} & $1/4$ & 
    1.588e-2 & -    & 8.769e-4 & - & 
    1.210e-3 & -    & 4.865e-5 & - &    8 \\
    \cline{2-11}
    & $1/8$ &
    1.923e-3 & 3.03 & 5.178e-5 & 4.08 & 
    1.377e-4 & 3.13 & 2.836e-6 & 4.10 & 8 \\
    \cline{2-11}
    & $1/16$ &
    2.293e-4 & 3.06 & 3.468e-6 & 3.91 & 
    1.723e-5 & 2.99 & 1.748e-7 & 4.02 & 9 \\
    \cline{2-11}
    & $1/32$ &
    2.872e-5 & 3.01 & 2.232e-7 & 3.96 & 
    2.153e-6 & 3.00 & 1.085e-8 & 4.01 & 10 \\
    \hline\hline
  \end{tabular}}
  \caption{Convergence of the Example 6.}
  \label{tab:example6}
\end{table}

\begin{figure}
  \centering
  \includegraphics[width=0.3\textwidth]{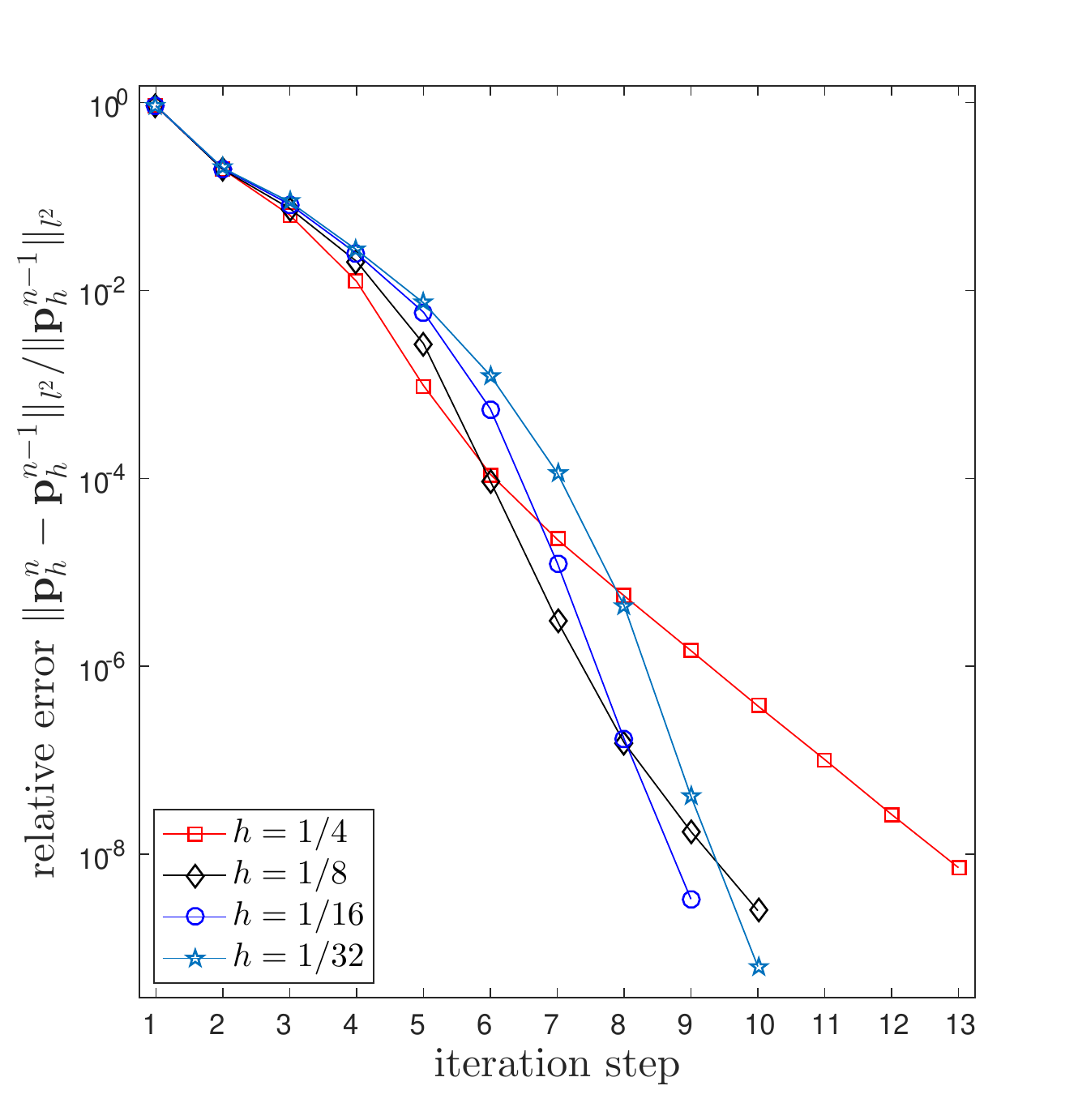}
  \hspace{5pt}
  \includegraphics[width=0.3\textwidth]{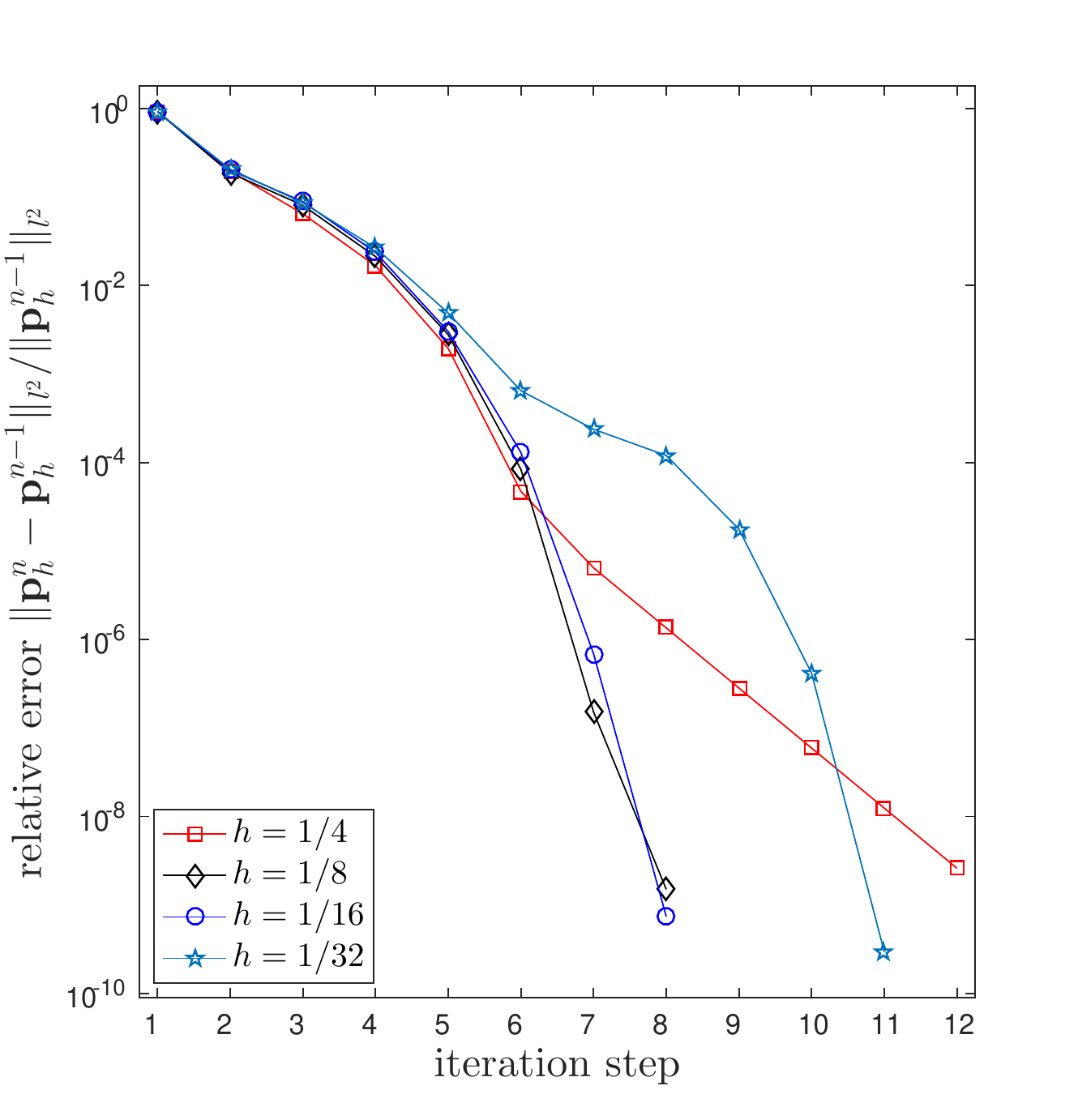}
  \hspace{5pt}
  \includegraphics[width=0.3\textwidth]{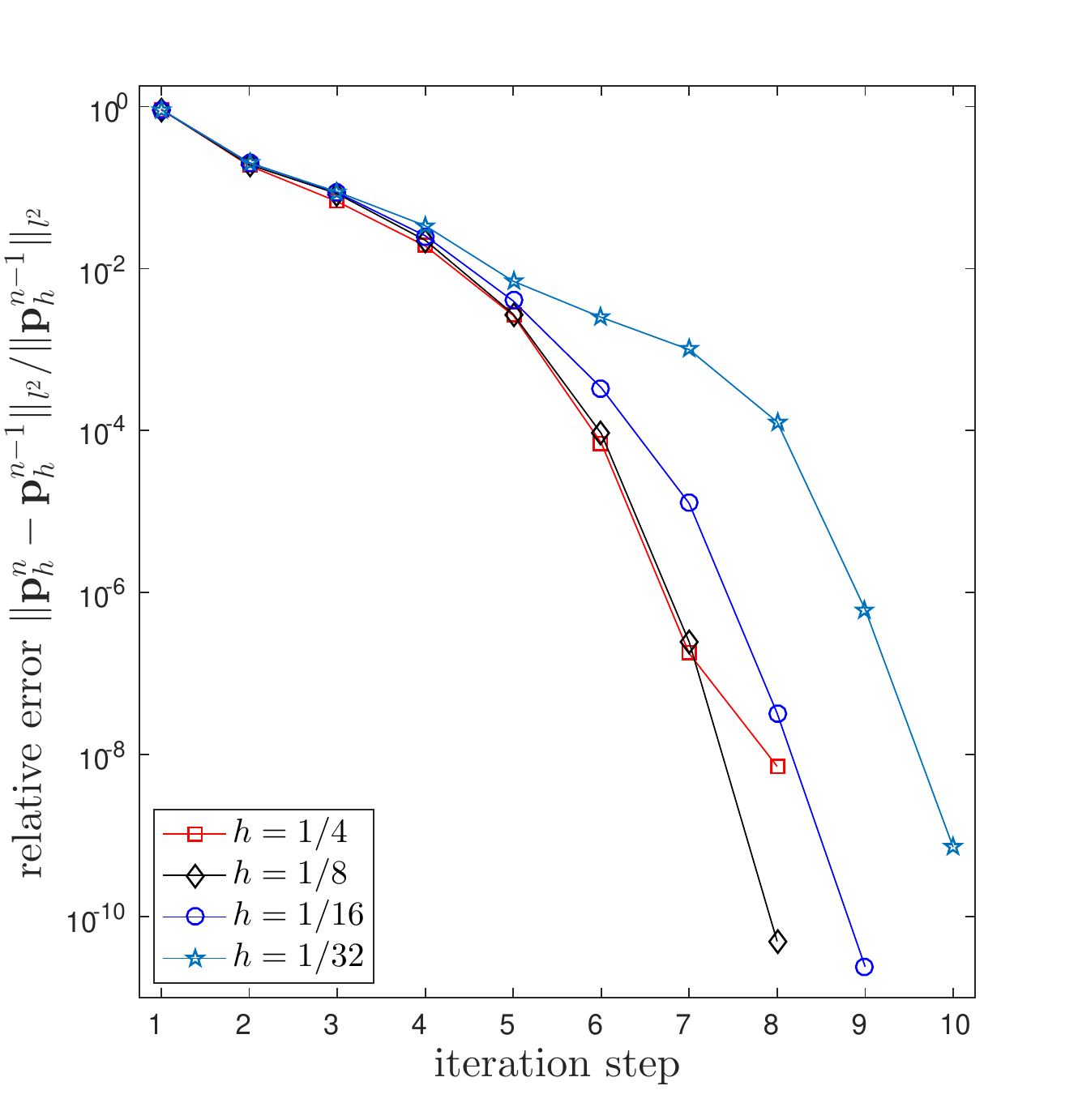}
  \caption{The convergence history of Example 6 for $m=1$ (left) /
  $m=2$ (middle) / $m=3$ (right).}
  \label{fig:ex6his}
\end{figure}

\begin{figure}
  \centering
  \includegraphics[width=0.4\textwidth]{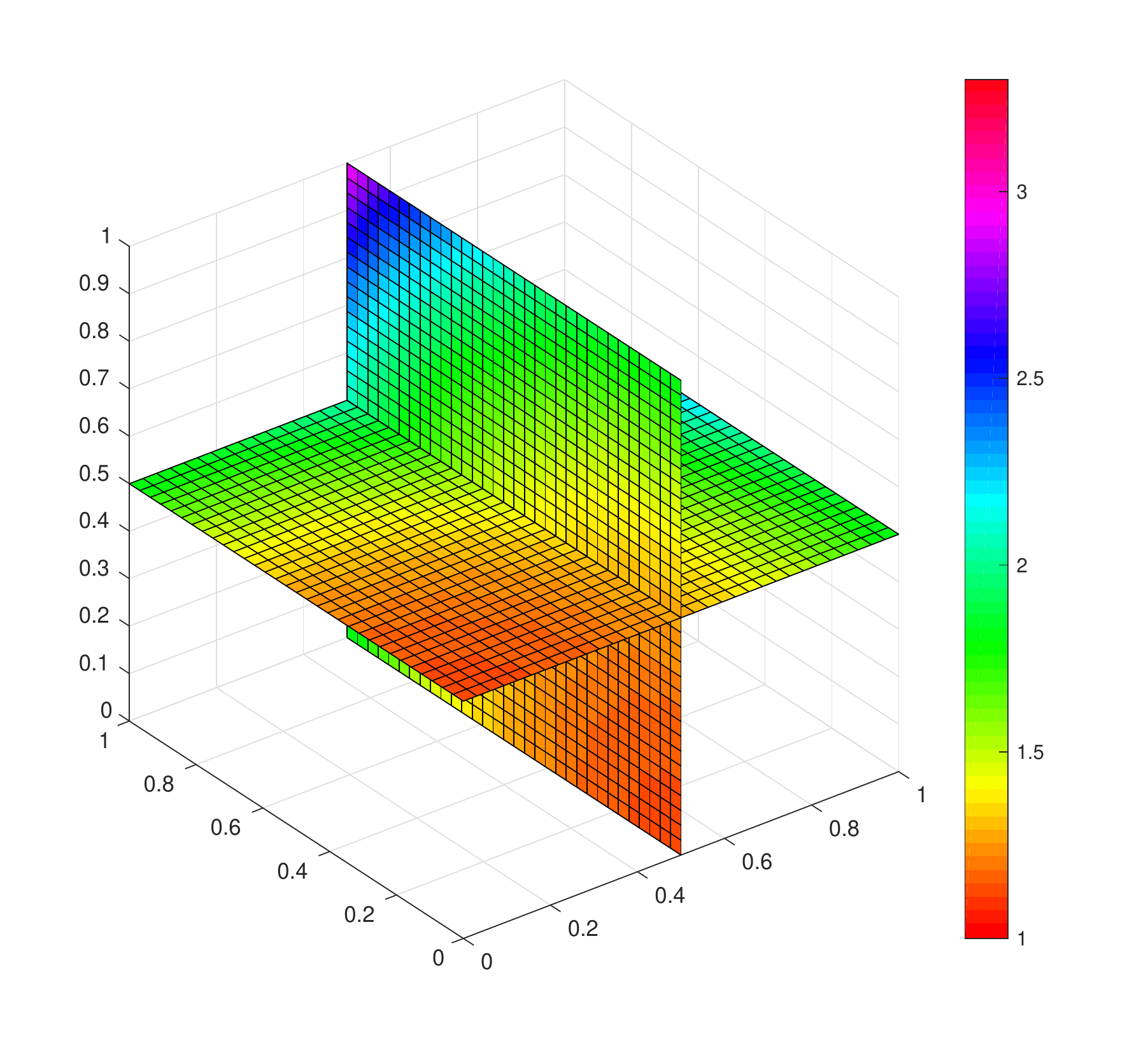}
  \includegraphics[width=0.4\textwidth]{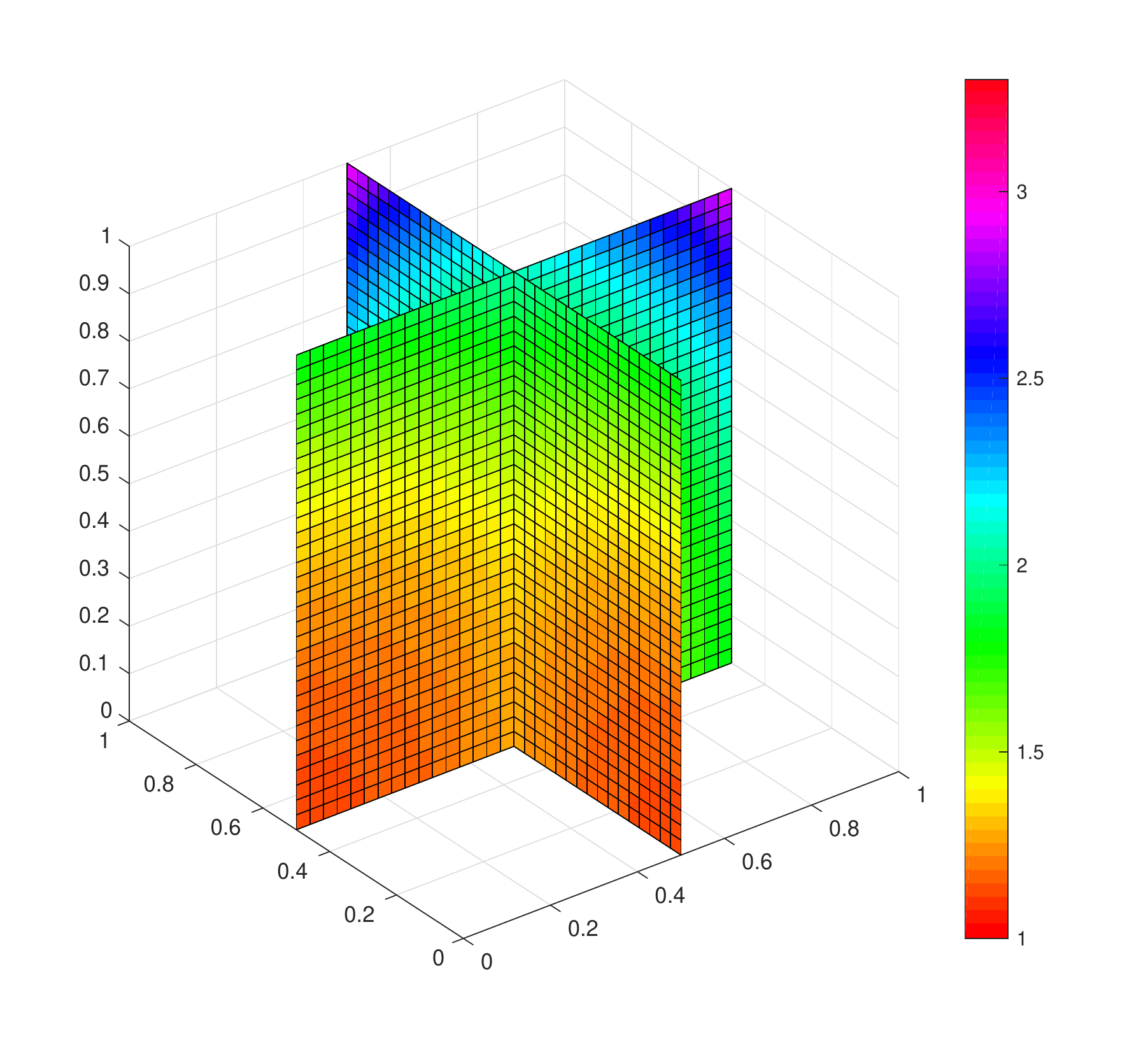}
  \caption{Two slices of the numerical solution for Example 6 on the
  mesh level $h =1/32$.}
  \label{fig:ex6slice}
\end{figure}

\begin{table}
  \renewcommand{\arraystretch}{1.2}
  \centering
  \begin{tabular}{c|c|c|c|c|c|c}
    \hline\hline
    $\quad m \quad$ & $\alpha$ & $0.01$ & $0.1$ & $1$ & $10$ & $100$ \\
    \hline 
    \multirow{2}{0.5cm}{1} & \# iter & 8 & 10 & 9 & 21 & 32  \\
    \cline{2-7}
    & $\pnorm{\bm{p} - \bm{p}_h}$ & 1.350e-1 &  1.350e-1 & 1.350e-1
    &1.350e-1 &1.350e-1 \\
    \hline 
    \multirow{2}{0.5cm}{2} & \# iter & 20 & 11 & 8 & 23 & 23  \\
    \cline{2-7}
    & $\pnorm{\bm{p} - \bm{p}_h}$ & 7.123e-3 &  7.123e-3 & 7.123e-3
    &7.123e-3 & 7.123e-3 \\
    \hline 
    \multirow{2}{0.5cm}{3} & \# iter & 18 & 28 & 9 & 17 & 47  \\
    \cline{2-7}
    & $\pnorm{\bm{p} - \bm{p}_h}$ & 2.293e-4 &  2.293e-4 & 2.293e-4
    &2.293e-4 &2.293e-4 \\
    \hline\hline
  \end{tabular}
  \caption{Iteration steps and numerical errors for the initial values
  with different $\alpha$. }
  \label{tab:ex6alpha}
\end{table}

\section{Comparison between $\bmr{U}_h^m$ and $\bmr{S}_h^m$}
\label{sec:comparison}
In this section, we present some numerical evidences to show that the
robustness is due to the space $\bmr{U}_h^m$ by making a
comparison with the standard piecewise irrotational polynomial space
$\bmr{S}_h^m$. As we stated in Remark \ref{re:space}, the nonlinear
system can also be solved with the space $\bmr{S}_h^m$. There are two
numerical evidences that strongly encourage us to adopt the
reconstructed space $\bmr{U}_h^m$ to solve the \MA problem.

The first reason is the efficiency of the approximation. The number of
the degrees of freedom of a specific discretized system can serve as
a proper indicator for the efficiency of the numerical scheme
\cite{hughes2000comparison}, and we have shown that for some classical
linear problems the reconstructed space demonstrates a better
numerical efficiency \cite{li2019eigenvalue, li2017discontinuous,
li2019least}. Here we compare the numerical efficiency between the
spaces $\bmr{U}_h^m$ and $\bmr{S}_h^m$ by adopting both approximation
spaces to solve the Example 1. The initial guess is taken as the
solution of the related problem
\begin{displaymath}
  \Delta u = 2 \sqrt{f} \text{ in } \Omega, \quad u = g \text{ on }
  \partial \Omega.
\end{displaymath}
In Fig.~\ref{fig:effcompare}, we plot the errors $\pnorm{\bm{p} -
\bm{p}_h}$ for both two approximation spaces against the number of
degrees of freedom with $1 \leq m \leq 3$. For all accuracy $m$, fewer
degrees of freedom are used with the reconstructed space $\bmr{U}_h^m$
to achieve the same error and the advantage of the efficiency is more
remarkable as $m$ increases. The great numerical efficiency is the
first reason that the reconstructed space $\bmr{U}_h^m$ is preferred.

\begin{figure}
  \centering
  \includegraphics[width=0.3\textwidth]{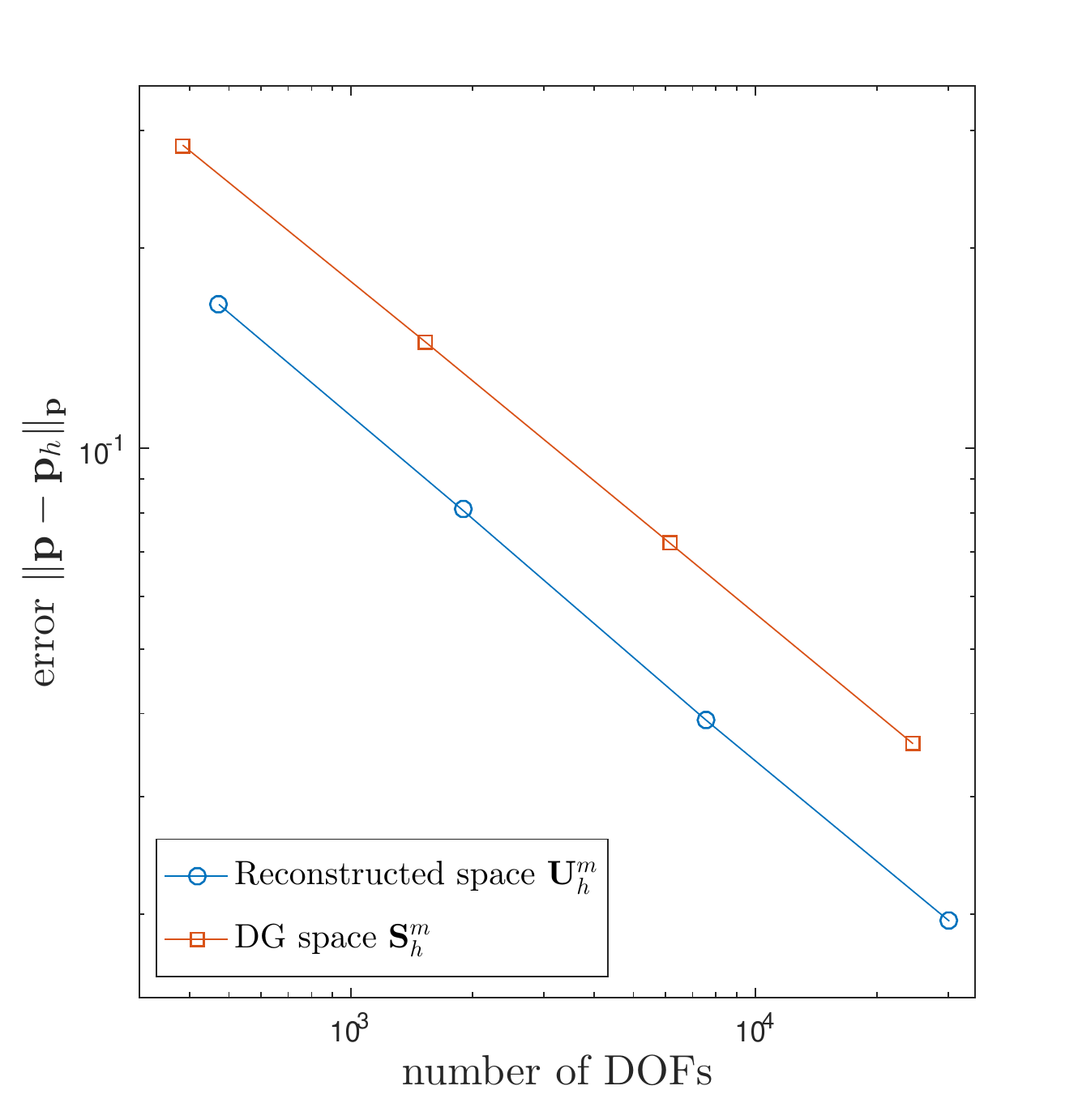}
  \hspace{5pt}
  \includegraphics[width=0.3\textwidth]{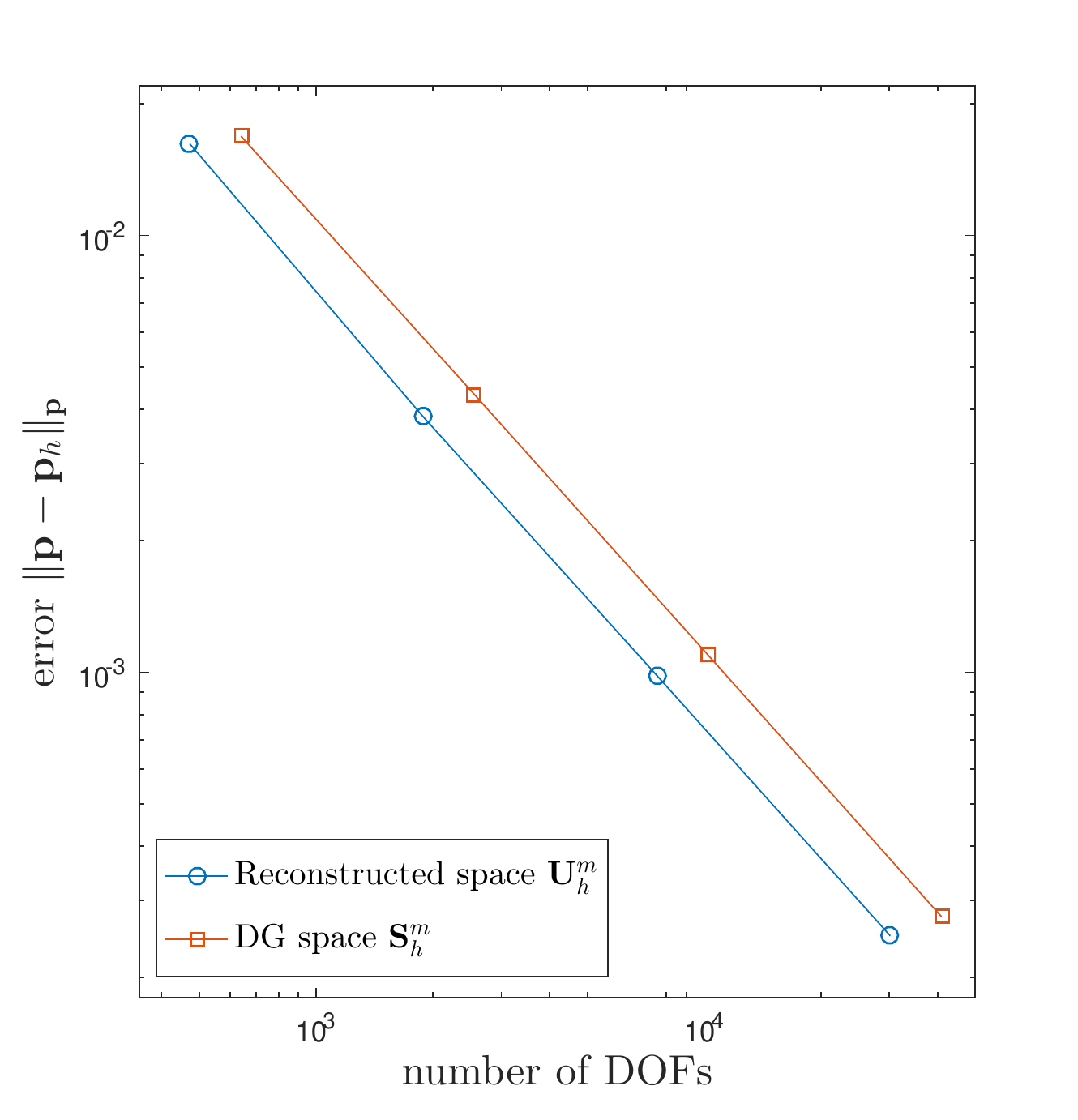}
  \hspace{5pt}
  \includegraphics[width=0.3\textwidth]{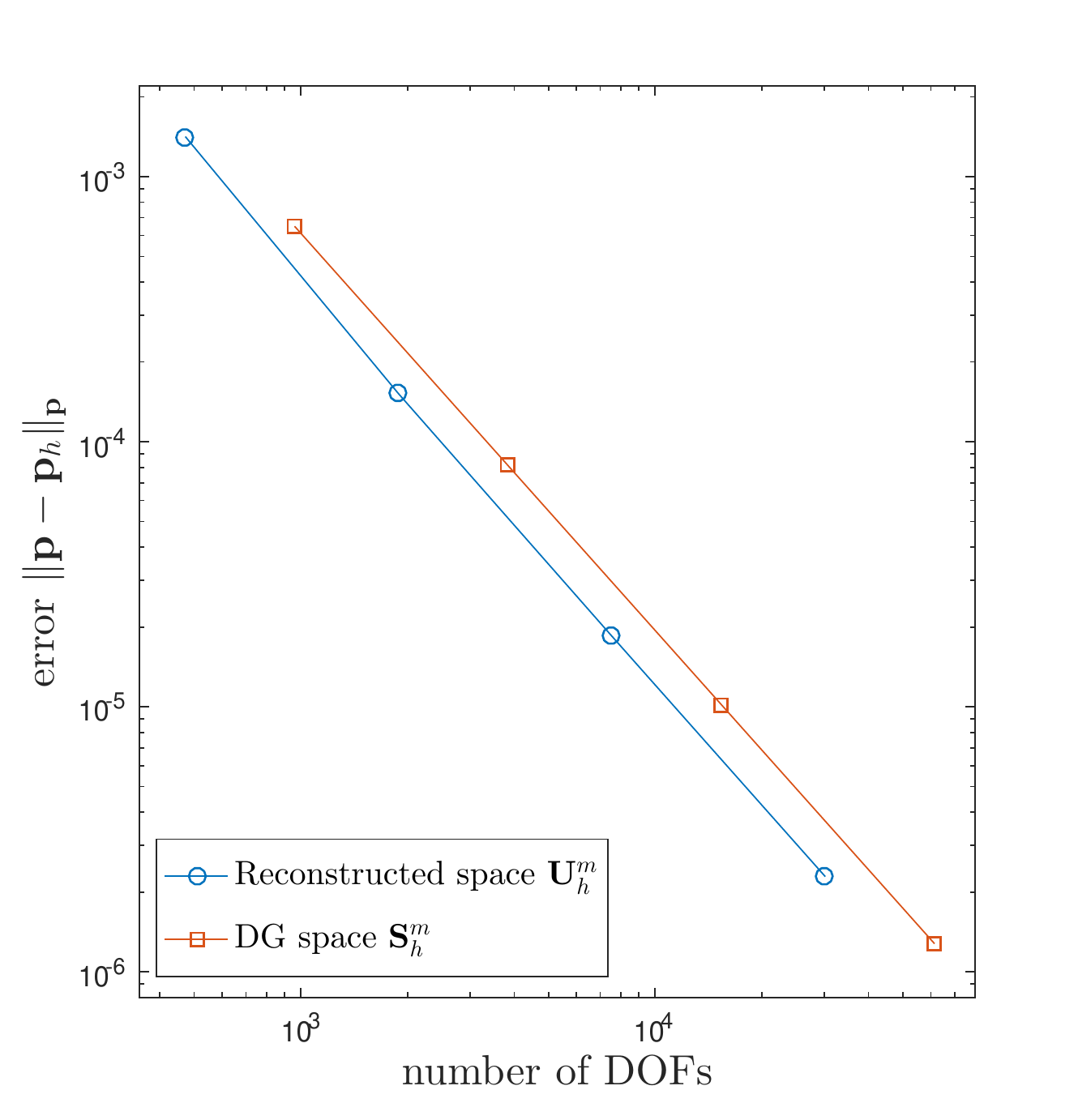}
  \caption{The error $\pnorm{\bm{p} - \bm{p}_h}$ against the number of
  DOFs for $m=1$ (left) / $m=2$ (middle) / $m=3$ (right).}
  \label{fig:effcompare}
\end{figure}

More importantly, the reconstructed space $\bmr{U}_h^m$ demonstrates a
better robustness in handling the nonlinearity. We solve the Example 1
with the initial guess \eqref{eq:test2initial} by both two
approximation spaces. The penalty parameter $\eta$ is still taken as
$\eta = 20$ for $1 \leq m \leq 3$ and the number of maximum iterations
is set as $100$. We list the number of the nonlinear iterations in
Tab.~\ref{tab:compsteps} and clearly the method with the space
$\bmr{U}_h^m$ has a much better numerical performance than the method
using the space $\bmr{S}_h^m$. Adopting the reconstructed space
$\bmr{U}_h^m$ gives a much faster convergence speed under the same
numerical setting. As a numerical observation, we find that the method
with the space $\bmr{S}_h^m$ is very sensitive to the penalty
parameter. For the initial guess that is far from the exact solution,
the penalty $\eta$ shall be large enough to guarantee the
convergence. We scale the initial guess \eqref{eq:test2initial} by
multiplying a coefficient $\alpha$ to show this phenomenon. We solve
the problem by the initial value $\alpha u^0(x, y)$
($\alpha = 1, 10, 100$) and the penalty parameter $\eta$ is taken as
$\eta = 20, 60, 100, 300$. The number of nonlinear iterations for
different initial values are reported in Tab.~\ref{tab:comppara1} -
Tab.~\ref{tab:comppara100} (In these tables, $\infty$ means that there
are no indications to convergence at all after 100 Newton iterative
steps). It can be observed that the case $\alpha = 10$ provides a
better initial guess than the other two cases, which is consistent
with the results in Fig.~\ref{fig:ex2alphahistory}.  In addition, we
can see that the method with the space $\bmr{S}_h^m$ requires a very
large penalty especially for the high-order accuracy.  For the
reconstructed space $\bmr{U}_h^m$, the number of iterations seems
essentially unrelated to the initial value and the penalty parameter
$\eta$. Here we give a numerical explanation for the reason why the
space $\bmr{S}_h^m$ may not work.  We solve the problem with the
initial guess $100 u^0(x, y)$ at the mesh level $h = 1/80$, and in
Fig.~\ref{fig:DGnospd}, we monitor the number of Non-Convex elements
at each iteration with different parameter $\eta = 20$, $60$, $100$.
It can be seen that the method cannot converge if there are always
Non-Convex elements in partition. For the space $\bmr{S}_h^m$, it
requires a very large penalty parameter to guarantee the piecewise
convexity of the numerical solution. In Fig.~\ref{fig:DGnospd}, it is
clear that there are less Non-Convex elements with larger $\eta$. We
also observe that for the case $\eta = 100$ and $m=1$, there are five
iterations to the convergence after the numerical solution is
piecewise convex. On the other hand, for the reconstructed space
$\bmr{U}_h^m$, the Tab.~\ref{tab:comppara1} -
Tab.~\ref{tab:comppara100} indicate that the number of iterations is
irrelevant to the penalty parameter. The space $\bmr{U}_h^m$ seems to
have a natural capability to capture the convexity of the solution.
As we illustrated in Example 3 and Example 4, a few nonlinear
iterations are involved in the first process to eliminate the
Non-Convex elements, and in the second process the convergence is very
fast.  The great robustness is another reason that strongly encourages
us to adopt the reconstructed space to solve the Monge-Amp\`ere
equation, and we leave the underlying reason to the future work.

\begin{table}
  \renewcommand{\arraystretch}{1.3}
  \centering
  \begin{tabular}{c|c|c|c|c}
    \hline\hline 
    \multicolumn{2}{c|}{$m$} & $\qquad 1 \qquad$ & $\qquad 2 \qquad$ & $\qquad 3 \qquad$ \\
    \hline
    \multirow{2}{1cm}{$\bmr{U}_h^m$} & $h=1/40$ & 10  & 12 & 11 \\
    \cline{2-5}
    & $h=1/80$ & 11 & 9  & 14 \\
    \hline
    \multirow{2}{1cm}{$\bmr{S}_h^m$} & $h=1/40$ 
    & 20  & 29 & $\infty$ \\
    \cline{2-5}
    & $h=1/80$ 
    & $\infty$ & $\infty$ & $\infty$ \\
    \hline\hline
  \end{tabular}
  \caption{Iteration steps for approximation spaces $\bmr{U}_h^m$ and
  $\bmr{S}_h^m$.}
  \label{tab:compsteps}
\end{table}

\begin{table}
  \renewcommand{\arraystretch}{1.3}
  \centering 
  \begin{tabular}{c|c|c|c|c}
    \hline\hline
    $\alpha=1, \eta = $ & $\qquad  20 \qquad$ & $\qquad  60 \qquad$
    & $\qquad 100 \qquad$ & $\qquad 300 \qquad$ \\
    \hline
    $\bmr{S}_h^1$ & $\infty$ & 25 & 12  & 12 \\
    \hline
    $\bmr{U}_h^1$ &  11   &  9  &  10  & 9  \\
    \hline
    $\bmr{S}_h^2$ & $\infty$ & 32 & 21  & 20 \\
    \hline
    $\bmr{U}_h^2$ &  9   &  12  &  13  & 12  \\
    \hline
    $\bmr{S}_h^3$ & $\infty$ & $\infty$ & $\infty$ & 37 \\
    \hline
    $\bmr{U}_h^3$ &  14   &  14  &  14  & 12  \\
    \hline\hline
  \end{tabular}
  \caption{Iteration steps for different parameter $\eta$ with $\alpha
  = 1$.}
  \label{tab:comppara1}
\end{table}

\begin{table}
  \renewcommand{\arraystretch}{1.3}
  \centering 
  \begin{tabular}{c|c|c|c|c}
    \hline\hline
    $\alpha=10, \eta = $ & $\qquad  20 \qquad$ & $\qquad  60 \qquad$
    & $\qquad 100 \qquad$ & $\qquad 300 \qquad$ \\
    \hline
    $\bmr{S}_h^1$ & 11  & 11 & 11  & 11 \\
    \hline
    $\bmr{U}_h^1$ &  8   &  8  &  8  & 8  \\
    \hline
    $\bmr{S}_h^2$ & 13 & 12 & 12  & 12 \\
    \hline
    $\bmr{U}_h^2$ &  8   &  8  &  9  & 9  \\
    \hline
    $\bmr{S}_h^3$ & 26 & 11  & 12 & 12 \\
    \hline
    $\bmr{U}_h^3$ &  8   &  9  &  9  & 10  \\
    \hline\hline
  \end{tabular}
  \caption{Iteration steps for different parameter $\eta$ with $\alpha
  = 10$.}
  \label{tab:comppara10}
\end{table}

\begin{table}
  \renewcommand{\arraystretch}{1.3}
  \centering 
  \begin{tabular}{c|c|c|c|c}
    \hline\hline
    $\alpha=100, \eta = $ & $\qquad  20 \qquad$ & $\qquad  60 \qquad$
    & $\qquad 100 \qquad$ & $\qquad 300 \qquad$ \\
    \hline
    $\bmr{S}_h^1$ & $\infty$ & $\infty$ &  23 & 17 \\
    \hline
    $\bmr{U}_h^1$ &  12   &  13  &  12  & 12  \\
    \hline
    $\bmr{S}_h^2$ & $\infty$ & $\infty$ & $\infty$  & 23 \\
    \hline
    $\bmr{U}_h^2$ &  13   &  13  &  12  & 12  \\
    \hline
    $\bmr{S}_h^3$ & $\infty$ & $\infty$ & $\infty$ & 63 \\
    \hline
    $\bmr{U}_h^3$ &  14   &  14  &  13  & 13  \\
    \hline\hline
  \end{tabular}
  \caption{Iteration steps for different parameter $\eta$ with $\alpha
  = 100$.}
  \label{tab:comppara100}
\end{table}


\begin{figure}
  \centering
  \includegraphics[width=0.3\textwidth]{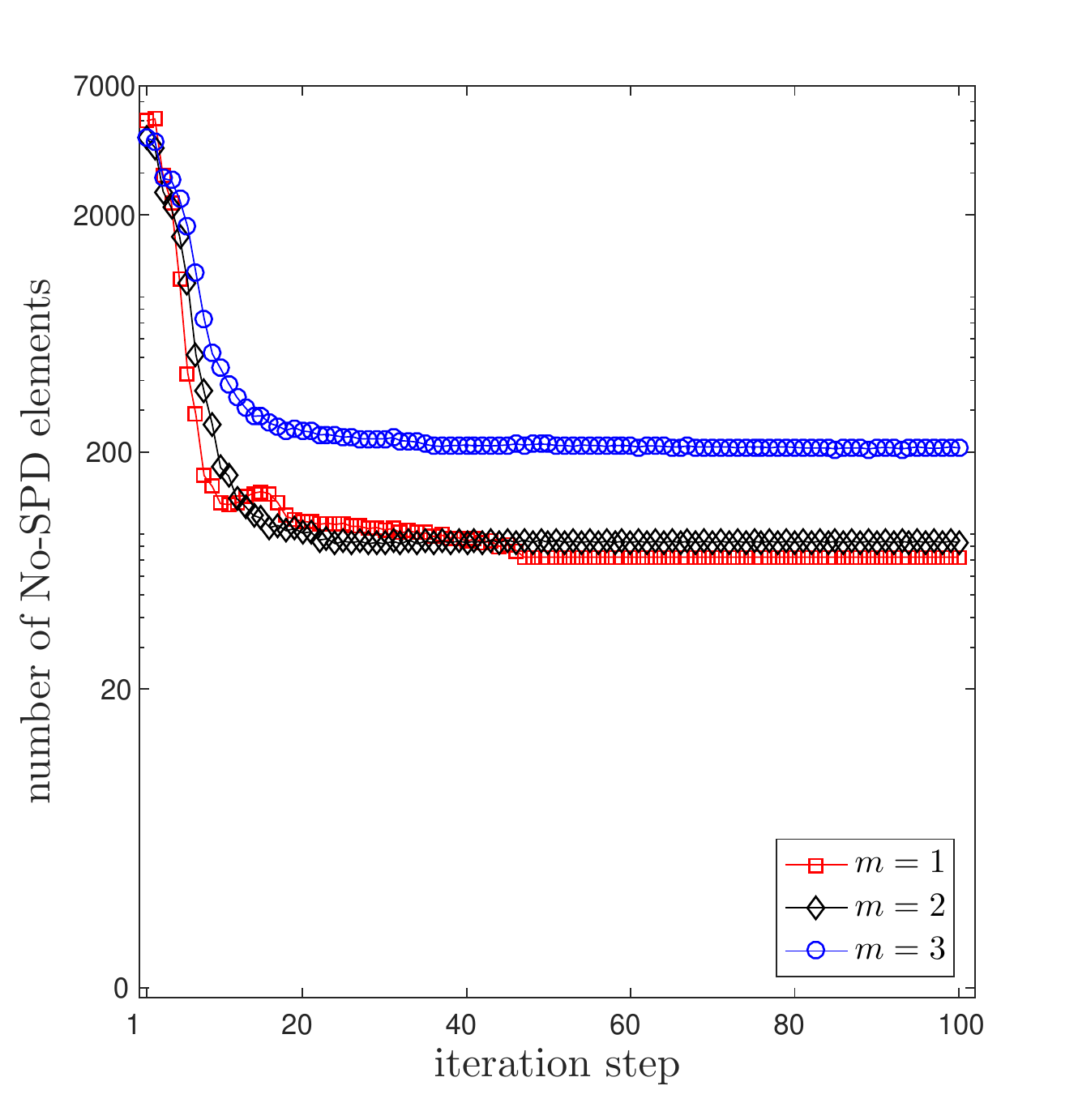}
  \hspace{5pt}
  \includegraphics[width=0.3\textwidth]{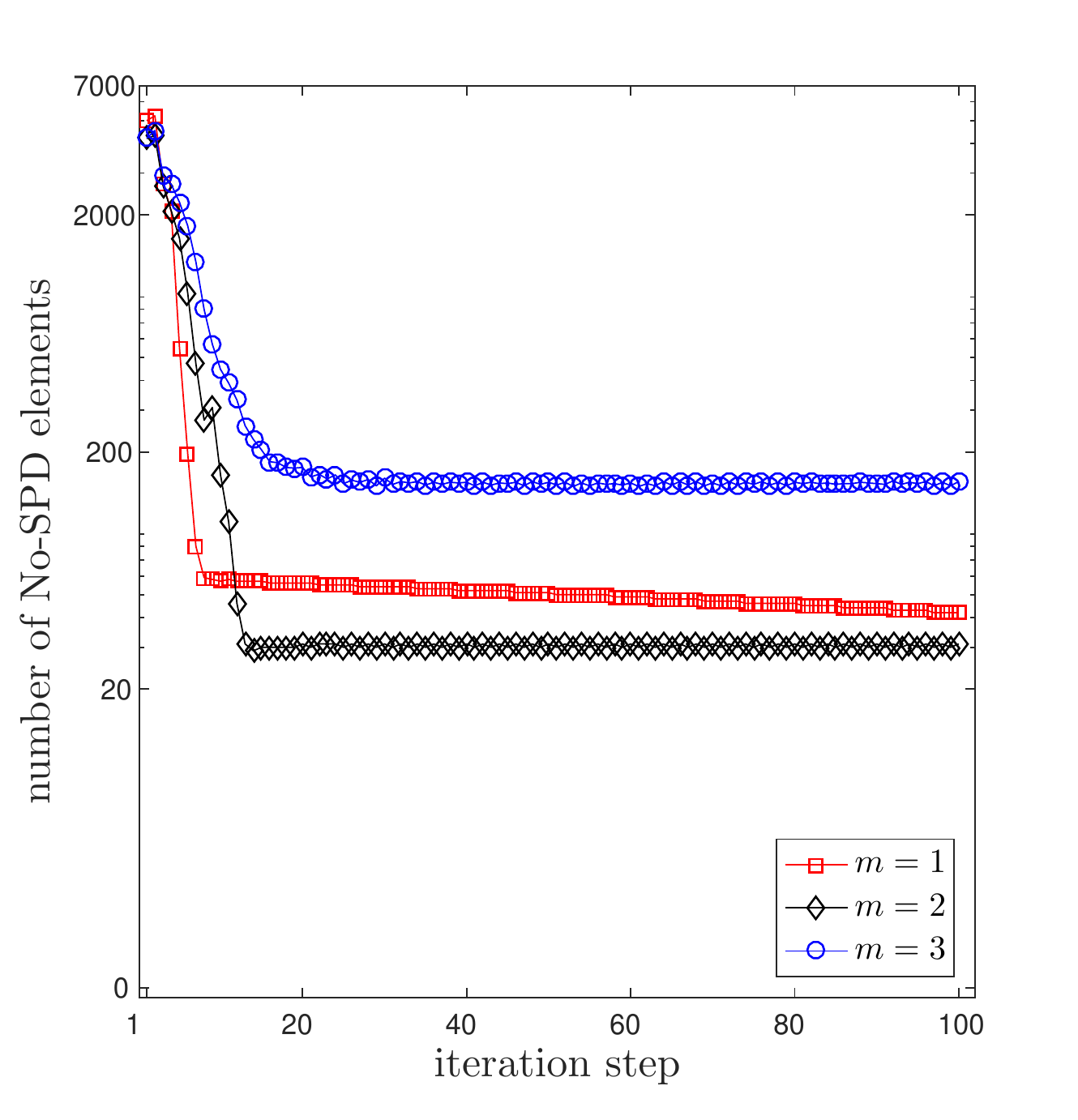}
  \hspace{5pt}
  \includegraphics[width=0.3\textwidth]{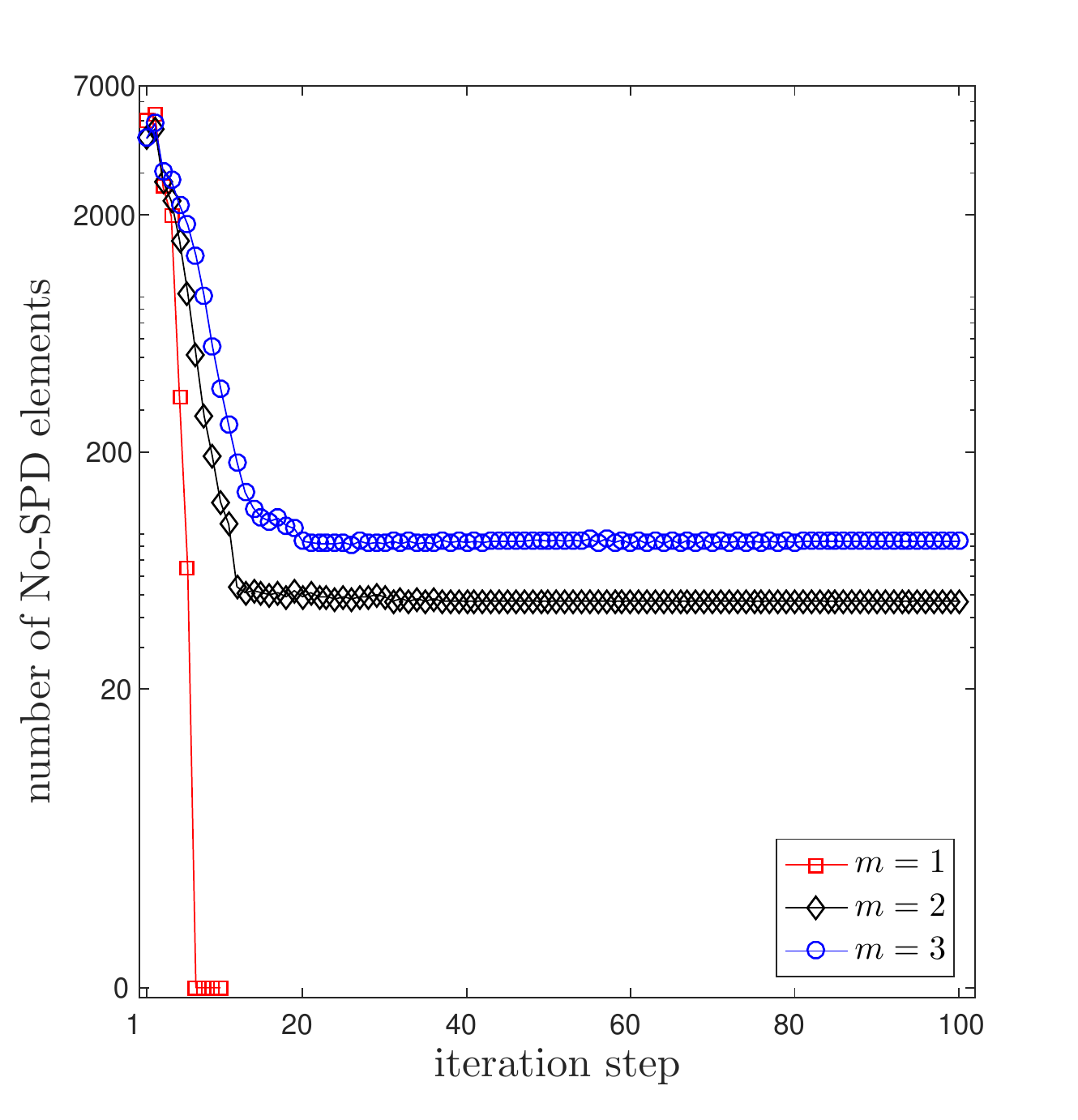}
  \caption{The number of Non-Convex elements at each iteration step for
  $\sigma = 20$ (left) / $\sigma = 60$ (middle) / $\sigma = 100$
  (right).}
  \label{fig:DGnospd}
\end{figure}


\section{Conclusion}
\label{sec:conclusion}
A reconstructed discontinuous approximation method for the \MA
equation with classical solutions was proposed. We adopted a piecewise
irrotational polynomial space which is constructed by patch
reconstruction. With this space, it is able to solve a nonlinear
system with the Newton iteration to obtain an approximation to
the gradient at first, and then the primitive variable is solved by
using standard Lagrange finite element space. The optimal convergence
rate and the robustness demonstrated by the numerical examples are
quite attractive features of the proposed method in solving the \MA
equation.

\section*{Acknowledgements} 
This research was supported by the Science Challenge Project (No.
TZ2016002) and the National Science Foundation in China (No.
11971041).

\begin{appendix}
  \section{}
  \label{sec:appendix}
  Below let us present some details of the program implementation
  to the reconstructed finite element space. We first outline the
  bases of the irrotational polynomial space $\bmr{S}(D)^k$. We 
  write $\bmr{S}(D)^k$ as $\bmr{S}(D)^k = \nabla V(D)^{k+1}$, where
  $V(D)^{k+1}$ is the polynomial space of degree $k+1$. This fact
  implies that we can obtain a base by taking the gradient of the base
  of $V(D)^{k+1}$, which reads
  \begin{displaymath}
    1, x, y, x^2, xy, y^2 , \ldots, xy^{k-1}, y^k.
  \end{displaymath}
  in two dimensions. For the case $k=1$, it is easy to obtain a
  base of $S(D)^1$, that is 
  \begin{equation}
    \vect{1}{0}, \quad \vect{0}{1}, \quad \vect{x}{0}, \quad
    \vect{0}{y}, \quad \vect{y}{x}. 
    \label{eq:k1base}
  \end{equation}
  Such a method can be easily extended to the case of three dimensions
  and the case of higher order accuracy.

  \begin{figure}[htp]
    \centering
    \begin{tikzpicture}[scale=1]
      \coordinate (A) at (1, 0); 
      \coordinate (B) at (-0.5, -0.6);
      \coordinate (C) at (-0.5, 0.8);
      \coordinate (D) at (1.2, 1.5);
      \coordinate (E) at (-2, 0);
      \coordinate (F) at (0.8, -1.5);
      \draw[fill, red] (A) -- (B) -- (C);
      \draw[thick, black] (A) -- (C) -- (B) -- (A);
      \draw[thick, black] (A) -- (D) -- (C);
      \draw[thick, black] (C) -- (E) -- (B);
      \draw[thick, black] (A) -- (F) -- (B);
      \node at(0, 0) {$K_0$}; \node at(1.7/3, 2.3/3) {$K_1$};
      \node at(1.3/3, -2.1/3) {$K_2$}; \node at(-1, 0.2/3) {$K_3$};
    \end{tikzpicture}
    \hspace{70pt}
    \begin{tikzpicture}[scale=1]
      \coordinate (A) at (1, 0); 
      \coordinate (B) at (-0.5, -0.6);
      \coordinate (C) at (-0.5, 0.8);
      \coordinate (D) at (1.2, 1.5);
      \coordinate (E) at (-2, 0);
      \coordinate (F) at (0.8, -1.5);
      \draw[fill, red] (A) -- (B) -- (C);
      \draw[thick, black] (A) -- (C) -- (B) -- (A);
      \draw[thick, black] (A) -- (D) -- (C);
      \draw[thick, black] (C) -- (E) -- (B);
      \draw[thick, black] (A) -- (F) -- (B);
      \draw[fill, black] (0, 0) circle [radius=0.03]; 
      \draw[fill, black] (1.7/3, 2.3/3)  circle [radius=0.03]; 
      \draw[fill, black] (1.3/3, -2.1/3) circle [radius=0.03];  
      \draw[fill, black] (-1, 0.2/3)     circle [radius=0.03];
      \node at(-0.1, -0.2) {$\bm{x}_{K_0}$}; 
      \node at(1.7/3-0.1, 2.3/3-0.2)  {$\bm{x}_{K_1}$};
      \node at(1.3/3-0.1, -2.1/3-0.2) {$\bm{x}_{K_2}$}; 
      \node at(-1-0.1, 0.2/3-0.2)     {$\bm{x}_{K_3}$};
    \end{tikzpicture}
    \caption{$K_0$ and its neighbours (left) / collocation points
    (right).}
    \label{fig:Kexample}
  \end{figure}
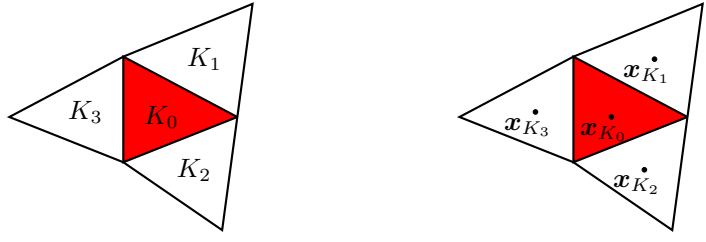
  Then we show the linear reconstruction process to present the
  details of the reconstructed space. We take $K_0$ in
  Fig.~\ref{fig:Kexample} as an illustration and we let its patch set
  $S(K_0)$ is formed by its face-neighbouring elements, $S(K_0) =
  \left\{ K_0, K_1, K_2, K_3 \right\}$ and $\mc{I}(K_0) = \left\{
  \bm{x}_{K_0},   \bm{x}_{K_1},  \bm{x}_{K_2},  \bm{x}_{K_3}
  \right\}(\bm{x}_{K_i} = (x_{K_i}, y_{K_i}))$. For a piecewise
  constant function $\bm{g} = (g^1, g^2) \in \bmr{U}_h^0$, the
  corresponding discrete least squares problem is 
  \begin{equation}
    \begin{aligned}
      \mathop{\arg \min}_{\bm{p} \in \bmr{S}(S(K_0))^1} \sum_{i =
      0}^4 \| \bm{p}(\bm{x}_i) - \bm{g}(\bm{x}_i) \|^2 \quad
        \text{s.t. }  \bm{p}(\bm{x}_{K_0}) = \bm{g}(\bm{x}_{K_0}). \\
    \end{aligned}
    \label{app:eq:ls}
  \end{equation}
  By \eqref{eq:k1base}, we could expand the $\bm{p}(\bm{x})$ as 
  \begin{displaymath}
    \bm{p}(\bm{x}) = a_0 \vect{1}{0} + a_1 \vect{0}{1} + a_2
    \vect{x - x_{K_0}}{0} + a_3 \vect{0}{y - y_{K_0}} + a_4 \vect{y -
    y_{K_0}}{x - x_{K_0}}.
  \end{displaymath}
  The constraint of the minimization problem directly gives the value
  of $a_0$ and $a_1$ and we rewrite $\bm{p}(\bm{x})$ as 
  \begin{displaymath}
    \begin{aligned}
      \bm{p}(\bm{x}) &= \vect{g^1(\bm{x}_{K_0})}{g^2(\bm{x}_{K_0})} +
      a_2 \vect{x - x_{K_0}}{0} + a_3 \vect{0}{y - y_{K_0} } + a_4
      \vect{y - y_{K_0}}{x - x_{K_0}}. \\
    \end{aligned}
  \end{displaymath}
  Hence, the problem \eqref{app:eq:ls} is equivalent to the following
  problem, 
  \begin{equation}
    \mathop{\arg \min}_{(a_2, a_3, a_4) \in \mb{R}^3} \sum_{i = 1}^3
    \left\| a_2 \vect{x_{K_i} - x_{K_0}}{0} + a_3 \vect{0}{y_{K_i} -
    y_{K_0}} + a_4 \vect{y_{K_i} - y_{K_0}}{x_{K_i} - x_{K_0}}
    - \vect{g^1(\bm{x}_{K_i}) - g^1(\bm{x}_{K_0})}{g^2(\bm{x}_{K_i}) -
    g^2(\bm{x}_{K_0}) } \right\|^2,
    \label{app:eq:lsa}
  \end{equation}
  and the solution to \eqref{app:eq:lsa} takes the form 
  \begin{displaymath}
    \vech{a_2}{a_3}{a_4} = (A^TA)^{-1}A^T  \begin{bmatrix} 
      g^1(\bm{x}_{K_1}) - g^1(\bm{x}_{K_0}) \\ 
      g^2(\bm{x}_{K_1}) - g^2(\bm{x}_{K_0}) \\
      g^1(\bm{x}_{K_2}) - g^1(\bm{x}_{K_0}) \\
      g^2(\bm{x}_{K_2}) - g^2(\bm{x}_{K_0}) \\
      g^1(\bm{x}_{K_3}) - g^1(\bm{x}_{K_0}) \\
      g^2(\bm{x}_{K_3}) - g^2(\bm{x}_{K_0}) \\ \end{bmatrix}, \quad A
      = \begin{bmatrix}
       x_{K_1} - x_{K_0} & 0 &  y_{K_1} - y_{K_0} \\
       0 &  y_{K_1} - y_{K_0} & x_{K_1} - x_{K_0} \\
       x_{K_2} - x_{K_0} & 0 &  y_{K_2} - y_{K_0} \\
       0 &  y_{K_2} - y_{K_0} & x_{K_2} - x_{K_0} \\
       x_{K_3} - x_{K_0} & 0 &  y_{K_3} - y_{K_0} \\
       0 &  y_{K_3} - y_{K_0} & x_{K_3} - x_{K_0} \\
      \end{bmatrix}.
  \end{displaymath}
  We rearrange the solution as 
  \begin{displaymath}
    \begin{bmatrix}
      a_0 \\ a_1 \\ a_2 \\ a_3 \\ a_4 \\
    \end{bmatrix} = \begin{bmatrix}
      I_{2 \times 2} & 0  \\
      -M I_{6 \times 2} & M \\
    \end{bmatrix} \begin{bmatrix}
      g^1(\bm{x}_{K_0}) \\ g^2(\bm{x}_{K_0}) \\ 
      g^1(\bm{x}_{K_1}) \\ g^2(\bm{x}_{K_1}) \\ 
      g^1(\bm{x}_{K_2}) \\ g^2(\bm{x}_{K_2}) \\ 
      g^1(\bm{x}_{K_3}) \\ g^2(\bm{x}_{K_3}) \\ 
    \end{bmatrix}, 
  \end{displaymath}
  where $I_{2 \times 2}$ is $2 \times 2$ identity matrix and $I_{6
  \times 2} = [I_{2 \times 2}, I_{2 \times 2},I_{2 \times 2}]^T $ and
  $M = (A^TA)^{-1}A^T$. We note that the matrix $M$ totally depends
  the collocation points located in $S(K_0)$ and according to the
  expansion \eqref{eq:writeR}, the matrix 
  \begin{displaymath}
    \wt{M} = \begin{bmatrix}
      I_{2 \times 2} & 0  \\
      -M I_{6 \times 2} & M \\
    \end{bmatrix} 
  \end{displaymath}
  contains all information of the function $\bm{\lambda}_{K_i}^j(0
  \leq i \leq 3, j = 1, 2)$ on the element $K_0$. Then we store the
  matrix $\wt{M}$ for all elements to represent our approximation
  space $\bmr{U}_h^1$. The procedure of this computer implementation
  could be adapted to the case of higher-order accuracy and high
  dimensions without any difficulties.

\end{appendix}


\bibliographystyle{amsplain}
\bibliography{../ref}

\end{document}